\documentclass{amsart}
\usepackage{graphicx}
\usepackage{enumerate}
\usepackage{amscd}
\usepackage{amssymb}
\usepackage{pdfsync}
\newtheorem{theorem}{Theorem}[section]

\newtheorem{lemma}[theorem]{Lemma}

\newtheorem{prop}[theorem]{Proposition}
\newtheorem{claim}[theorem]{Claim}

\newtheorem{mtheorem}{Theorem}

\newtheorem{mcorollary}[mtheorem]{Corollary}

\theoremstyle{definition}
\newtheorem{definition}[theorem]{Definition}

\newtheorem{remark}[theorem]{Remark}

\newtheorem{assumption}[theorem]{Assumption}
\numberwithin{figure}{section}
\numberwithin{equation}{section}

\newcommand{\blackboard}[1]{\ensuremath{\mathbb{#1}}}

\newcommand{\complexes}{\blackboard{C}}
\newcommand{\hyperbolic}{\blackboard{H}}

\newcommand{\integers}{\blackboard{Z}} %
\newcommand{\reals}{\blackboard{R}}
\newcommand{\naturals}{\blackboard{N}}

\newcommand{\teich}{\mathcal{T}}
\newcommand{\ie}{i.e.\ }
\newcommand{\pr}{\mathrm{pr}}
\newcommand{\Fr}{\mathrm{Fr}}

\newcommand{\PSL}{\ensuremath{\mathrm{PSL}}}

\newcommand{\length}{\mathrm{length}}

\newcommand{\EL}{\mathcal{EL}}
\newcommand{\pretop}{\prec_\mathrm{top}}
\newcommand{\base}{\mathbf{base}}
\newcommand{\pre}{\mathrm{pred}}
\newcommand{\suc}{\mathrm{succ}}
\newcommand{\subord}{\overset{d}{\searrow}}
\newcommand{\supord}{\overset{d}{\swarrow}}
\newcommand{\stab}{\mathrm{stab}}
\def\v{\mathrm{v}}

\def\hh{\hyperbolic}
\def\zz{\integers}
\def\nn{\naturals}
\def\rr{\reals}

\def\cc{\complexes}

\def\pants{\Sigma_{(0,3)}}

\def\ca{\mathcal{A}}
\def\cc{\mathcal{C}}
\def\cb{\mathcal{B}}
\def\ck{\mathcal{K}}
\def\cf{\mathcal{F}}

\def\cu{\mathcal{U}}
\def\cv{\mathcal{V}}

\def\eset{\emptyset}
\def\part{\partial}

\def\wh{\widehat}
\def\wt{\widetilde}

\def\fd{\pi_1}

\def\Int{\mathrm{Int}}
\def\sto{\rightarrow}
\def\ve{\varepsilon}
\def\L{\Lambda}
\def\O{\Omega}
\def\G{\Gamma}

\def\Sg{\Sigma}
\def\sg{\sigma}

\def\dist{\mathrm{dist}}

\def\ve{\varepsilon}
\def\L{\Lambda}
\def\O{\Omega}
\def\G{\Gamma}

\def\Sg{\Sigma}

\def\length{\mathrm{length}}

\title{Geometry and topology of geometric limits I}
\author{Ken'ichi Ohshika}
\address{Department of Mathematics, Graduate School of Science, Osaka University, Toyonaka, Osaka 560-0043, Japan}
\email{ohshika@math.sci.osaka-u.ac.jp}
\author{Teruhiko Soma}
\address{Department of Mathematics and Information Sciences,
Tokyo Metropolitan University,
Minami-Ohsawa 1-1, Hachioji, Tokyo 192-0397, Japan}
\email{tsoma@tmu.ac.jp}

\date{\today}

\begin{document}
\maketitle

\begin{abstract}
In this paper, we classify completely hyperbolic 3-manifolds corresponding to geometric limits of Kleinian surface groups isomorphic to $\pi_1(S)$ for a finite-type hyperbolic surface $S$.
In the first of the three main theorems, we construct  bi-Lipschitz model manifolds for such hyperbolic 3-manifolds,  which have  a structure called brick decomposition and  are embedded in $S \times (0,1)$.
In the second theorem, we  show that conversely, any such model manifold  admitting a brick decomposition  with reasonable conditions is bi-Lipschitz homeomorphic to a hyperbolic manifold corresponding to some geometric limit of quasi-Fuchsian groups.
In the third theorem, it is shown that we can define end invariants for hyperbolic 3-manifolds appearing as geometric limits of Kleinian surface groups, and that the homeomorphism type and the end invariants determine the isometric type of a manifold, which is analogous to the ending lamination theorem for the case of finitely generated Kleinian groups.
These constitute an attempt to give an answer to the 8th question  among the 24 questions raised by Thurston in \cite{th2}. 
\end{abstract}

\setcounter{section}{-1}
%\sloppy
\section{Introduction}

There are two notions of convergence in the theory of Kleinian group: algebraic convergence and   geometric convergence.
Algebraic convergence is a convergence with respect to the topology induced from the natural topology on the space of representations of a group into $\mathrm{PSL}_2\complexes$. 
On the other hand, geometric convergence corresponds to a convergence of the  quotient hyperbolic $3$-manifolds with respect to the pointed Gromov-Hausdorff topology.
One of the main topics in the theory of Kleinian groups is studying the topological structure of deformation spaces.
Deformation spaces have topologies induced from the algebraic convergence.
Still, their singularities, for instance, those which are called self-bumping points, are caused by difference between algebraic and the geometric convergences, as was shown by work of Anderson-Canary \cite{ac0} and McMullen \cite{Mc2}.
This suggests that studying geometric limits is important for understanding the deformation spaces.

For an algebraically convergent sequence of Kleinian groups, its geometric limit, which always exists by passing to a subsequence, contains the algebraic limit, but may be larger than that in general.
The difference between algebraic limit and geometric one is first observed by J{\o}rgensen and Marden.
In \cite{jm}, they gave an example of algebraically convergent sequence of infinite cyclic groups in $\PSL_2 \complexes$ which converges to a rank-2 parabolic group. 
This is a typical phenomenon for geometric limits, and is a cause of geometric limits larger than algebraic ones in more complicated situations such as an example of Kerckhoff-Thurston \cite{kt}.

Kerckhoff and Thurston considered a sequence in the Bers slice of a quasi-Fuchsian space of a surface $S$, parametrised as $(m_0, \tau^n n_0) \in \teich(S) \times \teich( \bar S)$  for a Dehn twist $\tau$ along an essential simple closed curve $c$ on $S$, where $m_0$ and $n_0$ are arbitrary points in the Teichm\"{u}ller spaces $\teich(S)$ and $\teich(\bar S)$.
They took a sequence of quasi-Fuchsian groups representing $(m_0, \tau^n n_0)$ so that it converges algebraically, which can always be done by Bers's compactness theorem, and  showed that such a sequence converges geometrically to a group $G$ such that $\hyperbolic^3/G$ is homeomorphic to $S\times (0,1) \setminus c\times \{\frac12\}$.
Here a tubular neighbourhood of $c\times \{\frac12\}$ in $S\times (0,1)$ corresponds to a $\zz\times \zz$-cusp of $\hh^3/G$ where a phenomenon as in the case of J{\o}rgensen-Marden  occurs.
By iterating this kind of procedure, it is also possible to construct an example of a geometric limit $G'$ of quasi-Fuchsian groups such that $\hh^3/G'$ has infinitely many $\zz\times \zz$-cusps as was shown by Bonahon-Otal \cite{BO}, (see also Ohshika \cite{oh1}).
In particular, this shows that the geometric limit of  quasi-Fuchsian groups of isomorphic to  $\pi_1(S)$ with a finite type surface $S$ can be  infinitely generated.

Another important example of geometric limits of quasi-Fuchsian groups is given by Brock \cite{br}.
He considered a homeomorphism $\phi: S \rightarrow S$ which is pseudo-Anosov on some essential subsurface $H$ of $S$ and is the identity outside, and a sequence parametrised as $(m_0, \phi^n n_0)$ in the Bers slice as in the case of Kerckhoff-Thurston.
He showed that the geometric limit of such a sequence is a Kleinian group $G''$ such that $\hh^3/G''$ is homeomorphic to $S\times (0,1)\setminus H\times \{\frac12\}$.

A natural problem arising from these examples is to determine what kind of Kleinian groups can appear as geometric limits of quasi-Fuchsian group, or more in general as a geometric limit of a sequence in the deformation space of a Kleinian group.
The purpose of this series of papers is to answer this question.
In the present paper, we shall consider only geometric limits of Kleinian groups isomorphic to surface groups preserving the parabolicity, which are sometimes called Kleinian surface groups.
In Theorems A, which is the first of the three main theorems of this paper, we shall give (bi-Lipschitz) model manifolds for geometric limits of Kleinian surface groups and determine the conditions which the model manifolds should satisfy.
In Theorem C we shall show that these conditions are in fact sufficient, \ie that any model manifold satisfying the conditions in Theorem A is homeomorphic to some geometric limit of quasi-Fuchsian groups.
Combining these two theorems, we characterise completely Kleinian group which can appear as geometric limits of Kleinian surface groups.
%We plan to deal with more general Kleinian groups and give similar theorems in the second part.

Another problem is to classify hyperbolic manifolds corresponding to geometric limits up to isometries completely, which is the subject of Theorem D.
The classification problem of finitely generated Kleinian groups, which was called the ending lamination conjecture and is now the ending lamination theorem, was solved by Minsky, collaborating with Brock, Canary and Masur (\cite{mm1}, \cite{mm2}, \cite{mi2}, \cite{bcm}).
(An alternative approach can be found in Bowditch \cite{bow4}.)
Since geometric limits of isomorphic non-elementary finitely generated Kleinian groups can be infinitely generated in general as explained above, the ending lamination theorem is not sufficient for our situation.
Using our model manifolds constructed in Theorem A, we shall  prove that the homeomorphism type and (generalised) end invariants are enough to determine the isometry type of geometric limits.
Indeed this is what Theorem D claims for geometric limits of Kleinian surface groups.

In \cite{th2}, Thurston listed 24 questions in the field of hyperbolic 3-manifolds and Kleinian groups which were open at that time.
The question 8 reads \lq\lq Analyse limits of quasi-fuchsian groups with accidental parabolics".
Otal  in \cite{otalj}, which is his very informative and well-written review on this Thurston's paper, interpreted this problem as the one for analysing geometric limits of algebraically convergent quasi-Fuchsian groups.
The results of this paper give a complete answer to Thurston's question 8 interpreted this way. 

There are applications of the results of this paper, which appeared in \cite{OhD} and \cite{OhR}.
In \cite{OhD}, we used Theorem A to analyse which points on the boundary of the quasi-Fuchsian space can be bumping points.
In \cite{OhR}, we studied a quotient space of the Bers boundary of Teichm\"{u}ller space, called the reduced  Bers boundary, on which the mapping class group action on the Teichm\"{u}ller space extends continuously.
We refer the reader also to \cite{OhD} for the overall picture of geometric limits, including the results of these two papers.

\section{Main results}
\label{main results}
In this section, we shall state main results of this paper.
We shall also give definitions of terms which are necessary for stating the main results, and a short outline of their proofs.

For a hyperbolic 3-manifold $N$, we denote by $N_0$  the complement 
 of the open $\ve$-cusp neighbourhoods in $N$ for $\ve >0$ less than the three-dimensional Margulis constant.
Its homeomorphism type does not depend on the choice of a constant $\ve$.
By the resolution of Marden's tameness conjecture by Agol \cite{ag} and Calegari and Gabai \cite{cg}, 
any relative end $e$ of a hyperbolic 3-manifold with finitely generated fundamental group 
is topologically tame, \ie it has a neighbourhood homeomorphic to $F \times (0,\infty)$, where $F=F\times \{0\}$ corresponds to the frontier component of a relative compact core of $N_0$ facing $e$.
It follows then from the results of Bonahon \cite{bon} and Canary \cite{ca1}, that $e$ is either geometrically finite 
or simply degenerate: the latter  means that there is a sequence of closed geodesics tending to the end which are projected in $F \times \reals$ to simple closed curves on $F$ whose projective classes  converge in the projectivised Masur domain.
However, in general, a hyperbolic 3-manifold $N$ with infinitely generated fundamental group may have infinitely many relative ends which are neither geometrically finite nor simply degenerate.
We call such ends \emph{wild}.
To our knowledge, suitable invariants of wild ends which play the role of end invariants for tame ends have  not been known up to now.
Still, we shall show for hyperbolic 3-manifolds corresponding to geometric limits of surface Kleinian groups, wild  ends are controlled in some way and they are determined only by the homeomorphism types, as we shall see in Theorem C.

Now, we are going to state our main results.
The first theorem, Theorem A, says that every geometric limit of Kleinian surface groups isomorphic to $\pi_1(S)$  has  a bi-Lipschitz model which admits a decomposition into standard blocks, and can be embedded topologically into $S \times (0,1)$.
This gives also  necessary conditions which hyperbolic 3-manifolds corresponding to geometric limits of Kleinian surface groups must satisfy.
Before stating the theorem, we shall explain terms which will be used in the statement.
A detailed account of these notions can be found in Section \ref{SS_EIBM}.
%Recall that we say that $N_0$ is acylindrical when there is no essential embedded annulus $A$ whose boundary lies on $\partial N_0$.
A \emph{brick} $B$ is a 3-manifold homeomorphic to $F\times J$ for a compact connected essential subsurface $F$ and an 
interval $J$ which is either closed or  half-open.
A \emph{brick manifold} is a union of countably many bricks $F_n\times J_n$ which are glued to each other along essential subsurfaces on their fronts 
$F_n\times \part J_n$.

In a brick manifold, we consider to attach to the end of each half-open brick either a conformal structure at infinity or an ending lamination (\ie a filling geodesic lamination).
We call the brick geometrically finite in the former case and simply degenerate in the latter.
Each half-open end of a brick constitutes an end of $M$, and the end is called geometrically finite or simply degenerate accordingly.
The ending lamination or the conformal structure attached there is called the end invariant.
The union of ideal boundaries on which conformal structures are given is called the boundary at infinity of $M$, and is denoted by $\partial_\infty M$.
A brick manifold endowed with these end invariants is called {\em a labelled brick manifold}.

We say that a labelled brick manifold admits a block decomposition when the manifold is decomposed into blocks in the sense of Minsky  and solid tori in such a way that each block has horizontal and vertical directions coinciding with those of bricks.
We also require the block decomposition for a half-open brick to accord with its end invariant.
The blocks have standard metrics and we can choose gluing maps to be isometries.
By identifying a solid torus with a Margulis tube which is determined by information coming from the block decomposition, we can put a metric on the labelled brick manifold.
We call such a metric  {\em a model metric}.
(See \S\S \ref{SS_block} and \ref{SS_block_metric} for details.)

%See Subsection \ref{SS_EIBM} for the precise definition of brick manifolds and Subsection \ref{SS_Conditions} for 
%\emph{accessible} wild ends.

\begin{mtheorem}\label{thm_a}
Let $S$ be an orientable connected hyperbolic surface of finite type.
Let $\{G_n\}$ be a sequence of Kleinian surface groups isomorphic to $\pi_1(S)$ preserving the parabolicity, converging geometrically 
to a non-elementary Kleinian group $G$.
Then there are a labelled brick manifold $M$  with the following properties, which admits  a block decomposition,   and  a $K$-bi-Lipschitz homeomorphism from $M$ with the model metric to the non-cuspidal part $N_0$ of the hyperbolic $3$-manifold $N=\hyperbolic^3/G$, with the constant $K$ depending only on $\chi(S)$.
\begin{enumerate}[\rm (i)]
%\item
%$M$ is acylindrical and contains no accessible wild ends.
\item
Each component of $\part M$ is either a torus or an open annulus.
\item
There is no properly embedded incompressible annulus in $M$ whose boundary components lie on distinct boundary components of $M$.
\item
If there is an embedded, incompressible half-open annulus $S^1 \times [0,\infty)$ in $M$ such that $S^1 \times \{t\}$ tends to a wild end $e$ of $M$ as $t \rightarrow \infty$, then its core curve is homotopic into an open annulus component of $\partial M$ tending to $e$.
\item The manifold $M$ is (not necessarily properly) embedded in $S \times (0,1)$ in such a way that each brick has a form $F \times J$ with an interval $J$ and an essential subsurface $F$ of $S$ with respect to the product structure of $S \times (0,1)$ and the ends of  geometrically finite bricks lie $S \times \{0,1\}$.
(We shall say that geometrically finite ends are peripheral to refer to the last condition.)
%\item
%$M$ can be realised as a subbrick manifold of $S\times (0,1)$ such that  
%each geometrically finite brick is peripheral.
\end{enumerate}
\end{mtheorem}

We call the labelled brick manifold $M$ in this theorem a \emph{model manifold} for the geometric limit $N$. 
It should be noted that a result similar to this was announced in the introduction of the first version of Brock-Canary-Minsky \cite{bcm}.
By (iv), we see that the geometric limit manifold $N_0$ has at most $-2\chi(S)$ geometrically finite ends.

%??????outline?????????????D
%Since each $G_n$ is a Kleinian surface group, $\hh^3/G_n$ is homeomorphic to $S\times (0,1)$ by Bonahon's theorem \cite{bon}.
%Since $\hyperbolic^3/G$ is a geometric limit of $\hyperbolic^3/G_n$, if we consider sub-brick manifold of $M$ consisting of finitely many bricks, it can be embedded in $S \times (0,1)$.
%%Thus, for any finite subbrick manifold $M_n$ of $M$, there exists an embedding $\eta_n:M_n\rightarrow S\times (0,1)$.
%This does not imply immediately the existence of an embedding of $M$ into $S\times (0,1)$ however.
%We shall need to rearrange the embeddings of sub-bricks by twisting them  as will be shown in Lemma \ref{l_21}. 

The following corollary is easily deduced from Theorem \ref{thm_a}.
It guarantees that we can make use of  a generalised version of Sullivan's rigidity theorem proved by McMullen \cite{Mc}, which is crucial in the proof of Theorem \ref{thm_d}.

\begin{mcorollary}\label{cor_b}
Let $G$ be a non-elementary geometric limit  of quasi-Fuchsian groups isomorphic to $\pi_1(S)$ preserving the parabolicity for $S$ as in Theorem \ref{thm_a}.
Then $N=\hh^3/G$ has at most countably 
many relative ends.
Furthermore, there is an upper bound depending only on $\chi(S)$ for the injectivity radii at points in the convex core of $N$.
%There is a constant $L$ depending only on $\xi(S)$ such that for every point $x$ in the convex core of $N_0$, there is a pleated surface $f_x: R \rightarrow N$ from an essential subsurface $R$ of $S$, whose image is within the distance $L$ from $x$.
\end{mcorollary}

The next theorem claims the existence of a geometric limit which is bi-Lipschitz equivalent to a brick manifold with the properties in Theorem \ref{thm_a} provided that there are no two simply degenerate ends with homotopic ending laminations.

\begin{mtheorem}\label{thm_c}
Suppose that $M$ is a labelled brick manifold satisfying the conditions {\rm (i)}-{\rm (iv)} in Theorem \ref{thm_a}  
 such that 
the ending laminations of  two simply degenerate ends of $M$ are not homotopic to each other in $M$.
(This condition is necessary only when $M$ is homeomorphic to $F \times (0,1)$, for a compact essential subsurface $F$ of $S$ since ending laminations are filling.)
Then $M$ has a block decomposition, and if we put on $M$ the model metric associated with the decomposition, then
% $M$ has a block decomposition, and 
there exists a non-elementary geometric limit $G$ of quasi-Fuchsian groups isomorphic to $\pi_1(S)$ 
such that $N=\hh^3/G$ admits a $K$-bi-Lipschitz homeomorphism $f:M\rightarrow N_0$ which can be extended continuously
to a conformal map $\part_\infty M\rightarrow \part_\infty N$ between the boundaries at infinity for a  
 constant $K\geq 1$ depending only on $\chi(S)$.
\end{mtheorem}

We shall often use the term \lq\lq uniform bi-Lipschitz map" to mean that its bi-Lipschitz constant depends only on 
%Here we say that a constant  is \emph{uniform} if it depends only on 
$\chi(S)$, and hence is independent of the end invariants.
%A \emph{block decomposition} for any brick manifold $M$ satisfying the assumptions in Theorem \ref{thm_c} is 
%defined in Subsection \ref{SS_block}.
%Moreover, a metric on $M$ associated to the decomposition is introduced in Subsection \ref{SS_block_metric}.

By applying Theorem \ref{thm_c}, we can construct various examples of geometric limits $G$ of quasi-Fuchsian 
groups isomorphic to $\pi_1(S)$  such that $N_0$ has infinitely many simply degenerate ends and infinitely many wild ends simultaneously.
%We shall describe in the second part of the series a way to obtain information on geometric limits from that on sequence of quasi-Fuchsian groups using the theory of hierarchy of tight geodesics by Minsky.

The last theorem is a classification theorem which is analogous to the ending lamination theorem for finitely generated case.

\begin{mtheorem}\label{thm_d}
Suppose that $G_1$ and $G_2$ are non-elementary geometric limits of Kleinian surface groups isomorphic to $\pi_1(S)$ preserving the parabolicity.
If $f:\hh^3/G_1\rightarrow \hh^3/G_2$ is a homeomorphism preserving their end invariants, then $f$ is properly homotopic to an isometry.
\end{mtheorem}

%Our argument to prove this is based on the proof of the Bi-Lipschitz Model Theorem by Minsky \cite{mi2} and 
%Brock, Canary and Minsky \cite{bcm}.
%In Section \ref{S_BLM}, this theorem will be presented in a form suitable to our arguments.
%Proofs of our theorems will be given in Section \ref{S_Proofs}.

\begin{remark}
In the beginning of the present work, we tried to use more classical topological approach involving only hyperbolic geometry to study topological properties of geometric limits of quasi-Fuchsian groups.
Subsequently we found that, by invoking the bi-Lipschitz model theorem by Brock-Canary-Minsky,  it is possible to  
simplify proofs of some results and moreover to obtain a deeper result on geometric properties of geometric limits.
Therefore, we have changed our original plan and adopted   the method relying upon work of \cite{mi2} and \cite{bcm}.
 On the other hand, we have noticed that our original approach on geometric limits 
gives rise to  a rather short proof of the bi-Lipschitz model theorem.
See Soma \cite{so}.
\end{remark}

%Outline of proofs.
Now, we shall outline the proofs of the main theorems.
To prove Theorem A, we shall first apply Minsky's bi-Lipschitz model theorem to each $\hyperbolic^3/G_n$ and get a model manifold $M_n$ which can be decomposed into blocks with a bi-Lipschitz homeomorphism $g_n$ from $M_n$ to $(\hyperbolic^3/G_n)_0$.
We define $M$ and a bi-Lipschitz homeomorphism from $M$ to $N_0$ to be the geometric limits of $M_n$ and $g_n$.
We shall verify that these satisfy the required conditions (i)-(iv) one by one, among which the most difficult is (iv).
Since $M$ is the geometric limit of $\{M_n\}$, each union of finite bricks can be proved to be embedded in $S \times (0,1)$ preserving the product structures, but this does not imply immediately that the entire $M$ can also be embedded.
We shall need to  rearrange the embeddings of sub-bricks by twisting them in such a way the twisting stabilises on each brick as will be shown in Lemma \ref{l_21}.

Next we turn to Theorem C.
We shall first consider an ascending exhausting sequence of sub-brick-manifolds $W_n$ consisting of finite bricks within the given labelled brick manifold $M$.
These $W_n$ may have very complicated homeomorphism types, but we shall construct from the $W_n$   brick manifolds $Z_n$ corresponding  to   geometrically finite Kleinian surface groups by applying Thurston's uniformisation theorem for compact irreducible atoroidal 3-manifolds with boundary, whose geometric limit is also $M$.
We shall approximate these Kleinian groups corresponding to $Z_n$ by quasi-Fuchsian groups, which are groups as we wanted.

Finally, we shall outline the proof of  Theorem D.
We are given two geometric limits $G_1$ and $G_2$ such that $N_1=\hyperbolic^3/G_1$ and $N_2=\hyperbolic^3/G_2$ share the same topological type and end invariants.
By Theorem A, we can construct a labelled model manifold $M$ of $(N_1)_0$.
By our assumption, there is a homeomorphism from $M$ to $(N_2)_0$ preserving the end invariants.
In Theorem \ref{blm}, which is a generalisation of the bi-Lipschitz model theorem by Brock-Canary-Minsky \cite{bcm}, we shall prove that such a homeomorphism can be homotoped to a uniform bi-Lipschitz homeomorphism.
This shows that $G_1$ and $G_2$ are quasi-conformally conjugate by a quasi-conformal homeomorphism which is conformal on the domain of discontinuity.
The second statement of Corollary B makes it possible to apply McMullen's generalisation of Sullivan's rigidity theorem and we shall be able to show that $G_1$ and $G_2$ are conformally conjugate.

\section{Preliminaries}\label{S_1}

We refer the reader to Thurston \cite{th1}, Benedetti and Petronio \cite{bp}, 
Matsuzaki and Taniguchi \cite{mt}, and Marden \cite{ma2} for the general theory on hyperbolic manifolds and Kleinian groups, and to 
Hempel \cite{he} for the 3-manifold topology.

Throughout this paper, all manifolds are assumed to be oriented, and all homeomorphisms between 
manifolds are assumed to be orientation-preserving.
When we talk about a surface $S$, we always assume that it is a connected surface of finite type possibly with punctures and $\chi(S) <0$.
Sometimes, we fix a hyperbolic structure of finite area on it for convenience.
We denote by $\Sg_{0,3}$, $\Sg_{0,4}$, $\Sg_{1,1}$ compact surfaces homeomorphic respectively to a three-holed sphere, a four-holed sphere and a one-holed torus.

\subsection{The curve graph and tight geodesics}\label{curve}
%In Masur-Minsky \cite{mm1} \cite{mm2} and Minsky \cite{mi2}, they considered open surfaces to define curve graphs.
%This is more convenient to define subsurface projections but makes it necessary to consider compact cores to construct blocks from resolutions of hierarchies.
%In the present paper, we shall consider compact surfaces instead of open ones, this is just to simplify the system of symbols in the construction of block decomposition.
%There is no essential difference between considering open surfaces and compact ones: compact surfaces are nothing but compact cores in open surfaces.
%
In this subsection we shall review the basic terminology and results on curve graphs and tight geodesics.
Most of these are due  to Masur-Minsky and can be found in  \cite{mm1, mm2}.

A  subsurface $\Sigma$ of $S$ is called \emph{essential} if no component of the frontier of $\Sigma$ is null-homotopic in $S$.
We regard $S$ itself also as an essential subsurface of $S$.
When $\Sigma$ is an open annulus we further assume that the frontier of $\Sigma$ is not homotopic to a puncture of $S$.
We consider both closed essential subsurfaces and open ones.
When we consider two essential subsurfaces, we assume that they do not have inessential intersection.
If two essential subsurfaces are isotopic, they are assumed to coincide.

Let $\Sigma$ be a connected surface of finite type, possibly  with punctures.
In this paper, when we talk about curve graphs, we only consider the situation where $\Sigma$ is an open essential subsurface of some fixed surface $S$, including the case when $\Sigma=S$.
The complexity of $\Sigma$ is defined by $\xi(\Sigma)=3g+p$, where $g$ is the genus 
of $\Sigma$ and $p$ is the number of punctures of $\Sigma$.
(This is more convenient than the Euler characteristic $\chi(S)$ for our purpose.)
A surface $\Sigma$ with $\xi(\Sigma)=3$ (resp.\ $\xi(\Sigma)=4$) is homeomorphic to the interior of $\Sg_{0,3}$ 
(resp.\  the interior of either $\Sg_{0,4}$ or $\Sg_{1,1}$).

When $\xi(\Sigma)>4$, we define the \emph{curve graph} $\cc(\Sigma)$ of $\Sigma$ to be a simplicial graph whose vertices are 
homotopy classes of non-contractible  simple closed curves on $\Sigma$ which are not homotopic to punctures such that two vertices are connected by an edge if and only if they have disjoint representatives.
We call a vertex of $\cc(\Sigma)$ or its representative a \emph{curve} on $\Sigma$.
For our convenience, we fix a complete hyperbolic structure on $\Sigma$ of finite area and take a uniquely determined closed geodesic as a representative for any curve in $\cc(\Sigma)$.
The notion of curve graph was first introduced by Harvey \cite{har} and extended and modified in \cite{mm1,mm2,mi1}.
In the case when $\xi(\Sigma)=4$, the curve graph $\cc(\Sigma)$ is defined so that  
the vertices are 
curves on $\Sigma$ and two curves $v,w$ are joined by an edge if and only if they have the 
minimum geometric intersection, 
\ie $i(v,w)=1$ when $\Sigma$ is $\Int \Sg_{1,1}$ and $i(v,w)=2$ when $\Sigma$ is $\Int \Sg_{0,4}$.
When $\Sigma$ is an open annulus embedded in $S$, we consider the covering $\tilde{\Sigma}$ of $S$ (with a fixed hyperbolic structure) associated to $\pi_1(\Sigma)$ and compactify $\tilde \Sigma$ to $\bar \Sigma$ by attaching the ideal boundary.
The vertices of curve complex $\cc(\Sigma)$ are homotopy classes of essential arcs on $\bar \Sigma$ fixing the endpoints.
Two vertices are connected by an edge if and only if they can be homotoped fixing the endpoints to arcs whose interiors are disjoint.

We put the path metric $d=d_{\cc(\Sigma)}$ on $\cc(\Sigma)$ by setting the length of  each edge to be $1$.
In the case when $\xi(\Sigma) >4$, a finite subset $v$ of vertices in $\cc(\Sigma)$ is said to constitute a \emph{simplex} if any two curves of $v$ are represented by  disjoint and non-parallel simple closed curves on $\Sigma$.
This naming comes from the fact that they actually span a simplex in the curve complex of $\Sigma$.
We only use this term and do not need to consider the curve complex itself.
The graph $\cc(\Sigma)$ is not locally finite but is proved to be Gromov hyperbolic as a metric space by Masur and Minsky 
\cite{mm1}.
 (See also Bowditch \cite{bow1} for an alternative approach.)
%A sequence $\{v_i\}$ of vertices in $\cc(F)$ is called a \emph{geodesic} if $d(v_i,v_j)=|i-j|$ for any 
%$v_i,v_j$.

Let $\mathcal{ML}(\Sigma)$ be the space of compact measured laminations on $\Sigma$ and $\mathcal{UML}(\Sigma)$ the quotient space of $\mathcal{ML}(\Sigma)$ obtained by forgetting the measures, and 
let $\mathcal{EL}(\Sigma)$ be the subspace of $\mathcal{UML}(\Sigma)$ consisting of filling laminations, which we call the ending lamination space of $\Sigma$.
Here a lamination $\mu$ in $\mathcal{UML}(\Sigma)$ is said to be \emph{filling} if, for any $\mu'\in \mathcal{UML}(\Sigma)$, either $\mu'=\mu$ or $\mu'$ intersects $\mu$ non-trivially and transversely.
(The term \lq \lq arational lamination" is used in some literature in the same meaning.)
Refer to \cite{flp} and \cite{bc} for the definition and basic facts about measured lamination space.

Gromov showed that there is a natural boundary at infinity for a Gromov hyperbolic space.
According to Klarreich \cite{kla} (see also Hamenst\"{a}dt \cite{ham}), there exists a homeomorphism $k$ from the 
Gromov boundary $\part \cc(\Sigma)$ of $\cc(\Sigma)$ to $\mathcal{EL}(\Sigma)$ such that a sequence $\{v_i\}$ of vertices 
of $\cc(\Sigma)$ converges to $\beta \in \part \cc(\Sigma)$ if and only if it converges to $k(\beta)$ in $\mathcal{UML}(\Sigma)$.

\begin{definition}
A sequence $g=\{v_i\}_{i\in J}$ of simplices in $\cc(\Sigma)$ is called a \emph{tight sequence} if it satisfies 
one of the following conditions depending on whether $\xi(\Sigma)$ is greater than $4$ or not, where $J$ is a finite or an infinite interval of $\zz$.
\begin{enumerate}[\rm (i)]
\item
When $\xi(\Sigma)>4$, for any vertices $w_i$ of $v_i$ and $w_j$ of $v_j$ with $i\neq j$, it holds that $d(w_i,w_j)=|i-j|$.
Moreover, if $\{i-1,i,i+1\}\subseteq J$, then $v_i$ is represented by the union of all components of $\part \Sigma_{i-1}^{i+1}$ 
that are non-peripheral in $\Sigma$, where $\Sigma_{i-1}^{i+1}$ is a  subsurface smallest up to isotopy in $\Sigma$ with essential boundary 
 containing the geodesic representatives of all the vertices of $v_{i-1}$ and $v_{i+1}$.
\item
When $\xi(\Sigma)=4$, each $v_i$ is a vertex in $\cc(\Sigma)$ and  $d(v_i,v_j)=|i-j|$. 
\end{enumerate}

The sequence $g$ is said to connect $v_{\inf J}$ with $v_{\sup J}$, where we define $v_{\inf J}$ to be $\lim_{i \rightarrow -\infty} v_i$ when $\inf J=-\infty$ and $v_{\sup J}$ to be $\lim_{i \rightarrow \infty} v_i$ when $\sup J=\infty$.
The surface $\Sigma$ is called the support of $g$ and is denoted by $D(g)$.
The length of $g$ is defined to be $\# J-1$, where $\#J-1$ is defined to be $\infty$ when $\# J=\infty$.
\end{definition}

\smallskip
We regard a single vertex as a tight sequence of length $0$.
It follows from the definition that  for any tight sequence $\{v_i\}$, if a vertex $w$ of $\cc(\Sigma)$ meets $v_i$ transversely, then 
$w$ meets at least one of $v_{i-1}$ and $v_{i+1}$ transversely.

\smallskip

For an open essential surface $F$ of  $\Sigma$ and a tight geodesic $g$ in $\cc(\Sigma)$, we denote by $\phi_g(F)$ the union of simplices on $g$ which are disjoint from $F$.
Here being disjoint means that they can be made disjoint by an isotopy.
For a curve $c$ on $F$, we use the symbol $\phi_g(c)$ to denote $\phi_g(A(c))$, where $A(c)$ is an annular neighbourhood of $c$.

The following theorem is Lemma 5.14 in \cite{mi2} (see also Theorem 1.2 in \cite{bow2}), which was crucial in the proof of the ending lamination conjecture.

\begin{theorem}\label{tight}
Let $u,w$ be distinct vertices or laminations in $\cc(\Sigma)\cup \mathcal{EL}(\Sigma)$.
Then there exists a tight sequence connecting 
$u$ with $w$.
\end{theorem}

A marking on $\Sigma$ is a simplex in $\cc(\Sigma)$ some of its vertices (possibly none) have transversals.
Here a transversal of a curve $c$ is defined to be a vertex of the curve complex of an annular neighbourhood of $c$.
For a marking $\boldsymbol I$, we denote by $B(\boldsymbol I)$ its  vertices with the transversals forgotten, and call it the base curves.
Suppose that each of $\boldsymbol{I},\boldsymbol{T}$ is either a marking on $\Sigma$   or a lamination in 
$\mathcal{EL}(\Sigma)$.
Then a tight sequence $g=\{v_i\}_{i\in I}$ on $\Sigma$ is said to be a \emph{tight geodesic} with the 
\emph{initial marking} $I(g)=\boldsymbol{I}$ and the \emph{terminal marking} $T(g)=
\boldsymbol{T}$ if it satisfies the following conditions.
\begin{itemize}
\item
If $i_0=\inf J>-\infty$, then $v_{i_0}$ is a vertex of $\cc(\Sigma)$ contained in $B(\boldsymbol{I})$. 
%(possibly $v_{i_0}=\boldsymbol{I}$), 
Otherwise, $\boldsymbol{I}=\lim_{i\sto-\infty}v_i\in \mathcal{EL}(\Sigma)$.
%Minsky?????`??????????simplex??vertex???????????????????D
\item
If $j_0=\sup J<\infty$, then $v_{j_0}$ is a vertex of $\cc(\Sigma)$ contained in $B(\boldsymbol{T})$. 
%(possibly $v_{j_0}=\boldsymbol{T}$), 
Otherwise $\boldsymbol{T} =\lim_{j\sto \infty}v_j\in \mathcal{EL}(\Sigma)$.
\end{itemize}

For a simplex $v$ of a geodesic $g$ supported on $\Sigma$, a component of $\Sigma \setminus v$ and an annulus with core curve in $v$  is called a {\em component domain} of $v$, and also a component domain of $g$.
For a simplex $v_j$ of $g=\{v_j\}$, we define its predecessor $\mathrm{pred} (v_j)$ to be $v_{j-1}$ if $j\neq 1$, and $I(g)$ if $j=1$.
Similarly we define the successor $\suc(v_j)$.
For a component domain $Y$ of $v_j$, we denote $\pre(v_j)|Y$ by $I(Y,g)$ and $\suc(v_j)|Y$ by $T(Y,g)$.
Here in the case when $Y$ is an annulus $\pre(v_j)|Y$ denotes a vertex in $\cc(Y)$ which $\pre(v_j)$ determines when $j\neq 1$ and the transversal of the vertex $v_j$ determines if $j=1$.
The same definition applies for $\suc(v_j)|Y$.
If $T(Y,g)\neq \emptyset$, then we write $Y \subord g$ and says that $Y$ is forward subordinate to $g$ at $v_j$.
Similarly we write $g \supord Y$ and says that $Y$ is backward subordinate to $g$ at $v_j$ if $I(Y,g)\neq \emptyset$.
If a tight geodesic $k$ is supported on $Y$, the domain $Y$ is forward subordinate to $g$ at $v_j$, and $T(k)=T(Y,g)$, we say that $k$ is forward subordinate to $g$ at $v_j$ and denote by $k \subord g$.
Similarly, we define $g \supord k$.

\begin{definition}
A {\em hierarchy} $H$ of geodesics on $S$ is a family of tight geodesics on essential open subsurfaces of $S$ with the following properties.
\begin{enumerate}
\item There is a unique geodesic $g_H$ in $H$ with $D(g_H)=S$, which we call the {\em main geodesic}.
\item Let $Y$ be a component domain of both a simplex $v$ of $g \in H$  and $w$ of $g' \in H$ such that $g \supord Y \subord g'$.
(The geodesic $g$ and $g'$ may be the same.)
Then there exists a unique geodesic $h$ in $H$ such that $D(h)=Y$ and $g \supord h \subord g'$.
\item For any geodesic $g$ in $H$ other than $g_H$, there exist geodesics $h, k \in H$ such that $h \supord g \subord k$.
\end{enumerate}
For a hierarchy $H$, we define $|H|$ to be the sum of the lengths of the geodesics constituting $H$.
\end{definition}
A hierarchy $H$ is said to be {\em complete} if for each component domain $X$ of $\xi(X) \neq 3$, there is a geodesic in $H$ supported on $X$.
A geodesic $g$ in a hierarchy in $H$ whose domain $D(g)$ satisfies $\xi(D(g))=4$ is called a {\em $4$-geodesic}.
A sub-hierarchy of a complete hierarchy $H$ consisting of all the geodesics in $H$ supported on domains with $\xi \geq 4$ is called the $4$-sub-hierarchy.

%A hierarchy $H$ is said to be {\em complete} if for any $g\in H$ and its component domain $D$, there is a geodesic $h \in H$ supported on $D$.
%We say that $H$ is $4$-complete if this holds except for the case when $D$ is an annulus.

\begin{definition}
\label{slice}
Let $H$ be a hierarchy of geodesics on $S$.
A {\em slice} of $H$ is a set of pairs $\sigma=\{(g,v)\}$ of a geodesic $g \in H$ and a simplex $v$ on $g$ which has the following properties.
\begin{enumerate}
\item If $(g,v_1)$ and $(g,v_2)$ are contained in $\sigma$, then $v_1=v_2$.
\item There is a pair $(g_\sigma, v_\sigma)$ called the bottom pair, and except for the bottom pair every pair $(h,w) \in \sigma$ is supported in a component domain of some other $(k, u) \in \sigma$.
\end{enumerate}
We also call $g_\sigma$ the bottom geodesic and $v_\sigma$ the bottom simplex of $\sigma$.

A slice $\sigma$ is said to be {\em saturated} if for any $(g,v) \in \sigma$ and its component domain $D$ for which there is a geodesic $h$ in $H$ supported on $D$, there is some simplex $w$ of $h$ such that $(h, w) \in \sigma$.
We say that $\sigma$ is {\em non-annular saturated} if the above holds provided that $D$ is not an annulus.
For a slice $\sigma$, $\mathbf{base}(\sigma)$ denotes the union of all vertices contained in simplices which appear in $\sigma$, which forms a simplex of $\cc(D(g_\sigma))$.
\end{definition} 
\subsection{Hyperbolic 3-manifolds and geometric limits}\label{hyp}
%
%Let $A$ be a closed subset of a metric space $(X,d)$.
%For any $R>0$, the $R$-neighbourhood $\{x\in X;\, d(x,A)\leq R\}$ is denoted by $\cn_R(A,X)$.
%We set also $\cn_R(x,X)=\cn_R(A,X)$ when $A$ is a one-point subset $\{x\}$ of $X$.
%
%
%
%
%Geodesic subsurfaces (resp.\ simple geodesic loops) of $S$ are simply called \emph{g-subsurfaces} (resp.\ \emph{g-loops}) of $S$.
%The complement in $S$ of a disjoint union of g-subsurfaces and g-loops in $S$ is called an \emph{open g-subsurface} of $S$.
%
%
%
%Hyperbolic 3-space $\hh^3$ is the Riemannian 3-manifold with the underlying space $\mathbf{C}\times \mathbf{R}_+=\{ (z,t)\in \mathbf{C}\times \mathbf{R};t>0\}$ and the metric $ds^2=(|dz|^2+dt^2)/t^2$.
%One can regard that the Riemann sphere $\widehat{\mathbf{C}}=\mathbf{C}\cup \{\infty\}$ is the boundary of 
%$\hh^3$ at infinity.
%The group $\mathrm{Isom}^+(\mathbf{H}^3)$ of orientation-preserving isometric transformations on $\mathbf{H}^3$ is naturally identified with the group $\mathrm{PSL}_2(\mathbf{C})$ of M\"{o}bius transformations on $\widehat{\mathbf{C}}$.
%
%
%
%
%
A \emph{Kleinian group} $\G$ is a discrete subgroup of $\PSL_2\complexes$.
When $\G$ contains an abelian subgroup of finite index, it is called \emph{elementary}.
In this paper, we always assume that Kleinian groups are torsion-free, or equivalently that they contain no elliptic elements.
Under this assumption, a Kleinian group is elementary if and only if it is isomorphic to a free abelian group of rank at most two.
For a Kleinian group $\Gamma$, the quotient space $N=\hyperbolic^3/\Gamma$ is called the \emph{hyperbolic $3$-manifold corresponding to} $\Gamma$.

The \emph{limit set} $\L_\G$ of $\G$ is the set of accumulation points of the orbit space $\G x_0$ in the 
closed 3-ball $\hh^3\cup \hat \complexes$ for a fixed point $x_0\in \hh^3$.
It should be noted that $\L_\G$ is contained in $\hat \complexes$ since $\G$ acts on $\hh^3$ properly discontinuously.
The complement of $\L_\G$ in $\hat \complexes$ is called the \emph{region of discontinuity} of $\G$, and is denoted by $\O_\G$.
We can regard $N$ as the interior of the manifold $(\hh^3\cup \O_\G)/\G$, which is called 
the \emph{Kleinian manifold} corresponding to $\G$.
The boundary at infinity $\O_\G/\G$ is also denoted by $\part_\infty N$.
The Nielsen convex hull $H_\G$ is the smallest closed convex set in $\hh^3$ containing all geodesics with endpoints on $\L_\G$, which is also $\G$-invariant.
Its quotient $C_\G=H_\G/\G$ is called the \emph{convex core} of $N$.
The Kleinian group $\G$ is said to be \emph{geometrically finite} if the volume of the $\delta$-neighbourhood of $C_\G$ in $N$ 
is finite for  some $\delta>0$.

For a positive number $\varepsilon$, the $\varepsilon$-\textit{thin part} 
$N_{(0,\varepsilon]}$ of $N$ is the set consisting of all points 
$x\in N$ such that there exists a non-contractible loop $l$ of length $\leq \varepsilon$ based at $x$.
The complement of its interior $N_{[\varepsilon,\infty)}=N\setminus \mathrm{Int} N_{(0,\varepsilon]}$ is called the $\varepsilon$-\emph{thick part} of $N$.
A \emph{Margulis tube} is an embedded, equidistant, tubular neighbourhood of a simple closed geodesic in $N$.
A $\integers$ or a $\integers\times \integers$-\emph{cusp neighbourhood } $P$ is a subset of $N$ such that each component of $p^{-1}(P)$ is a horoball whose stabiliser in $\Gamma$ is isomorphic to either $\integers$ or $\integers \times \integers$, where $p:\hh^3\rightarrow N$ is the universal covering.
By  Margulis' lemma \cite[Corollary 5.10.2]{th1}, there exists a constant $\epsilon_0>0$ independent 
of $\G$, called the \emph{Margulis constant}, such that, for any 
$0<\ve< \epsilon_0$, each component of 
$N_{(0,\ve]}$ is either a Margulis tube or a $\integers$ or a $\integers\times \integers$-cusp neighbourhood.
Let $N_0=N_0^\ve$ be the union of $N_{[\varepsilon,\infty)}$ and all the Margulis tube components of $N_{(0,\varepsilon]}$, which we call the {\it non-cuspidal part} of $N$.
For any $\ve_1 <\ve_2< \epsilon_0$, there exists a  $K$-bi-Lipschitz deformation retraction 
$N_0^{\ve_2}\rightarrow N_0^{\ve_1}$ for some constant $K\geq 1$ depending only on $\ve_1$ and $\ve_2$. 
It should also be noted that that $N_0$ is a deformation retract of  $N$.
The end of $N_0$ is called the \emph{relative end} of $N$.
Each component of the boundary $\part N_0$ is 
either a Euclidean torus or a Euclidean open annulus. 
Since any parabolic cusp neighbourhood  of $N$ is covered by a horoball in $\hh^3$ based at a single point of $\wh{\mathbf{C}}$, the boundary at infinity $\part_\infty N_0$ of $N_0$ is equal to $\part_\infty N$.

A sequence $\{(X_n,x_n)\}$ of complete metric spaces with base points \emph{converges geometrically} (in the sense of Gromov) to a complete 
metric space $(Y,y)$ if there exist $(K_n, L_n)$-quasi-isometric, $L_n$-dense map $g_n:B_{R_n}(X_n, x_N)\rightarrow B_{K_nR_n}(Y,y)$ with $K_n\searrow 1, L_n \searrow 0$ and $R_n\rightarrow \infty$, where $B_R(X,x)$ denotes the $R$-metric ball in $X$ centred at $x$.
A sequence of Kleinian groups $\{G_n\}$ is said to converge \emph{geometrically} to a Kleinian group $G$ if (i) each $\gamma\in G$ is the limit of a sequence $\{\gamma_n\}$ with $\gamma_n\in \Gamma_n$ and (ii) the limit of any convergent sequence $\{\gamma_{n_i}\}$ with $\gamma_{n_i}\in \Gamma_{n_i}$ is an element of $G$.
It is well known that $\{\hh^3/G_n\}$ converges geometrically to $\hh^3/G$ if we set basepoints to be the projections of a common basepoint point $x_0$ in $\hh^3$ if and only if $\{G_n\}$ converges to $G$ geometrically.
Refer to \cite{jm}, \cite[Chapter E]{bp} for more details on properties of geometric limits.

Suppose that $\Sigma$ is  an open  essential subsurface of $S$, possibly $S$ itself.
The Teichm\"{u}ller space of $\Sigma$ is denoted by $\teich(\Sigma)$, for which we assume that every frontier or puncture corresponds to a parabolic cusp.
For a point $\sigma \in \teich(\Sigma)$, the surface $\Sigma$ with a hyperbolic metric representing $\sigma$ is denoted by $\Sigma(\sigma)$.
A proper map $f$ from $\Sigma(\sigma)$ to a hyperbolic 3-manifold $N$ with $\sigma\in \teich(\Sigma)$ is called a 
\emph{pleated surface} realising a geodesic lamination $\lambda$ in $\Sigma(\sigma)$ if $f$ satisfies the following conditions.
\begin{enumerate}[\rm (i)]
\item $f$ maps each parabolic cusp of $\Sigma(\sigma)$ to a parabolic cusp in $N$.
\item The path-metric induced from $N$ by $f$ coincides with $\sigma$, that is, for any rectifiable path $\alpha$ in $\Sigma(\sigma)$, its image $f(\alpha)$ is also a rectifiable path in $N$ with $\mathrm{length}_{\Sigma(\sigma)}(\alpha)=\mathrm{length}_N(f(\alpha))$.
\item $f(l)$ is a geodesic in $N$ for each leaf $l$ of $\lambda$.
\item For each component $\Delta$ of $\Sigma\setminus \lambda$, the restriction $f|\Delta$ is a totally geodesic immersion into $N$
\end{enumerate}
%Check the condition on the boundary.

A relative end $e$ of  hyperbolic 3-manifold $N$ is said to be  \emph{topologically tame} if  there is a properly embedded compact surface $F$ in $N_0$ which separates a submanifold containing $e$ which is  homeomorphic to $F\times [0,\infty)$.
All topologically tame ends of hyperbolic 3-manifolds considered in this paper are assumed to be \emph{incompressible}, 
\ie  the inclusion $F\subset N$ is $\fd$-injective.
A topologically tame relative end $e$ is called \emph{geometrically finite} if $e$ has a neighbourhood which intersects no closed geodesics.
(Here we need to assume $e$ to be topologically tame since we are considering also the case when $\pi_1(N)$ is infinitely generated.)
For a geometrically finite end, the conformal structure $\nu(e)$ on the component of $\part_\infty N$ corresponding to $e$ is defined to be the end invariant of $e$.
If $\G$ itself is geometrically finite, then every relative end of $N$ is geometrically finite.

As was shown by Bonahon \cite{bon}, if $e$ is topologically tame and incompressible but not geometrically finite, then 
there exists a sequence of closed geodesics tending to $e$ in a neighbourhood $E \cong F \times [0,\infty)$ of $e$ which are homotopic in $E$ to essential simple closed curves $c_n$ on $F$. 
Moreover, it is shown in \cite{th1} that $\{c_n\}$ converges in $\mathcal{UML}(\Int F)$ to a lamination $\nu(e)$ contained in $\mathcal{EL}(\Int F)$ which is determined uniquely, independently of the choice of closed geodesics tending to $e$.
This $\nu(e)$ is  called the \emph{ending lamination} of $e$.
In this situation, we say that the relative end $e$ is \emph{simply degenerate} and define the end invariant of $e$ to be the ending lamination $\nu(e)$.
An end which is not topologically tame is called \emph{wild}.
(Recall that we are {\em not} assuming the fundamental group of $N$ is finitely generated.)
Any reasonable invariant for a wild end is not know up to now.
This forces us to define the \emph{end invariants} of $N$ to be only those of topologically tame relative ends of $N$.

%Suppose that $f_n:F_n\rightarrow N$ are $\fd$-injective proper embeddings exiting $e$ and excising submanifolds from $N$ adjacent to $e$.
%Then we say that $e$ is \emph{homologically finite} if $\sup_n\{-\chi(F_n)\}<\infty$. 
%Relative ends $e_i$ of $N$ are \emph{uniformly homologically finite} if there exists a sequence of 
%$\fd$-injective embeddings $f_n^{(i)}$ exiting $e_i$ as above with $\sup_n\{-\chi(F_n^{(i)})\}<C$ for some constant $C$ 
%independent of $e_i$.

%

\section{Brick manifolds}\label{S_2}

\subsection{Embeddings of  brick manifolds with infinite bricks}\label{SS_EIBM}
We first introduce some notation for denoting the union of sets in a family which is convenient in the following discussion on brick manifolds.
Let $\mathcal{Y}=\{Y_\alpha\}_{\alpha\in A}$ be a family of subsets of  some set $X$.
We denote by $\bigvee\mathcal{Y}$  the subset $\bigcup_{\alpha\in A}Y_\alpha$ of $X$.
It should be noted that even when we are considering a sequence of families $\{\mathcal Y_n\}$ of subsets of $X$,  the union $\bigvee \mathcal{Y}_n$ is taken for each $n$.
%, but $\bigcup_n\mathcal{Y}_n$ 
%is the family of subsets of $X$ each of whose elements belongs to at least one of $\mathcal{Y}_n$.  

Now we shall give a precise definition of  brick manifolds, upon which we have touched lightly before stating the main results in \S\ref{main results}. 
Model manifolds of geometric limits which we shall use to prove our main results have structures of brick manifolds as we explained there.
%primitive?????`???????????D
%A \emph{primitive brick} $\hat B$ is a 3-manifold homeomorphic to $\hat F\times J$, where $\hat F$ is either $S$ itself or a connected open 
% essential subsurface of $S$ and $J$ is either $[0,1]$ or $[0,1)$ or $(0,1]$.
%In the latter two cases of $J$, the  primitive brick $\hat B$ is said to be \emph{half open}.
%The  primitive brick is regarded  the  structure of a fibre bundle with fibres $\hat F\times \{t\}$ $(t\in J)$.
%We set $\part_-\hat B=\hat F\times \{0\}$ and $\part_+\hat B=\hat F\times \{1\}$.
%For a half-open primitive brick $\hat B$,  one of  $\part_- \hat B, \part_+ \hat B$ which is not contained in 
%$\hat B$ is called the \emph{ideal front} of $\hat B$ and the other the \emph{real front}.
%We also call $\partial_- B$ the \emph{lower front} and $\partial_+ B$  the \emph{upper front}.

Throughout this subsection, $S$ denotes some fixed surface with $\xi(S) \geq 4$.
A brick is a $3$-manifold homeomorphic to $F \times J$ for a compact essential subsurface $F$ of $S$ with $\xi(F) \geq 3$ and $J$ is either $[0,1]$ or $[0,1)$ or $(0,1]$.
In the latter two cases of $J$, the  brick is said to be \emph{half open}.
We define $\xi(B)$ to be $\xi(F)$.
For a brick $B$, we set $\part_- B=  F\times \{0\}$ and $\part_+ B= F\times \{1\}$ and called the {\em upper front} and  the {\em lower front} respectively, even when $B$ is half open.
When $B$ is half open, a front which is not contained in $B$ is called the {\em ideal front} of $B$.
On the other hand, $\partial F \times J$ is called the vertical boundary of $B$, and is denoted by $\partial_\v B$.
A brick $B=F \times J$ has two foliations: the horizontal (codimension-1) foliation whose leaves consist of $F \times \{t\}$ and vertical (codimension-2) foliation whose leaves consist of $\{x\} \times J$.
A map from a brick to  $S \times I$ (where $I$ is an interval in $\reals$) is said to be {\em leaf-preserving} when leaves of the horizontal and the vertical foliations are mapped to leaves of the corresponding foliation of the range.
Here, for $S \times I$,  the horizontal foliation consists of $S \times \{t\}$ whereas the vertical foliation consists of $\{x\} \times I$.

%A set $\hat\ck=\{\hat B_1,\dots,\hat B_m\}$ of finitely many primitive bricks realised as subsets of a 3-manifold 
%$X$ with pairwise disjoint interiors is said to be 
%a finite \emph{primitive brick complex} if (i) $\bigcup_{i=1}^m \hat B_i$ is connected and 
%(ii) for any two bricks $\hat B_i,\hat B_j$ in $\hat\ck$ with $\hat F_{ij}=\hat B_i\cap \hat B_j\neq \eset$,  there exists a fibre-preserving embedding 
%$\eta:\hat B_i\cup \hat B_j\rightarrow S\times [-1,1]$ with $\eta(\hat B_i)\subset S\times [-1,0]$, $\eta(\hat B_j)\subset 
%S\times [0,1]$ such that $\eta(\hat F_{ij})$ is an open  essential subsurface of $S\times \{0\}$.
%We call $\hat F_{ij}$  the \emph{joint} of $\hat B_i$ and $\hat B_j$, and the joint is said to be \emph{inessential} if 
%$\part_- \hat B_i=\hat F_{ij}=\part_+ \hat B_j$.
%The union $\bigvee \hat \ck$ is called a \emph{primitive brick manifold} with primitive brick decomposition $\hat\ck$.

Before defining  brick complexes and brick manifolds in general, we shall first define finite brick complexes and finite  brick manifolds.
A finite brick complex is a family of finitely many bricks $\ck=\{B_1, \dots , B_m\}$ realised as subsets of a 3-manifold  with pairwise disjoint interiors satisfying the following two conditions:
\begin{enumerate}
\item $\bigcup_{i=1}^m  B_i$ is connected.
\item For any two bricks $ B_i, B_j$ in $\ \ck$ with $ F_{ij}=  B_i\cap   B_j\neq \eset$,  there exists a leaf-preserving embedding 
$\eta: B_i\cup   B_j\rightarrow S\times [-1,1]$ with $\eta(  B_i)\subset S\times [-1,0]$, $\eta (B_j)\subset 
S\times [0,1]$ such that $\eta(F_{ij})$ is an   essential subsurface of $S\times \{0\}$.
\end{enumerate}
The union $\bigvee  \ck$ is called a \emph{finite brick manifold} with brick decomposition $\ck$.
We call $F_{ij}$  in the second condition above the \emph{joint} of $ B_i$ and $ B_j$.
A joint $F_{ij}$ is said to be \emph{inessential} if 
$\part_-  B_i=  F_{ij}=\part_+   B_j$.

%Note that a primitive brick manifold has a local fibre-bundle structure induced from those on $\hat B_i$.
%Let $\{\hat\ck_n\}_{n=1}^\infty$ be an ascending sequence of  brick complexes with $\hat\ck_n\subsetneqq \hat\ck_{n+1}$.
%Then the union $\hat\ck=\bigcup_{n=1}^\infty \hat\ck_n$ is called an infinite primitive brick complex, and 
%$\bigvee \hat\ck$ is said to be a primitive brick manifold with primitive brick decomposition $\hat\ck$.
%For a compact core $F$ of $\hat F$, the submanifold $B=F\times J$ is called a \emph{brick}, which is a deformation retract of the primitive brick $\hat B=\hat F\times J$.
%We set $\part_\pm B=\part_\pm \hat B\cap B$ and $\part_{\mathbf{v}}B=\part F\times J$, that is, 
%$\part_{\mathbf{v}}B=\part B\setminus \Int(\part_- B\cap \part_+ B)$.
%For a primitive brick complex $\hat\ck=\{\hat B_j\}$, we can take the fibres $F_j$ of bricks $B_j$ so that 
%$F_{ij}=B_i\cap B_j$ is a compact core of $\hat F_{ij}$ 
%whenever $\hat F_{ij}=\hat B_i\cap \hat B_j\neq \eset$ for $i\neq j$, and $\part_- B_i= F_{ij}=\part_+ B_j$ 
%whenever $\hat F_{ij}$ is inessential.
%Then we say that $\ck=\{B_j\}$ is a \emph{brick complex} and the union $M=\bigvee \ck$ is a 
%\emph{brick manifold} with brick decomposition $\ck=\{B_j\}$.
%When a fibre-preserving embedding $\eta:M\rightarrow S\times (0,1)$ is given, a half-open brick $B$ in $\ck$ is 
%a \emph{peripheral} with respect to $\eta$ if the ideal front of $\eta(B)$ is contained in 
%$S\times \{0\}\cup S\times \{1\}$.

Now we define brick complexes and brick manifolds.
Let $\{\ck_n\}_{n=1}^\infty$ be an ascending sequence of finite brick complexes.
Then the union $ \ck=\bigcup_{n=1}^\infty  \ck_n$ is called a  brick complex, and 
$\bigvee \ck$ is said to be a  brick manifold with  brick decomposition $\ck$.
In the situation where a leaf-preserving embedding $\eta:M\rightarrow S\times (0,1)$ of a brick manifold is given, a half-open brick $B$ in $\ck$ is  said to be
\emph{peripheral} with respect to $\eta$ if the ideal front of $\eta(B)$ is contained in 
$S\times \{0\}\cup S\times \{1\}$.

The following lemma is a key step in the proof of Theorem \ref{thm_a}, to whose proof the rest of this subsection is devoted.
In the setting of Theorem \ref{thm_a}, the model manifold $M$ for $N$ is a brick manifold which is a geometric limit of model manifolds for $(\hh^3/G_n)_0$.
It follows that $M$ contains an ascending exhausting sequence of finite brick manifolds which admit leaf-preserving embeddings into $S \times (0,1)$.
The following lemma then implies that there is a leaf-preserving embedding of $M$ itself into $S \times (0,1)$.

\begin{lemma}\label{l_21}
Let $\{M_n\}$ be a sequence of finite brick manifolds with   brick complexes $\ck_n$ such that $\ck_n\subsetneqq \ck_{n+1}$.
If there exists a leaf-preserving embedding $\eta_n:M_n\rightarrow S\times (0,1)$ for each $n\in \nn$, then 
the brick manifold $M =\bigcup_{n=1}^\infty M_n$ has the following properties.
\begin{enumerate}[\rm (i)]
\item
There exists a leaf-preserving embedding $\eta_\infty:M\rightarrow S\times (0,1)$.
\item
The ends of $M$ are countable.
\item
If $B\in \ck_m$ is peripheral with respect to $\eta_n$ for all $n\geq m$, then 
$B$ is also peripheral with respect to $\eta_\infty$.
\end{enumerate}
\end{lemma}

We use the symbols $\pr_\mathrm{h}: S \times [0,1] \rightarrow [0,1]$ to denote the projection to the second factor, and $\pr_\v: S \times [0,1] \rightarrow S$ to denote that to the first factor.
For any brick $B_i \in \ck_n$, we set $\pr_\mathrm{h} \circ \eta_n(\part_-B_i)=\alpha_{i,n}$ and $\pr_\mathrm{h} \circ\eta_n(\part_+B_i)=\beta_{i,n}$.
(Here we regard $\eta_n$ as extended to ideal fronts continuously.)
A half-open brick $B_i$ is peripheral with respect to $\eta_n$ if and only if either $\alpha_{i,n}=0$ or $\beta_{i,n}=1$.
For integers $n,m$ with $1\leq n\leq m$, let $T_{n,m}$ be the subset of $[0,1]$ consisting of the $\alpha_{i,m}, \beta_{i,m}$ for $B_i\in \ck_n$, and set $T_n=T_{n,n}$.
Consider the correspondence 
$\tau_{n,m}:T_{n}\rightarrow T_{m}$ which transfers $\alpha_{i,n},\beta_{i,n}$ respectively to $\alpha_{i,m},\beta_{i,m}$.
Note that $\tau_{n,m}$ is not necessarily a map.
In fact, it may occur that  $\alpha_{i,n}=\alpha_{j,n}$ (resp.\ $\alpha_{i,n}=\beta_{j,n}$) but $\alpha_{i,m}\neq \alpha_{j,m}$ (resp.\ $\alpha_{i,m}\neq \beta_{j,m}$) etc.

To prove Lemma \ref{l_21}, we shall make use of  the following two kinds of rearrangement for $\{\ck_n\}$.
In Rearrangement I, by taking a subsequence and modifying the embeddings $\eta_n$, we shall make $\alpha_{i,n}$ and $\beta_{i,n}$ independent of $n$.

\subsection*{Rearrangement I}
Fix $n \in \naturals$.
Then by passing to a subsequence, we can make $\tau_{m, m'}|T_{n, m}$ is a map for $m' > m \geq n$.
Moreover, since there are only finitely many bricks in $\ck_n$, there are only finitely many ways to give them an order.
Therefore,  we can take a subsequence $\{\ck_{n_k}\}$ of $\{\ck_m\}_{m\geq n}$ so that the restriction $\tau_{n_k,n_l}|T_{n,n_k}:T_{n,n_k}\rightarrow T_{n,n_l}$ is an order-preserving bijection whenever  
$n_k\leq n_l$. 
For any $k \geq n$, we define a new embedding  $\eta_k$ to be the old $\eta_{n_k}|M_k$.
Repeating the same argument, we can assume that $\tau_{m_1,m_2}|T_{n,m_1}:T_{n,m_1}\rightarrow T_{n,m_2}$ is 
an order-preserving bijection for any $n\leq m_1\leq m_2$.
Since $\eta_n$ and $\eta_m$ embed    $\{\partial_- B_i,\partial_+ B_i \mid B_i \in \ck_n\}$ in the same order, we can deform the new $\eta_n$  by ambient isotopies of $S\times I$ in such a way that we have
$\alpha_{i,n}=\alpha_{i,m}$ and $\beta_{i,n}=\beta_{i,m}$ for any $n\leq m$ and any $i$ with $B_i\in \ck_n$.
In particular,  $T_n$ can be made a subset of $T_m$.

\bigskip

%For any subset $\Sg$ of $S$ and any $a,b\in I$ with $a\leq b$, the subset $\Sg\times [a,b]$ 
%(resp.\ $\Sg\times \{a\}$) of $S\times I$  is denoted by $\Sg_{[a,b]}$ (resp.\ by $\Sg_a$).

\subsection*{Rearrangement II}
Set $T_n=\{a_0,a_1,\dots,a_t\}$, where elements are arrayed in the increasing order, and $R^j_n=\eta_n^{-1}(S\times [a_{j-1},a_j])$.
See Figure\ \ref{fig2_1}.
Passing again to a subsequence of $\{\eta_n\}$ if necessary, we may assume that, for any $j=1,\dots,t$, 
all $\eta_m|R^j_n$ $(m\geq n)$ define 
the same embedding up to isotopies and changes of  the markings of $S\times [a_{j-1},a_j]$, \ie there exists an orientation-preserving 
homeomorphism $\gamma_{m,n}:S \times [a_{j-1},a_j]\rightarrow S\times [a_{j-1},a_j]$ with 
$\gamma_{m,n}\circ (\eta_m|R^j_n)=\eta_{n}|R^j_n$.
For, if we fix a topological type of a compact essential subsurface $F$ of $S$, there are only finitely many embeddings, up to isotopies and  changes of markings, of $F$ into $S$ as an essential subsurface.

We note that this $\gamma_{m,n} \circ (\eta_m|R^j_n)$ may not extend to the entire $M_m$.
In fact even for a brick $B$ in $\ck_m \setminus \ck_n$ with both $\partial_+B$ and $\partial_- B$ contained in $M_n$, it may be possible that $\gamma_{m,n} \circ \eta_m(\partial_- B)$ and $\gamma_{m,n} \circ \eta_m(\partial_+ B)$ are not isotopic.
%%%%%%%%%%%%%%%%%%%%%%%%%%%%%%%%%%%%%%%%%%%%%%%%%%%%%%%
\begin{figure}[hbtp]
\centering
\scalebox{0.5}{\includegraphics[clip]{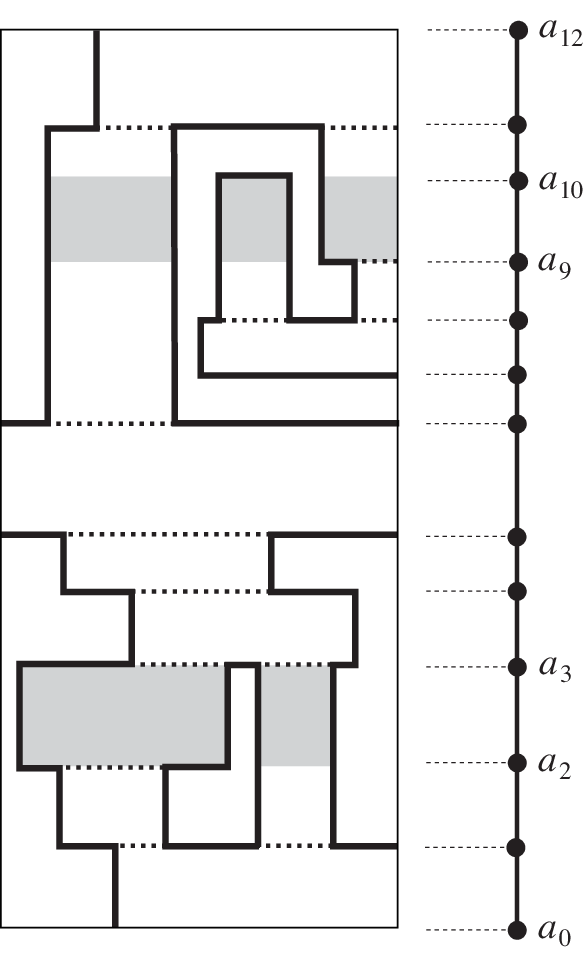}}
\caption{The union of the shaded regions in the lower (resp.\ higher) level 
is $\eta_n(R^3_n)$ (resp.\ $\eta_n(R^{10}_n)$).}
\label{fig2_1}

\end{figure}
%%%%%%%%%%%%%%%%%%%%%%%%%%%%%%%%%%%%%%%%%%%%%%%%%%%%%%%

\bigskip
%To prove Lemma  \ref{l_21}, we need to modify the embedding $\eta_m$ further so that the embedding

To construct embeddings of the $M_n$ which stabilise on each brick after finite steps, we need to modify the embeddings $\eta_n$ as above by composing \lq\lq twists" which will be defined below.
Before the definition, we shall observe the local structure of the embeddings $\eta_n(M_n)$ at horizontal levels near the accumulation points of $\cup_m T_m$.

%Denote the bricks of $\ck_n$ by $B_j^{(n)}$.
% and 
%set $L_n=\eta_n(M_n)$ and $D_j^{(n)}=\eta_n(B_j^{(n)})$.
For each $c \in I$ and $n \in \naturals$, we call $\Sg_c^{(n)}:=(S\times \{c\})\setminus \Int (\eta_n(M_n))$  the \emph{slit} for $\eta_n(M_n)$ at $c$.
By Rearrangement I and II,  for all sufficiently large $n$, the topological type of  $\Sg_c^{(n)}$ does not vary with $n$.
The slit $\Sg_c^{(n)}$ is said to be \emph{stable} if all the $\Sg_c^{(m)}$ $(m\geq n)$ are homeomorphic.
For $c \in I$, we define $\chi_{\stab}(\Sg_c)$ to be $\chi(\Sg_c^{(n)})$ for stable $\Sg_c^{(n)}$.
Since the embedding of every brick intersects $S \times \{c\}$ at an essential subsurface with negative Euler characteristic, we see that $\chi(\Sg_c^{(n)})$ is monotone increasing and once the equality $\chi(\Sg_c^{(n)})=\chi_{\stab}(\Sg_c)$ holds, $\Sg_c^{(n)}$ is stable.

Let $T_\infty'$ be the set of accumulation points of $T_\infty :=\bigcup_{n\geq 1}T_n$.
For $c \in T_{\infty}'$, consider a sufficiently large $n$ such that  $\Sg_c^{(n)}$ is stable.
Suppose that $B_1^{(n)},\dots,B_k^{(n)}$ are the bricks in $\ck_n$ with $\eta_n(B_i^{(n)})\cap S \times \{c\} \neq \eset\, (i=1, \dots , k)$.
Take a sufficiently small $\delta>0$ so that $S \times ([c-\delta,c) \cup(c,c+\delta])$ meets none of the images  under $\eta_n$ of the fronts 
of $B_i^{(n)}$ $(i=1,\dots,k)$.   
Then we call the set 
$$Q_\delta(\Sg_c^{(n)}):=\bigl(S \times ([c-\delta,c)\cup (c,c+\delta])\setminus \eta_n(B_1^{(n)})\cup\dots\cup \eta_n(B_k^{(n)})\bigr)\cup \Sg_c^{(n)}$$
 the \emph{$\delta$-region} of the slit $\Sg_c^{(n)}$ for $\eta_n(M_n)$.
 See Figure \ \ref{fig2_2}.
 When $c$ is $0$ or $1$, we need to modify the definition a little: we define $Q_\delta(\Sg_c^{(n)})$ to be $S \times (0,\delta]$ when $c=0$ and  $S \times [1-\delta, 1)$ when $c=1$.
%%%%%%%%%%%%%%%%%%%%%%%%%%%%%%%%%%%%%%%%%%%%%%%%%%%%%%%
%D ??\eta_n(B)$?????????D
\begin{figure}[hbtp]
\centering
\scalebox{0.5}{\includegraphics[clip]{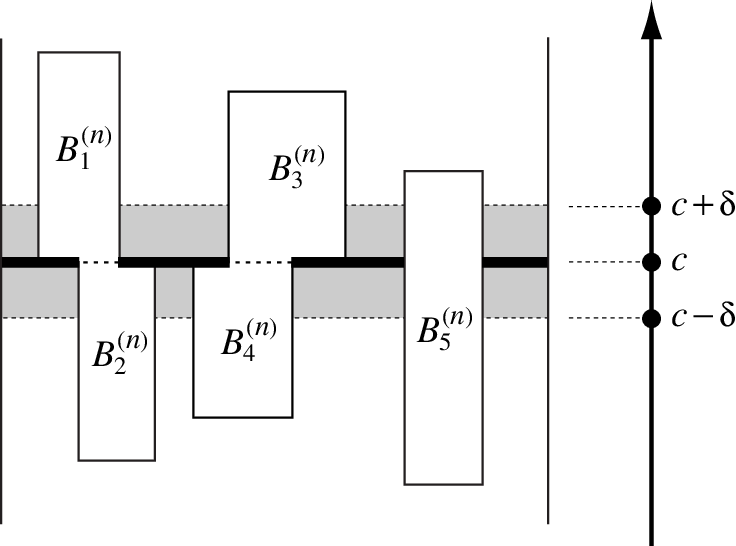}}
\caption{The union of the bold horizontal segments represents $\Sg_c^{(n)}$.
The union of $\Sg_c^{(n)}$ and the shaded regions is the $\delta$-region $Q_\delta(\Sg_c^{(n)})$.}
\label{fig2_2}
\end{figure}

%%%%%%%%%%%%%%%%%%%%%%%%%%%%%%%%%%%%%%%%%%%%%%%%%%%%%%%

%, where $-\chi_{\min}(\Sg_c)$ is defined to be $\min_{m}\{-\chi(\Sg_c^{(m)})\}$.

For $m \geq n$, if $\Sg_d^{(m)}$ $(d\in I)$ is contained in $Q_\delta(\Sg_c^{(m)})\setminus \Sg_c^{(m)}$ then $\chi(\Sg_d^{(m)})\geq \chi_{\stab}(\Sg_c)$.
If  the equality holds, then $\Sg_d^{(m)}$ is parallel to $\Sg_c^{(m)}$ in $S\times [0,1]\setminus \eta_m(M_m)$ (for, since $M_m$ is connected, there cannot be a brick obstructing the parallelism), and even if the strict inequality holds, $\pr_v(\Sg_d^{(m)})$ is contained in $\pr_v(\Sg_c^{(m)})$ (up to isotopy).
Therefore, in particular if $d$ lies on a side of $c$ from which $T_\infty$ accumulates to $c$, the strict inequality $\chi(\Sg_d^{(m)}) > \chi(\Sg_c^{(m)})$ holds.
%, which implies that $J\cap T_{\infty,s}'=\eset$, where $J$ is an open interval between $c$ and $d$.
%(The last implication uses the assumption that  $M_n$ is connected.)
%If $d$ lies on the side of $c$ from which $T_\infty$ accumulates to $c$, we can take a large $m$ so that $\Sg_d^{(m)}$ is parallel into $\Sg_c^{(m)}$.
Since the only bricks that contribute to increase $\chi(\Sg_c^{(m)})$ are those with one of their fronts lying on $S \times \{c\}$, and their other fronts lie outside the $\delta$-region, we see that even for  $m$ smaller than $n$, we have the inequality $\chi(\Sg_d^{(m)}) \geq \chi(\Sg_c^{(m)})$.
Thus we have shown the following claim.
\begin{claim}
\label{delta-region}
For $c\in T_{\infty}'$, there exists $\delta(c)>0$ depending only on $c$ such that $\chi_{\stab}(\Sg_d)\geq \chi_{\stab}(\Sg_c)$ if $d$ lies in $[c-\delta(c), c+\delta(c)]$.
In particular, if $d$ lies on a side of $c$ from which $T_\infty$ accumulates to $c$, we have $\chi_{\stab}(\Sg_d)>\chi_{\stab}(\Sg_c)$.
% and  on the side from which $T_\infty$ accumulates to $c$.
%Without the last condition of the side in which $d$ lies, the inequality $\chi_{\stab}(\Sg_d)\geq  \chi_{\stab}(\Sg_c)$ holds.

In general, for every $n$, the inequality $\chi(\Sg_d^{(n)})\geq  \chi(\Sg_c^{(n)})$ holds provided that $d$ lies in $[c-\delta(c), c+\delta(c)]$, and $\pr_v(\Sg_d^{(n)})$ is contained in $\pr_v(\Sg_c^{(n)})$ up to isotopy.
\end{claim}

For an integer $s\geq 1$, we define $T_{\infty,s}'$ to be the subset of $T_\infty'$ consisting of elements $c\in T_\infty'$ for which 
$-\chi_{\stab}(\Sg_c)=s$.
Suppose that $c$ is contained in $T_{\infty, s}'$.
Then by the claim above, if $d$ lies on a side of $c$ from which $T_\infty$ accumulates to $c$, and $|d-c|< \delta(c)$, then $-\chi_{\stab}(\Sg_d) < s$.
Taking into account also the side from which $T_\infty$ does not accumulate to $c$, we can take possibly smaller $\delta(c)$ such that for {\em any $\Sigma_d^{(n)}$ with $d \in T_\infty \cup T_\infty'$ contained in $Q_{\delta(c)}(\Sg_c^{(n)})\setminus \Sg_c^{(n)}$, we have $-\chi_{\stab}(\Sg_d)<s$}.
This implies that $(c-\delta(c), c+\delta(c)) \cap T_{\infty, s}'=\{c\}$.
It follows that  $T_{\infty,s}'$ is a countable subset of $[0,1]$ for every $s$, and hence so is $T_\infty'$.

\smallskip

By making $\delta(c)$ smaller if necessary, 
we can assume 
%that
% $Q_{\delta(c)}(\Sg_c^{(w)})$ is disjoint from 
%$\eta_w(B)$ for all $w\geq n$ and 
that  for any $c,c'
\in T_\infty'$, either $[c-\delta(c),c+\delta(c)]$ 
and $[c'-\delta(c'),c'+\delta(c')]$ are  disjoint or one of them contains the other.
Since $T_\infty\cup T_\infty'$ is compact, there exists a finite subset $\{c_1,\dots,c_k\}$ of 
$T_\infty'$ such that $T_\infty\setminus \bigcup_{i=1}^k[c_i-\delta(c_i),c_i+\delta(c_i)]$ covers $T_\infty \cup T_\infty'$ except for finitely many elements $b_1,\dots,b_u$ of $T_\infty$.
See Figure \ref{fig2_4}.

For a point $a \in T_\infty$ we define $c(a)$ to be a point in $T_\infty'$ such that $[c(a)-\delta(c(a)), c(a)) \cup (c(a), c(a)+\delta(c(a))]$ contains $a$ and is the smallest among such sets with respect to the inclusion.
In the case when there is no such set, \ie if $a$ is among $b_1, \dots, b_u$,  we define $c(a)$ to be $1$ by convention.

%%%
\begin{figure}[hbtp]
\centering
\scalebox{0.5}{\includegraphics[clip]{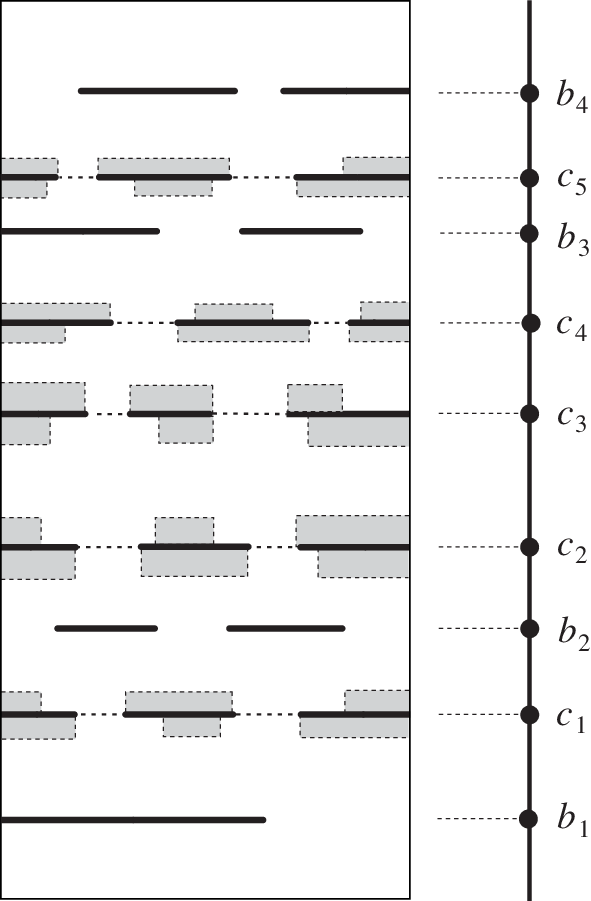}}
\caption{}
\label{fig2_4}
\end{figure}
%$D^{(w)}$??$\eta_w(B)$?????????D
%%%%%%%%%%%%%%%%%%%%%%%%%%%%%%%%%%%%%%%%%%%%%%%%%%%%%%%

Now we shall define maps called   twists, which will be used to modify embeddings.
Let $F$ be a compact essential subsurface of $S\times \{a\}$ with $0<a<1$ and $\varphi:F\rightarrow F$ an orientation-preserving homeomorphism 
such that $\varphi|\part F$ is the identity.
Consider a 3-manifold $N_\varphi$ obtained from $S\times [0,1]\setminus \Int F$ by identifying the $(\pm)$-sides 
$F^{(\pm)}$ of $F$ by $\varphi:F^{(-)}\rightarrow F^{(+)}$ instead of the identity.
The original $S\times [0,1]\setminus \Int F$ is naturally regarded as a subset of $N_\varphi$. 
We say that $N_\varphi$ is the manifold obtained from $S\times [0,1]\setminus \Int F$ by the \emph{$\varphi$-twist along} $F$.
Thus obtained manifold is homeomorphic to $S \times [0,1]$, by a homeomorphism which we specify as follows.
Let $C_0$ be either $F \times [0,a)$ or $F \times (a,1]$.
Then we have a homeomorphism $\xi_0:N_\varphi\rightarrow S\times [0,1]$ such that   $\xi_0|(N_\varphi\setminus 
C_0)$ is the identity, whereas $\xi_0|C_0$ is $\varphi ^{-1} \times \mathrm{id}_{[0,a)}$ if $C_0$ is $F \times [0,a)$, and $\varphi \times \mathrm{id}_{(a,1]}$ if $C_0$ is $F \times (a,1]$.
The part of $N_\phi$ where the homeomorphism is not the identity is called the \emph{affected region} of the twist.
In this case, $C_0$ is the affected region.
See Figure \ref{fig2_3}\,(a).
%%%%%%%%%%%%%%%%%%%%%%%%%%%%%%%%%%%%%%%%%%%%%%%%%%%%%%%
\begin{figure}[hbtp]
\centering
\scalebox{0.5}{\includegraphics[clip]{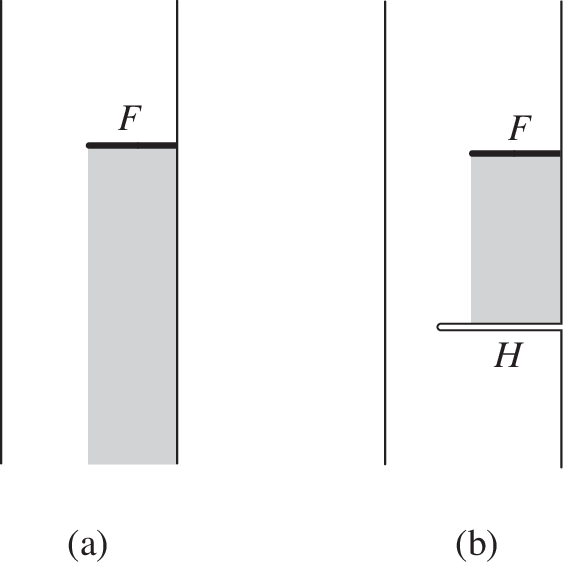}}
\caption{The shaded parts represent the affected regions.}
\label{fig2_3}
\end{figure}
%%%%%%%%%%%%%%%%%%%%%%%%%%%%%%%%%%%%%%%%%%%%%%%%%%%%%%%

For the proof of  Lemma \ref{l_21}, we need to reduce the affected region using the following trick.
%We shall explain our trick under the assumption that  $C_0$ is $F \times [0,a)$.
%The same trick works also for the case when $C_0$ is $F \times (a,1]$.
Let $H$ be a non-peripheral horizontal essential subsurface in $S\times [0,1]$ with $\mathrm{pr}_\mathrm{v}(H)\supset \mathrm{pr}_\mathrm{v}(F)$ which  lies in $S \times \{b\}$ for some 
$b$ with $F \times\{b\} \subset C_0$.
Then there exists a homeomorphism $\xi_1: N_\varphi\setminus H\rightarrow S\times [0,1]\setminus H$ whose affected 
region is $C_1=F\times \langle b,a\rangle$, where $\langle a,b \rangle$ denotes $(a,b)$ if $b>a$ and $(b,a)$ if $a> b$, \ie $\xi_i|N_\varphi \setminus C_1$ is the identity.
 See Figure \ref{fig2_3}\,(b).
 In the following proof of Lemma \ref{l_21}, we shall use this trick letting $b$ be $c(a) \in T_\infty'$ defined above. 
%
%We note that even if there exists another essential subsurface $J$ of $S \times \{c\}$ for some $b<c<a$ satisfying $\mathrm{pr}_\mathrm{v}(\Int F)
%\cap \mathrm{pr}_\mathrm{v}(\Int J)\neq \eset$ and $R=\mathrm{pr}_\mathrm{v}(F)\setminus \mathrm{pr}_\mathrm{v}(J)\neq \eset$,
% we do not take into account  a  partial contribution of $J$ for  reducing  the affected region of the $\varphi$-twist.
%To be more precise, the homeomorphism $\xi_2:N_\varphi\setminus H\cup J\rightarrow S\times [0,1]\setminus H\cup J'$ is the restriction 
%of $\xi_1$, where $J'=\xi_1(J)$.
%In particular, the affected region of $\xi_2$ is $C_1\setminus J$,  not $C_1\setminus R\times (b,c)$. 
%If the affected region is contained in a subset $Q$ of $S\times [0,1]$, then we say that 
%the effect of the twist is \emph{absorbed} in $Q$ under $\xi_1$.
%\xi_0??????\xi_1???????D

Now we are ready to formally start the proof of Lemma \ref{l_21}.

\begin{proof}[Proof of Lemma \ref{l_21}]
First we shall show the part (i).
We shall define inductively a leaf-preserving embedding $h_n: M_n \rightarrow S \times [0,1]$ with $h_n^{-1}(S \times [a_{j-1}, a_j])=\eta_n^{-1}(S \times [a_{j-1}, a_j])$ for $T_n=\{a_0,a_1,\dots, a_t\}$.
Here $\eta_n$ denotes the one which we obtained after applying Rearrangements I and II for the original $\eta_n$.
We set $h_1=\eta_1$.
We assume that $h_{n-1}$ has already been defined, and shall define $h_n$ inductively so that the $h_n$ retain the properties obtained by Rearrangements I and II.

Recall that we defined $R_n^j$ to be $\eta_n^{-1}(S \times [a_{j-1}, a_j])$.
By Rearrangement I, we have $R_n^j \cap M_{n-1}=R_{n-1}^j$ for any $j=1,\dots,t$.
By applying Rearrangement II for $h_{n-1}$ and $\eta_n$, we see that there exists an embedding $\hat h_n^j:R_n^j\rightarrow S\times [a_{j-1},a_j]$ such that $\hat h_n^j \circ \eta_n|R_n^j\cap M_{n-1} 
= h_{n-1}|R_n^j\cap M_{n-1}$.
We note that the union of $\hat h_n^j \circ \eta_n$ does not necessarily match up on the boundaries of the $R_n^h$ to define an entire embedding from $M_n$ to $S\times [0,1]$.
Let $\hat T_n$ be the subset of $T_n$ consisting of elements $a_j\in T_n$ for which $-\chi(\Sg_{a_j}^{(n-1)})>-\chi(\Sg_{a_j}^{(n)})$, 
where $\Sg_{a_j}^{(n-1)}=S\times \{a_j\}\setminus \Int(h_{n-1}(M_{n-1}))$ and $\Sg_{a_j}^{(n)}=S\times \{a_j\}\setminus  \Int(\bigcup_{j=1}^t \hat h^j_n \circ \eta_n(R_n^j))$ are slits for $h_{n-1}$ and $\hat h^j_n \circ \eta_n$.
In particular, $c\in \hat T_n$ implies that $\Sg_c^{(n-1)}$ is unstable.

%For a brick $B$ of $M_n$ which is not contained in $M_{n-1}$, the embeddings which $h_{n-1}$ induces on $\partial_- B$ and $\partial_+ B$ may not be isotopic even if they are defined.
To construct an embedding on the entire $M_n$ from this $\hat h_n^j \circ \eta_n$, we need to perform twist as defined before.
For each $a_j \in \hat T_n$, we  choose an orientation-preserving homeomorphism $\varphi_{a_j}:\Sg_{a_j}^{(n-1)}\rightarrow \Sg_{a_j}^{(n-1)}$  with $\varphi_{a_j}|\part \Sg_{a_j}^{(n-1)}$ being the identity  so that $\bigcup_{j=1}^t \hat h^j_n \circ \eta_n$ extends to an embedding $\hat h_n:M_n\rightarrow N_n$, where $N_n$ is the manifold obtained from $S\times [0,1]\setminus \bigcup_{a_j\in \hat T_n}\Sg_{a_j}^{(n-1)}$ by the composition of the $\varphi_{a_j}$-twists.
By our definition of $\hat h^j_n$, if we identify $N_n$ with $S\times [0,1]$ so that the non-affected regions do not move as was explained before, then the difference between $\hat h_n|M_{n-1}$ and $h_{n-1}$ is the composition of the $\varphi_{a_j}$-twists.

%As was explained before starting the proof, for any $c\in T_\infty'$, we can take  sufficiently small  $\delta(c)>0$ so that for any $d \in [c-\delta(c),c) \cup (c, c+\delta(c)]$ with $d \in T_\infty$, if the surface $\Sigma_d^{(n)}$ intersects $Q_{\delta(c)}(\Sigma_c^{(n)}
%) \setminus \Sigma_c^{(n)}$, then $-\chi_{\min}(\Sigma_d) < -\chi_{\min} (\Sigma_c)$.
%By making $\delta(c)$ smaller if necessary, we can assume 
%that
% $Q_{\delta(c)}(\Sg_c^{(w)})$ is disjoint from 
%$\eta_w(B)$ for all $w\geq n$ and 
%that  for any $c,c'
%\in T_\infty'$, either $[c-\delta(c),c+\delta(c)]$ 
%and $[c'-\delta(c'),c'+\delta(c')]$ are  disjoint or one of them contains the other.
%Since $T_\infty\cup T_\infty'$ is compact, there exists a finite subset $\{c_1,\dots,c_k\}$ of 
%$T_\infty'$ such that $T_\infty\setminus \bigcup_{i=1}^k[c_i-\delta(c_i),c_i+\delta(c_i)]$ contains 
%only finitely many elements $b_1,\dots,b_u$.
%See Figure \ref{fig2_4}.

%Let $\Sg_\infty^{(n)}$ be the union of stable slits $\Sg_{c_i}^{(n)}$ with $c_i \in \{c_1, \dots , c_k\}$, 
%and regard $\Sg_\infty^{(n)}\subset S\times [0,1]\setminus \bigcup_{a_j\in \hat T_n}\Sg_{a_j}^{(n-1)}$  naturally  as a subset of $N_n$.
%As was shown before we began the proof, for each $\Sg_{a_j}\ (a_j \in T_n)$ with $a_j \in [c_i-\delta(c_i),c_i) \cup (c_i, c_i+\delta(c_i)]$, the $\varphi_{a_j}$-twist can be absorbed in either $\Sg_{c_i}^{(n)} \times [c_i,a_j]$ or $\Sg_{c_i}^{(n)} \times [a_j,c_i]$.
Now we consider to make the affected region of the $\varphi_{a_j}$-twist smaller.
Recall that for $a_j$, there is a point $c(a_j) \in T_\infty'$ defined above such that $[c(a_j)-\delta(c(a_j)),c(a_j)) \cup (c(a_j), c(a_j)+\delta(c(a_j))]$ contains $a_j$ and  is the smallest among such sets.
By Claim \ref{delta-region}, we see that $\pr_v(\Sg_{a_j}^{(n-1)})$  is contained in $\pr_v(\Sg_{c(a_j)}^{(n-1)})$ for the embedding $h_{n-1}$.
In general, there might be other $a_k$ among $a_1, \dots, a_t$ between $a_j$ and $c(a_j)$.
By our definition of the function $c$, in this case we have $\langle a_k, c(a_k)\rangle \subset \langle a_j, c(a_j)\rangle$.
This implies that the $\varphi_{a_k}$-twist does not change the condition that $\pr_v(\Sg_{a_j}^{n-1})$  is contained in $\pr_v(\Sg_{c(a_j)}^{(n-1)})$.
(This is valid even when $c(a_j)=1$.)
Therefore, there is a homeomorphism $\xi_n:N_n\setminus \bigcup_{a_j \in \hat T_n} \Sg_{c(a_j)}^{(n)}\rightarrow S\times [0,1]\setminus \bigcup_{a_j \in \hat T_n} {\Sg_{c(a_j)}^{(n)}}'$ such that the affected region of $\varphi_{a_j}$-twist is $S \times \langle a_j, c(a_j)\rangle$, 
%absorbed in $S \times  [c_j-\delta(c_j),c_j+\delta(c_j)]$ by composing $\xi_n$,
where ${\Sg_{c(a_j)}^{(n)}}'$ is the  horizontal essential subsurfaces in $S\times [0,1]$ corresponding to the slit $\Sg_{c(a_j)}^{(n)}$ and we regard $\langle a_j, c(a_j)\rangle$ as $(a_j, 1]$ when $c(a_j)=1$.
%and the affected region does not extend beyond $c$.....
Then $\xi_n\circ \hat h_n$ extends to a leaf-preserving embedding $h_n :M_n\rightarrow S\times [0,1]$, whose restriction to $M_{n-1}$ coincides with $h_{n-1}$ outside the affected regions,
but $h_n$ may not be an extension of $h_{n-1}$ on the affected regions.
%For all $m\geq n$, an embedding $h_m$ is defined similarly.
For thus defined sequence of embeddings $h_n$,  we shall show  that the restriction $h_n|B$ to each brick $B \in \ck$ is eventually  the same map even as $n$ varies.

%%%%%%%%%%%%%%%%%%%%%%%%%%%%%%%%%%%%%%%%%%%%%%%%%%%%
\sloppy
Let $B$ be a brick of $\ck$.
Then, there is $m$ such that $M_m$ contains $B$.
Take a sufficiently large $w_0\in \nn$  so that  $w_0 >m$, and all the $\Sg_{j}^{(w_0)}$ are stable for $j\in\{b_1,\dots,b_u\}$.
This also means that all the twists along slits at $b_1,\dots,b_u$,  which are contained in no $[c-\delta(c), c+\delta(c)]$, are already done by the $w_0$-th step.
For  $n > w_0$ consider a twist  performed in the construction of $h_n$ at $a$.
If $S \times \langle a, a(c) \rangle$ is disjoint from $h_m(B)$, hence from   $h_{n-1}(B)$, the image of $B$ under $h_n$ is the same as that of $h_{n-1}(B)$.
Also, since for each $a \in T_\infty$, there are only finitely many $n$ such that $a$ is contained in $\hat T_n$, if there are only finitely many $n$ and twists at $a^n$ for  which $S \times \langle a^n, c(a^n)\rangle$ intersects $h_m(B)$, then the image of $B$ stabilises after finite steps.

Suppose that there are infinitely many  regions $S \times \langle a_j^{n(j)}, c(a_j^{n(j)})\rangle\, (n(j) \geq w_0)$ with $a_j^{n(j)} \in \hat T_{n(j)}$ intersecting $h_{n-1}(B)$.
We claim that then the $\langle a_j^{n(j)}, c(a_j^{n(j)})\rangle $ are contained in $(\pr_h(\partial_- h_m(B)), \pr_h(\partial_+ h_m(B)))$ except for finitely many of them.
Suppose, on the contrary, that  infinitely many of them, which we denote again by $\langle a_j^{n(j)}, c(a_j^{n(j)})\rangle$, are not contained in $(\pr_h(\partial_- h_m(B), \pr_h(\partial_+ h_m(B)))$.
Passing to a subsequence, we can assume that $\{a_j^{n(j)}\}$ converges to a point $b \in T_\infty'$.
This implies that $[b-\delta(b), b+\delta(b)]$ contains $a_j^{n(j)}$ for sufficiently large $j$, and that $c(a_j^{n(j)})$ is not greater than $b$ and converges to $b$ as $j \rightarrow \infty$.
Since $S \times \langle a_j^{n(j)}, c(a_j^{n(j)})\rangle$ intersects $h_{n-1}(B)$, the only possibility is that $a_j^{n(j)}$ is contained in $(\pr_h(\partial_- h_m(B), \pr_h(\partial_+ h_m(B)))$ for all large $j$.
Therefore $\langle a_j^{n(j)}, c(a_j^{n(j)})\rangle$ must be contained in $(\pr_h(\partial_- h_m(B), \pr_h(\partial_+ h_m(B)))$, which is a contradiction.

Therefore we have only to consider $\phi_{a_j^{n(j)}}$-twists such that the $\langle a_j^{n(j)}, c(a_j^{n(j)})\rangle$ are contained in $(\pr_h(\partial_- h_m(B), \pr_h(\partial_+ h_m(B)))$.
Then $\phi_{a_j}^{n(j)}$ is supported on $\Sg_{a_j}^{n(j)}$, which is disjoint from $S \times \{a_j\} \cap h_{n(j)-1}(B)$.
Therefore the embedding $h_{n(j)-1}(B)$ does not change after performing the $\phi_{a_j}^{n(j)}$-twist.
Thus we have shown that the embedding of $B$ stabilises after finite steps.
%By our choice of $\delta(c_j)$, we have $\eta_n(B) \cap Q_{\delta(c_i)}(\Sigma_{c_i}^{(n)})=\eset$ for all $n \geq w_0$.
%Since the effect of any twist along the slit $\Sg_d^{(n)}\ (d \in T_\infty)$ with $d\in [c_i-\delta(c_i),c_i+\delta(c_i)]$  
%is absorbed in $Q_{\delta(c_i)}(\Sg_{c_i}^{(n)})$ and $\eta_n(B)\cap 
%Q_{\delta(c_i)}(\Sg_{c_i}^{(n)})=\eset$ for $i=1,\dots,k$ and $n\geq w_0$, it follows that all $h_{n}|B$ with 
%$n\geq w_0$ are the same map.
%Thus a leaf-preserving embedding $\eta_\infty:M \rightarrow S\times [0,1]$ is well defined by setting $\eta_\infty|B=h_{w_0}|B$ for such $w_0$.
It follows that a leaf-preserving embedding $\eta_\infty:M \rightarrow S\times [0,1]$ is well defined by setting $\eta_\infty|B=h_{n}|B$ for large $n$.
 Since the rearranged $\eta_n$ maps $M_n$ into $S \times (0,1)$, so does $h_n$.
 Hence the image of $\eta_\infty$ lies in $S \times (0,1)$.
This completes the proof of (i).

If $B_j^{(m)}\in \ck_m$ is peripheral with respect to $h_n$ for all $n\geq m$, then 
either $\alpha_{j,n}=0$ or $\beta_{j,n}=1$ for all $n \geq m$, even after Rearrangement I.
It follows from our definition of $\eta_\infty$ that either $\alpha_{j,\infty}=0$ or $\beta_{j,\infty}=1$ holds.
This shows the part (iii).

Finally, we turn to the part (ii).
We consider the ends of the embedded image $\eta_\infty(M)$ instead of $M$ itself.
Fix a basepoint $x_0$ in $\eta_\infty(M)$.
For an end $e$ of $\eta_\infty(M)$, consider an arc $\alpha_e$ in $\eta_\infty(M)$ emanating from 
$x_0$ and tending to $e$ which meets each horizontal leaf of all bricks $\eta_\infty(B_j)$ $(B_j\in \ck)$ with $\alpha_e\cap \eta_\infty(B_j) \neq 
\eset$ transversely in a single point except for the one containing $x_0$.
This implies that $\alpha_e$ meets each $S \times \{c\}$ at most at $-\chi(S)$ points.
It follows that $\pr_{\mathrm{h}}(\alpha_e)$ converges to a point $b(e)$ of $T_\infty'$. 

Now, for $c \in T_\infty'$, 
suppose that $e_1,\dots,e_m$ are distinct $m$ ends of $\eta_\infty(M)$ with $b(e_1)=\cdots=b(e_m)=c$.
For a sufficiently large $n$,  these ends are contained in distinct components of 
$\eta_\infty(M \setminus M_n)$.
Therefore,  for each $j=1, \dots, m$,  we can choose a subarc $\beta_{e_j}$ of $\alpha_{e_j}$ tending to $e_j$ in such a way that  $\beta_{e_j}$ and $\beta_{e_{j'}}$ do not pass through the 
same bricks of $\eta_\infty(M)$ if $j \neq j'$.
If we take a sufficiently small $\delta>0$, then each $\beta_{e_j}$ passes through the $\delta$-region $S\times [c-\delta, c)\cup S \times (c,c+\delta]$ transversely to the horizontal leaves.
It follows that $m\leq -2\chi(S)$ since there are at most $-\chi(S)$ ends lying on $S \times \{c\}$ in each of $S \times [c-\delta, c)$ and $S \times (c, c+\delta]$.
Since $T_\infty'$ is a countable set as was seen before, this implies that  the ends of $\eta_\infty(M)$ are countable.
This completes the proof of the part (ii).
\end{proof}

\subsection{Conditions on labelled brick manifolds}\label{SS_Conditions}

%Consider a brick manifold $M$ with brick decomposition $\ck=\{B_i\}_{i=0}^\infty$.
%Let $A$ be a half open annulus properly embedded in $M$ such that  a core of $A$ is not contractible 
%in $M$.
%We say that $A$ \emph{exits} an end $e$ of $M$ if there exists a sequence of half-open annuli $A_n$ properly embedded in $M$ with 
%$$A=A_0\supset A_1\supset A_2\supset\cdots$$
%and $A_n\subset E_n(e)$, where $E_n(e)$ is the component of $M\setminus \bigcup_{k=1}^n B_k$ adjacent to $e$.
%The $A$ is \emph{eventually essential} in $M$ if a core of $A_n$ is not homotopic to a loop in $\part M$ in $E_n(e)$ 
%for some $n\in \nn$.
%A wild end $e$ of $M$ is \emph{accessible} if there exists an eventually essential half-open annulus in $M$ exiting $e$.  
%We say that $M$ is \emph{acylindrical} if there exists no properly embedded incompressible annulus in $M$ connecting distinct components of $\part M$. 

%
%Here we consider a brick manifold $M$ satisfying the following conditions.
%\begin{enumerate}[({2}.a)]
%\item
%$M$ is acylindrical and contains no accessible wild ends.
%\item
%Each component of $\part M$ is either a torus or an open annulus.
%\item
%Each brick $B$ in $\ck$ can be embedded in $S\times (0,1)$ as a sub-bundle. 
%\item
%For the ideal front $F$ of any half open brick $B$ of $\ck$, either $\Int F$ is equipped with  
%a conformal structure $\sg(B)\in \mathrm{Teich}(\Int F)$ or $F$ has a filling lamination $\mu(B)\in \mathcal{EL}(F)$.
%Moreover, in the former case, the real front of $B$ is an inessential front, see Fig.\ \ref{fig2_4}\,(a).
%\end{enumerate}

A labelled brick manifold is a brick manifold $M$ in which every half-open brick has either a point in the Teichm\"{u}ller space or an ending lamination  attached to it  as follows.
Let $B$ be a half-open brick in $M$ which is homeomorphic to $F \times J$, where $J$ is either $[0,1)$ or $(0,1]$.
Half-open bricks are divided into two categories: geometrically finite bricks and simply degenerate bricks.
If $B$ is geometrically finite, then a point in $\teich(\Int F)$ is given to $B$, 
 otherwise an \emph{ending lamination} of $B$, which is contained in $\EL(F)$ is given.
For a geometrically finite brick $B$, the interior of the ideal front of $B$ is 
denoted by $\part_\infty B$, and the point in $\teich(\Int F)$ is regarded as a marked conformal structure on $\partial_\infty B$.
Also for a simply degenerate brick, the given ending lamination is regarded as attached to the end corresponding to its ideal front.
%It is well known that there exists a constant $L>0$ depending only on the topological type of $S$ such that 
%$\part_\infty B(\sg(B))$ admits a pants decomposition $\boldsymbol{r}(\sg(B))= 
%r_1\cup\cdots\cup r_{m}$ on $\part_\infty B$ with $\mathrm{length}_{\sg(B)}(r_k)<L$ for 
%$k=1,\dots,m$.
%Let $\boldsymbol{s}(B)$ be the paths decomposition on the real front of $B$ which is parallel to 
%$\boldsymbol{r}(B)$ in $B$, see Fig.\ \ref{fig2_4}\,(a) again.

As in Theorem A, we shall consider  labelled brick manifolds $M$ satisfying the following conditions.
\begin{enumerate}[{A}-(1)]
\item Every component of $\partial M$ is either a torus or an open annulus.
\item
There is no properly embedded essential annulus whose boundary components lie in distinct boundary components of $M$.
\item
If there is an embedded, incompressible half-open annulus $S^1 \times [0,\infty)$ in $M$ such that $S^1 \times \{t\}$ tends to a wild end $e$, then its core curve is homotopic into an open-annulus component of $\partial M$ tending to $e$.
\item $M$ is embedded into $S \times (0,1)$ preserving the horizontal and the vertical leaves in such a way that the ends of geometrically finite bricks are peripheral.
\item
Every geometrically finite half-open brick has real front which is an inessential joint: \ie its real front is contained in the intersection with other bricks.

\end{enumerate}

We shall explain the meanings of  these conditions briefly.
We consider a model manifold $M$ of a geometric limit of Kleinian surface groups, whose corresponding hyperbolic 3-manifold we denote by $N$.
The boundary of $M$ corresponds to the frontier of the non-cuspidal part $N_0$.
This shows that the condition A-(1) must be satisfied.
Moreover by Margulis's lemma, no essential loops on two distinct components of $\Fr N_0$ can be homotopic to each other.
This implies the condition A-(2).

To illustrate the meaning of the condition A-(3), we consider the situation where $M$ is embedded in $S \times (0,1)$ preserving the horizontal and vertical leaves, which is required by A-(4).
A-(3) says that if $M$ has a wild end $e$, there must be a sequence of the complementary components of $M$ in $S \times (0,1)$ which tends to the image of $e$ in $S \times (0,1)$ in such a way that no closed curves can be homotoped to $e$ without obstructed by the complementary components except those lying on an annulus boundary component tending to $e$.
We note that model manifolds of  Kleinian surface groups (isomorphic to $\pi_1(S)$) constructed by Minsky can be regarded as labelled brick manifolds as will be explained later.
Such brick manifolds can be embedded in $S \times (0,1)$ preserving the horizontal and the vertical leaves.
Lemma \ref{l_21} implies that  model manifolds of geometric limits can also be embedded in $S \times (0,1)$ preserving the horizontal and the vertical leaves in such a way the geometrically finite ends are peripheral, which implies the condition A-(4).

%wild end ??????????A-(4)?????????????D

The last condition A-(5) is just for convenience in defining a metric on a brick manifold later.
\subsection{Tight tube unions}\label{B_d}
To construct model manifolds of Kleinian surface groups, Minsky considered a hierarchy of tight geodesics.
In his construction, a tight geodesic is realised in the model manifold as a sequence of Margulis tubes.
We shall consider a similar realisation of a tight geodesic in the model manifold, which we call a tight tube union.
%In the present paper (Part I), we shall work only within the model manifold, not considering hierarchies of geodesics for geometric limits as separate entities.
%We plan to deal with hierarchies for geometric limits in Part II of this series.

Consider a brick $B=F\times [0,1]$ with $\xi(F)>4$.
Suppose that we are given a pair of multi-curves  $\boldsymbol{I}\times \{0\}$ and $\boldsymbol{T}\times \{1\}$ lying on  $\Int \part_- B$ and $\Int \part_+ B$, which 
represent simplices in  $\cc(\Int F)$ by identifying $\partial_-B$ and $\partial_+B$ with $F$ naturally.
Let $g=\{v_i\}_{i=0}^n$ be a tight geodesic in $\cc(\Int F)$ with $I(g)=\boldsymbol{I}$ and 
$T(g)=\boldsymbol{T}$. 
Then $\bigcup_{i=0}^n v_i\times [i/(n+1),(i+1)/(n+1)]$ is a disjoint union $\ca_B$ of vertical annuli in $B$.
We call the union $\ca_B$ a \emph{tight annulus union} in $B$ connecting $\boldsymbol{I}\times \{0\}$ with 
$\boldsymbol{T}\times \{1\}$.

Next we consider the case when $B$ is a half-open brick $F\times [0,1)$ with $\xi(F)>4$.
Since we are not going to put an annulus union or a tube union for geometrically finite bricks, we assume that $B$ is simply degenerate.
Suppose then that $\boldsymbol{I}\times \{0\}$ is a multi-curve on $\Int \partial_-B=\Int F$, and that
$\boldsymbol{T}\times \{1\}$ is an element of $\mathcal{EL}(\Int \part_+ B)=\mathcal{EL}(\Int F)$, which is the ending lamination of $B$.
Let $g=\{v_i\}_{i=0}^\infty$ be a tight geodesic ray in $\cc(\Int F)$ with $I(g)=\boldsymbol{I}$ and 
$T(g)=\boldsymbol{T}$.
Then the union $\ca_B=\bigcup_{i=0}^\infty v_i\times [1-1/2^i,1-1/2^{i+1}]$ of vertical annuli in $B$ is called 
a \emph{tight annulus union} in $B$ connecting $\boldsymbol{I}\times \{0\}$ with $\boldsymbol{T}\times \{1\}$.
We can consider a similar construction for a half-open brick $F \times (0,1]$ when an ending lamination on $\Int \partial_+B$ and a multi-curve on $\Int \partial_+B$ are given, and define $\ca_B=\bigcup_{i=0}^\infty v_i\times [1/2^{i+1},1/2^i]$.

When $\xi(F)=4$, we need to modify our definition above to make annuli pairwise disjoint.
In this case, we define a tight annulus union $\ca_B$ by $\bigcup_{i=0}^n v_i\times [i/(n+1),(2i+1)/(2n+2)]$ if $B=F\times [0,1]$, by $\bigcup_{i=0}^\infty v_i\times [1-1/2^i,1-3/2^{i+2}]$ if $B=F\times [0,1)$, and $\ca_B=\bigcup_{i=0}^\infty v_i\times [ 3/2^{i+2},1/2^i]$ if $B=F\times (0,1]$.

Let $\ca_B=\bigcup_i v_i\times J_i$ be a tight annulus union in a brick $B$.
Take a sufficiently thin annular neighbourhood $R_i$ of $v_i$ on $F$ so that $R_i\times J_i$ are pairwise 
disjoint in $B$.
Then $\cv_B=\bigcup_i R_i\times J_i$ is called a \emph{tight tube union} in $B$ connecting $\boldsymbol{I}\times \{0\}$ with $\boldsymbol{T}\times \{1\}$.

\subsection{Block decompositions of labelled brick manifolds}\label{SS_block}

In this subsection, we shall show that a labelled brick manifold $M$ admits a decomposition into blocks in the sense of Minsky provided that its 
brick decomposition $\ck$ satisfies the conditions A-(1)-(5) and the following additional condition (EL), 
which corresponds to the 
assumption on ending laminations of  simply degenerate ends of $M$ given in Theorem \ref{thm_c}.

\bigskip

\noindent{\bf (EL)}
For any two simply degenerate bricks $B,B'$ in $\ck$, their ending laminations $\mu(B)$ and $\mu(B')$ are 
not homotopic in $M$.

\bigskip

 Under the conditions A-(1)-(5), this condition is automatically satisfied unless $M$ is  homeomorphic to $F\times (0,1)$ 
for a compact essential subsurface $F$ of $S$ as we can see in the following way.
Let $B_1$ and $B_2$ be two simply degenerate bricks with $B_1 =F_1 \times J_1$ and $B_2=F_2 \times J_2$, where $J_1$ and $J_2$ are half-open intervals.
Note that each  component of $\partial_\v B_1$ and $\partial_\v B_2$ lies in $\partial M$.
The condition A-(2) shows that $F_1 \times \{t\}$ and $F_2 \times \{t'\}$ cannot be homotopic in $M$ unless $M$ is homeomorphic to $F_1 \times (0,1)$.
Since $\mu(B_1)$ is contained in $\EL(\Int F_1)$ whereas $\mu(B_2)$ lies in $\EL(\Int F_2)$, which means that they are filling on non-homotopic surfaces, they cannot be homotopic in $M$ unless $F_1$ and $F_2$ are homotopic in $M$.
Therefore $M$ must be homeomorphic to $F_1 \times (0,1)$ if $B_1$ and $B_2$ have homotopic ending laminations.

\medskip

Let $\ck_{\mathrm{gf}}$ be the subset of $\ck$ consisting of geometrically finite bricks, 
and set $\ck_{\mathrm{int}}=\ck\setminus \ck_{\mathrm{gf}}$.
The union $\part_\infty M=\bigcup_{B\in \ck_{\mathrm{gf}}}\part_\infty B$ is called \emph{the boundary at infinity} of $M$ .
Bricks contained in  $\ck_{\mathrm{int}}$ are called \emph{internal bricks}. 

We modify a brick decomposition by performing the following two operations.
\begin{enumerate}
\item {\bf Removing inessential joints:}
Suppose that there is an inessential joint $F$ of two bricks $B,B'$ in $\ck_{\mathrm{int}}$.
Then 
 we replace  $B,B'$ with the single brick $B\cup B'$.
 In the exceptional case when $M$ is homeomorphic to $F \times (0,1)$ and has two simply degenerate bricks, this may generate a \lq \lq brick" homeomorphic to $F \times (0,1)$, which was not allowed in our definition.
 We still allow this operation and call thus obtained brick an {\em open brick}.
 \item {\bf Splitting bricks with non-overlapping annuli on the boundary:}
 Suppose that there is a brick $B=F \times [0,1]$ in $\ck_{\mathrm{int}}$ with a component $A$ of $\partial M \cap \partial_-B$ which does not overlap $\partial M \cap \partial_+B$.
  Here an annulus $A_1$ in $B$ is said to {\em overlap} a union of annuli $\mathcal A$ in $B$ when the vertical projections  of $A_1$ and $\mathcal A$ to $F$ intersect essentially.
 %overlap???????????t?????`?????D
Then we remove $A \times [0,1)$ from $B$ and split $B$ into two bricks $B_1, B_2$.
We can naturally identify $M$ with $M \setminus A \times [0,1)$ and regard $(\ck \setminus \{B\}) \cup \{B_1, B_2\}$ as a new brick decomposition of $M$.
We can perform the same operation also when there is an annulus in $\partial M \cap \partial_+B$ which does not overlap $\partial M \cap \partial_-B$.
\end{enumerate}
By repeating these two kinds of operations, we can assume 

\begin{assumption}
\label{modification}
(1)
that there is no inessential joint for two bricks in $\ck$, \\
(2)
and that for any brick $B$ both of whose fronts $\partial_- B$ and $\partial_+B$ are real, each component of $\partial_- B \cap \partial M$ overlaps  $\partial_+ B \cap \partial M$ and each component of $\partial_+ B \cap \partial M$ overlaps $\partial_- B \cap \partial M$.
\end{assumption}

By the condition A-(1), $\part M$ is a union of tori and open annuli.
Since $M$ is a brick manifold, each of such tori and annuli consists of horizontal annuli and vertical annuli whose interiors are pairwise disjoint, and contains at least one horizontal annulus except for the case when it is a vertical annulus corresponding to a component of $\partial S \times (0,1)$.
Let $H_A$ be the union of core curves of the horizontal annuli constituting the boundary components of $M$.
(We take one core curve from each horizontal annulus.)
For each geometrically finite brick $B_i$, we fix a multi-curve $s(B_i)$ on its real front $F_i$ which is the shortest pants decomposition of $F_i$ with respect to the hyperbolic structure given to $B_i$.
Note that although we gave a conformal structure on the ideal front, we put the pants decomposition on the real front.
Let $\boldsymbol{l}(\ck)$ be the union of $H_A$, the $s(B_i)$ for the geometrically finite bricks $B_i$,  and the ending laminations 
$\mu(B_j)$ for all simply degenerate brick $B_j$ in $\ck_{\mathrm{int}}$, which we regard as lying on the ideal fronts.
See Figure\ \ref{fig2_5}\,(a).
%%%%%%%%%%%%%%%%%%%%%%%%%%%%%%%%%%%%%%%%%%%%%%%%%%%%%%%
\begin{figure}[hbtp]
\centering
\scalebox{0.5}{\includegraphics[clip]{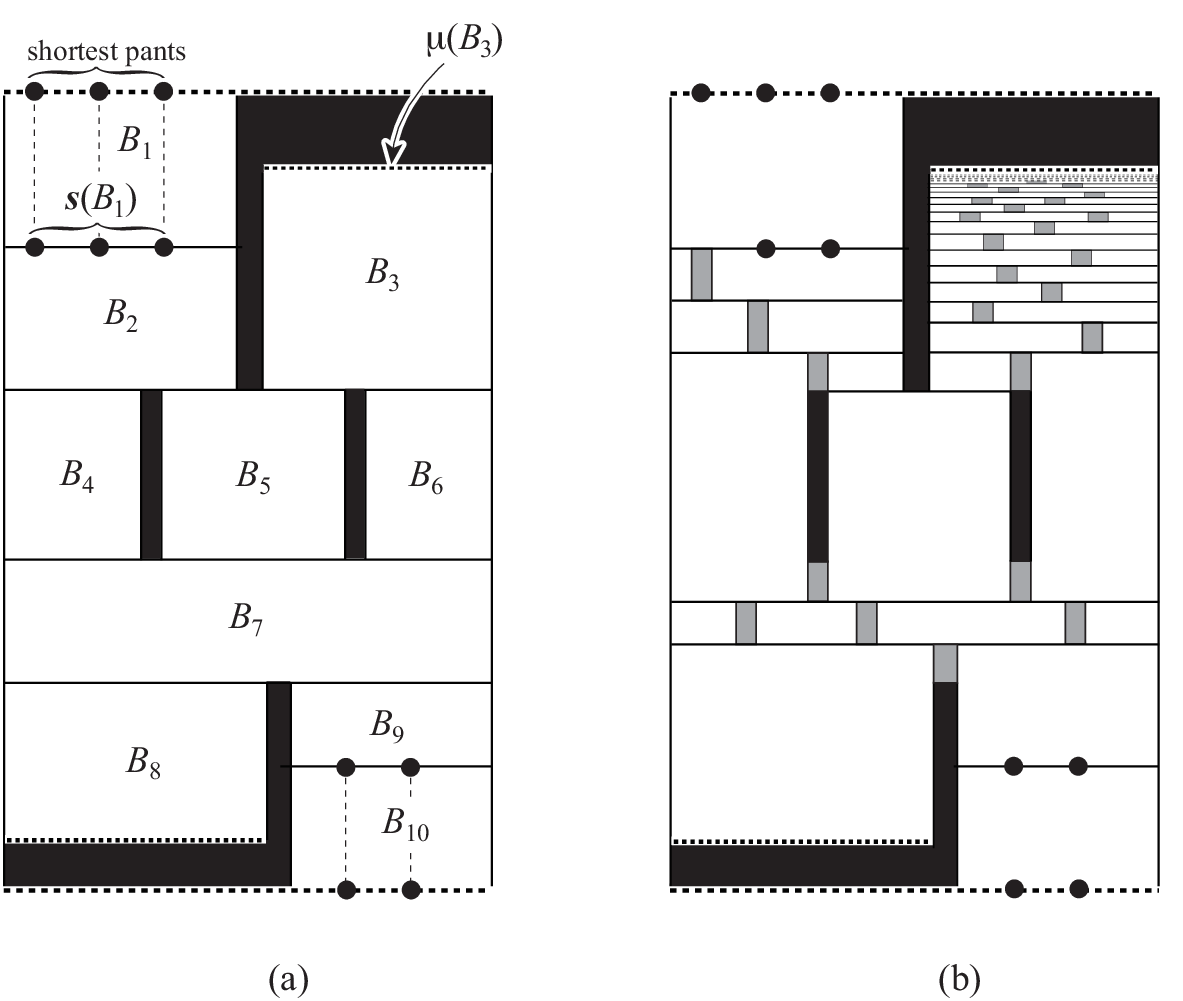}}
\caption{(a) $B_1,B_{10}$ is geometrically finite and $B_3,B_8$ are simply degenerate.
The real fronts of $B_1$ and $B_{10}$ are inessential joints.
$B_2,B_3,B_7$ are connectable.
(b) The union of shaded rectangles represents $\cv^{(1)}$.
The white rectangles are bricks in $\ck_{\mathrm{int}}^{(1)}\cup \ck_{\mathrm{gf}}$.
}
\label{fig2_5}
\end{figure}
%r(B_1)???????D
%%%%%%%%%%%%%%%%%%%%%%%%%%%%%%%%%%%%%%%%%%%%%%%%%%%%%%%

We set $M_{\mathrm{int}}=\bigvee \ck_{\mathrm{int}}$.
A brick $B$ in $\ck_{\mathrm{int}}$ is said to be \emph{connectable} if 
neither $\boldsymbol{I}(B)=\part_- B\cap \boldsymbol{l}(\ck)$ nor $\boldsymbol{T}(B)=\part_+ B\cap \boldsymbol{l}(\ck)$ is empty.
Notice that if $B$ is a simply degenerate brick, although $\mu(B)$ does not lie inside $M$, either $\partial_-B$ or $\partial_+B$ intersects $\mu(B)$.
It should be also noted that any brick $B$ in $\ck_{\mathrm{int}}$ that has greatest $\xi(B)$ among the bricks in $\ck_{\mathrm{int}}$ 
is connectable unless $\xi(B)= 3$ since we removed inessential joints.
We denote by $\xi_0$  the greatest $\xi(B)$.

For any connectable brick $B$ of $\ck_{\mathrm{int}}$ with $\xi(B)\geq 5$, we take a tight 
tube union in $B$ connecting $\boldsymbol{I}(B)$ with $\boldsymbol{T}(B)$, and denote it by $\cv_B$.
In the case when $B$ is an open brick, the condition (EL) guarantees that there is a tight tube union connecting $\boldsymbol{I}(B)$ and $\boldsymbol{T}(B)$.
We set $\cv_B=\eset$ if either $B$ is not connectable or $\xi(B)\leq 4$, and 
define $\bar \cv^{(1)}=\bigcup_{B\in \ck_{\mathrm{int}}}\cv_B$.
See Figure\ \ref{fig2_5}\,(b).
Now, if there are two tubes $T_1, T_2$ in $\bar\cv^{(1)}$ which are homotopic in $M \setminus (\bar\cv^{(1)} \setminus (T_1 \cup T_2))$ we merge them into one tube: we can assume that they are vertically isotopic, and by putting a tube between them which is also a thicken annulus, we can make them parts of a larger tube.
Repeating this operation, we get a union of tubes $\cv^{(1)}$ in which no two tori are homotopic in the complement of the rest of the tubes.

Let $M_{\mathrm{int}}^{(1)}$ be the closure of $M_{\mathrm{int}}\setminus \cv^{(1)}$ in $M_{\mathrm{int}}$.
Since $\cv^{(1)}$ consists of tubes which are thicken vertical annuli in $M_{\mathrm{int}}$, 
the $3$-manifold $M_{\mathrm{int}}^{(1)}$ has the local product structure induced from that on $M_{\mathrm{int}}$.
Thus $M_{\mathrm{int}}^{(1)}$ has a brick decomposition $\ck_{\mathrm{int}}^{(1)}$ allowing a brick also to be an open one having a form $F \times (0,1)$  such that each brick is the closure of a maximal union of vertically parallel horizontal leaves in $M_{\mathrm{int}}^{(1)}$.
By our operation modifying $\bar \cv^{(1)}$ to $\cv^{(1)}$, the condition A-(2) for $M$, and the fact that two simplifies on a geodesic at the distance $2$ have essential intersection, the same condition A-(2) holds also for $M_{\mathrm{int}}^{(1)}$.
%By the condition (2.a) on $M$, $M$ has no accessible wild ends.

Let $B$ be a half-open  or open brick in $M_{\mathrm{int}}^{(1)}$.
Suppose that $B$ meets infinitely many original internal bricks $\hat B_p$ of $\ck_{\mathrm{int}}$. 
Then we can take an essential simple closed curve on the horizontal surface of $B$ which is not  homotopic into an annulus component of $\partial M$, and is vertically isotopic into each of the $\hat B_p$.
This gives rise to an incompressible half-open annulus with core curve not homotopic into  an annulus component of $\partial M$, which tends to a wild end of $M$ to which  the $\hat B_p$ tend, contradicting the condition A-(3) for $M$.
(This end cannot be simply degenerate since each simply degenerate end is contained in one brick of $\ck$.)
Therefore, any brick in $\ck_{\mathrm{int}}^{(1)}$ meets only finitely many bricks of $\ck_{\mathrm{int}}$.
Also, we can see that an ideal front $F$ of $B$ cannot be contained in  the ideal front $F'$ of some simply degenerate brick $B'=F' \times J$ of $\ck_\mathrm{int}$ since $\mu(B')$ is contained in $\EL(F')$, and hence there is no open annulus in $B'$ disjoint from the tight union of tubes which we extracted to construct $M_{\mathrm{int}}^{(1)}$.
Thus we have shown that $M_{\mathrm{int}}^{(1)}$ contains neither half-open nor open bricks.
%If $\ck_{\mathrm{int}}$ had an element $B'$ breaking this condition, then $B'$ would be a half-open brick 
%the ideal front of which does not have  any end invariants inherited from $M$.
%This fact suggests that the condition (2.a) is crucial to proceed with our process.
We should note that the greatest $\xi(B)$ for the bricks $B$ in $M_{\mathrm{int}}^{(1)}$, which we denote by $\xi_1$, is less than $\xi_0$ since bricks in $M_\mathrm{int}$ with $\xi=\xi_0$ are all connectable.

Next we consider the union $\cv^{(2)}$ of tubes which we obtained by modifying the union of all tight tube unions $\cv_B$ for all $B\in \ck_{\mathrm{int}}^{(1)}$ in the same way as  we defined $\cv^{(1)}$ in $\ck$ merging homotopic tubes, and the closure $M_{\mathrm{int}}^{(2)}$ of $M_{\mathrm{int}}^{(1)} 
\setminus \cv^{(2)}$ in $M_{\mathrm{int}}^{(1)}$.
By the same reason as before, the greatest $\xi(B)$ for the bricks $B$ in $M_\mathrm{int}^{(2)}$ is less than $\xi_1$.
Therefore, repeating the same procedure  at most $\xi(S)-4$ times, we reach a brick decomposition 
$\ck_{\mathrm{int}}^{(k)}$ on $M_{\mathrm{int}}^{(k)}$ such that $\xi(B)$ is 
either $3$ or $4$ for every brick $B\in \ck_{\mathrm{int}}^{(k)}$.

Let $\cv^{(k+1)}$ be the union of tubes obtained by modifying in the same way as before the union of tight tube unions $\cv_B$ for bricks $B\in \ck_{\mathrm{int}}^{(k)}$ with 
$\xi(B)=4$, and let $\ck_{\mathrm{int}}^{(k+1)}$ be the brick decomposition on the closure 
$M_{\mathrm{int}}^{(k+1)}$ of $M_{\mathrm{int}}^{(k)}\setminus \cv^{(k+1)}$ such that each brick is a maximal union of  parallel leaves with respect to the horizontal foliation on $M^{(k+1)}_\mathrm{int}$.
Moving components of $\cv^{(k+1)}$ vertically by an ambient isotopy of $M_{\mathrm{int}}^{(k)}$ if necessary, we can assume that for every brick $B$ of $\ck_{\mathrm{int}}$, its fronts $\partial_\pm B$  does not go though the gaps of tubes of $\cv^{(k+1)}$, \ie the following holds.
\bigskip

\noindent{\bf (BB)}
For any $B\in \ck_{\mathrm{int}}$ and $B'\in \ck_{\mathrm{int}}^{(k)}$ with $H=(\part_+B\cup \part_-B)\cap 
B'\neq \eset$, each component of $H\setminus \Int \cv_{B'}$ is homeomorphic to $\Sg_{0,3}$.

\bigskip

We set $\cb_{\mathrm{int}}=\ck_{\mathrm{int}}^{(k+1)}$, $\cb=\ck_{\mathrm{int}}^{(k+1)}\cup \ck_{\mathrm{gf}}$, 
$M[0]_{\mathrm{int}}=M_{\mathrm{int}}^{(k+1)}$, $M[0]=M[0]_{\mathrm{int}}\cup (\bigvee \ck_{\mathrm{gf}})$, and
$\cv=\bigcup_{m=1}^{k+1} \cv^{(m)}$.
We call $\cb$  a \emph{block decomposition} of $M[0]$ and each element of $\cb$ a \emph{block}.
Note that each block in $\cb_{\mathrm{int}}$ is homeomorphic to either $\Sg_{0,3}\times J$ or $\Sg_{1,1}\times J$  
or $\Sg_{0,4}\times J$, where $J$ is a closed or half-open or open interval, since every brick in $M_\mathrm{int}^{(k+1)}$ has $\xi$ at most $4$.
Also by our definition of bricks for $M_{\mathrm{int}}^{(k+1)}$ no two blocks meet at inessential joint.

\begin{remark}
\label{two ways of block}
It may appear that our definition of blocks is slightly different from that of Minsky in \cite{mi2} as we allow blocks homeomorphic to $\Sg_{0,3} \times J$.
Still the difference is just a minor point since we can convert our block decomposition into that \`{a} la Minsky just by cutting  a block of the form $\Sg_{0,3} \times J$ into halves and paste one of them to the block above it and the other to the one below it.
\end{remark}

Each component of $\cv$ is a  solid torus which is foliated by vertically parallel horizontal annuli.
For each solid torus $V$ in $\cv$, its boundary $\partial V$ is contained in $\partial M[0] \cup \partial M$.
If $M[0]\cap V$ consists of two vertical annuli $A_1,A_2$ for some $V\in \cv$, then 
$\part V\setminus \Int (A_1\cup A_2)$ is a union of two horizontal annuli contained in $\part M$, and 
hence each of $A_1,A_2$ is a properly embedded essential annulus in $M$.
(These annuli cannot be boundary-parallel since a brick is not allowed to be a solid torus by definition.)
This contradicts the condition A-(2) saying that $M$ must be acylindrical.
Therefore, for any component $V$ of $\cv$, the intersection $M[0]\cap V$ is either a torus or an annulus.
 See Figure\ \ref{fig2_6}. 
%%%%%%%%%%%%%%%%%%%%%%%%%%%%%%%%%%%%%%%%%%%%%%%%%%%%%%%
\begin{figure}[hbtp]
\centering
\scalebox{0.5}{\includegraphics[clip]{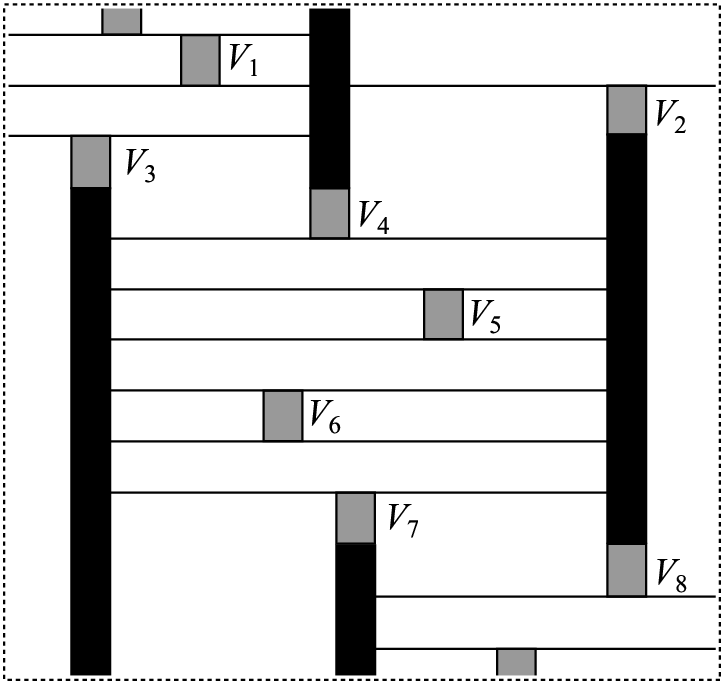}}
\caption{A local picture of $M$ in the case of $k=0$.
The white region is $M[0]$.
$V_1\cup V_5\cup V_6\subset \cv[0]$ and $V_2\cup V_3\cup V_4\cup V_7\cup V_8\subset \cv\setminus \cv[0]$.}
\label{fig2_6}
\end{figure}
%%%%%%%%%%%%%%%%%%%%%%%%%%%%%%%%%%%%%%%%%%%%%%%%%%%%%%%

Let $\cv[0]$ be the union of all components $V$ of $\cv$ such that $M[0]\cap V$ is a 
torus, and set $M^0 =M[0]\cup \cv[0]$.
Then $M^0$ is obviously a deformation retract of $M$ and  there exists a homeomorphism $\eta_M:M^0\rightarrow M$  homotopic to the inclusion such that 
the restriction $\eta_M|_{\cv[0]}$ is the identity.
We often identify the original brick manifold $M$ with $M^0$ via the map $\eta_M$.
%, which is 
%called an \emph{identification} of $M$ and $M^0$.

\subsection{Model metrics on brick manifolds}\label{SS_block_metric}
Now we shall define a metric on a brick manifold induced from its decomposition into blocks.
We shall put a standard metric on each block as was done in  Minsky \cite{mi2}, which is slightly different from his for our convenience.
Fix  $\ve_1>0$ less than the three-dimensional Margulis constant, and a hyperbolic metric on the three-holed sphere $\Sg_{0,3}$ with respect to which  each component of $\part \Sg_{0,3}$ is a closed geodesic of length $\ve_1$.
Let $B_{0,3}$ be $\Sg_{0,3}\times [0,1]$ endowed with the product metric of the hyperbolic metric on $\Sg_{0,3}$ and the standard metric on $[0,1]$.

Consider two essential simple closed curves $l_0,l_1$ on $\Sg_{0,4}$ (resp.\ $\Sg_{1,1}$) with 
the geometric intersection number $i(l_0,l_1)=2$ (resp.\ $i(l_0,l_1)=1$) and set $B_{\alpha}$ to be a brick in the form $\Sg_{\alpha}\times [0,1]$ for $\alpha\in \{(0,4),(1,1)\}$.
Let $A_-$ and $A_+$ be annular neighbourhoods of $l_0 \times \{0\}$ and $l_1 \times \{1\}$ in $\partial_- B_\alpha$ and $\partial_+ B_\alpha$ respectively.
We define a piecewise Riemannian metric on $B_\alpha$ such that each component of $\partial_-B_\alpha\setminus  \Int A_- $ and $\partial_+B_\alpha\setminus \Int A_+$ is isometric to $\Sg_{0,3}$ with the hyperbolic metric given above, all of $A_-,A_+$ and $\part_{\v}B_\alpha$ 
are isometric to the product annulus $S^1(\ve_1)\times [0,1]$ and 
$\dist_{B_{\alpha}}(\part_-B_{\alpha},\part_+ B_{\alpha})=1$, where $S^1(\ve_1)$ is a round circle in the Euclidean plane of circumference $\ve_1$.

For any brick $B\in \cb_{\mathrm{int}}$ of type $\beta\in \{(0,3),(0,4),(1,1)\}$, consider a diffeomorphism  
$h_B:B_\beta\rightarrow B$ such that $h_B(\part_{\v}B_{\beta})=\part_{\v}B$ and 
moreover $h_B(A_\pm)=\part_\pm B\cap \cv[0]$ when $\xi(B)=4$.
We can choose these homeomorphisms so that  for any $B,B'$ of types $\beta, \beta' \in \{(0,3),(0,4), (1,1)\}$ in $\cb_{\mathrm{int}}$ with 
$F=\part_+ B\cap \part_- {B'}\neq \eset$, $((h_{B'}|_{F})^{-1}\circ h_B|_{h_B^{-1}(F)}) $ is 
an isometry with respect to the metrics on $B_\beta$ and $B_{\beta'}$ defined above.
Then $M[0]_{\mathrm{int}}$ has a piecewise Riemannian metric induced from those on $B_{0,3},B_{0,4},B_{1,1}$ via 
embeddings $h_B:B_\beta \rightarrow M[0]_{\mathrm{int}}$.

We shall next define metrics on  geometrically finite bricks. 
Each geometrically finite brick $B$ of $\cb$ is identified with $F\times [-1,\infty)$ preserving the horizontal and the vertical leaves  for a 
compact core $F$ of some open  essential subsurface $\mathring F$ of $S$ with $\xi(F) \geq 3$.
Since $\mathring F$ can be identified with $\Int F$, by our definition of geometrically finite bricks,  $\mathring F=\mathring F\times \{\infty\}$ is given a conformal structure.
Let $\sigma(B)$ a complete hyperbolic metric on $\mathring F$ which is compatible with the given conformal structure.
We regard $F$ as obtained from $\mathring F(\sg(B))$ by deleting the cusp neighbourhoods which are components of
$\mathring F(\sg)_{(0,\ve_1)}$.
Consider a piecewise Riemannian metric $\tau(B)$ on $\mathring F$ obtained by rescaling $\sg(B)$ on the points of $\mathring F$ in such a way that $\tau(B)/\sg(B)$ is  continuous and is equal 
to $1$ on $\mathring F(\sg(B))_{[\ve_1,\infty)}$, and each component of $\mathring F(\sg(B))_{(0,\ve_1]}$ is a Euclidean cylinder with 
respect to the $\tau(B)$-metric.
On the other hand, we put another piecewise Riemannian metric $\upsilon(B)$ on $F$ such that 
each component of $F(\upsilon(B))_{(0,\ve_1]}$ is a Euclidean cylinder, 
$F(\upsilon(B))_{(0,\ve_1]}\times \{-1\}$ coincides with $\part M[0]_{\mathrm{int}}\cap B$,
and each component of $F(\upsilon(B))_{[\ve_1, \infty)}$ is isometric to $\Sg_{0,3}$. 
We choose such a metric so that the identity $F(\tau(B))
\rightarrow F(\upsilon(B))$ is uniformly bi-Lipschitz (\ie the bi-Lipschitz constant is bounded by a constant independent of $B$ and $F$).
We call a metric constructed  as the latter metric $\upsilon(B)$ a {\em cylinder-$\Sigma_{0,3}$} metric on $F \times \{-1\}$.
We note that  our $\upsilon(B)$ corresponds to the metric $\sg^{m'}$ given in \cite[Subsection 8.3]{mi2}.

We put  a piecewise Riemannian metric on $F \times [-1,0]$ such 
that  its restriction to $F\times \{-1\}$ is equal to $F(\upsilon(B))$, its restriction to $F\times \{0\}$ 
is equal to $F(\tau(B))$, and the induced metric on $F \times \{t\}$ is uniformly bi-Lipschitz to $\tau(B)$ via the identification of $F \times \{t\}$ with $F$.
Recall that $F$ is a compact core of an open surface $\mathring F$.
We take a diffeomorphism $\eta:F\times [0,\infty)\rightarrow \mathring F\times [0,\infty)$ such that 
the restriction $\eta|F\times \{0\}$ is the identity and $\eta(\partial F \times [0,\infty))$ lies on $\mathring F \times \{0\}$.
We put on $F\times [0,\infty)$  the induced metric $\eta^*(ds^2)$, where $ds^2$ is a 
piecewise Riemannian metric on $\mathring F\times [0,\infty)$ defined by 
$$
ds^2=\tau(B) e^{2r}+dr^2\qquad (r\in [0,\infty)).
$$
We define a piecewise Riemannian metric on $B$ by pasting the metrics on $F \times [-1,0]$ and $F \times [0,\infty)$ along $F \times \{0\}$, which has the metric $\tau(B)$ on both sides.
We may assume that the metric on $M[0]_{\mathrm{int}}$ and that on $B$ are equal on $M[0]_{\mathrm{int}}\cap B=F\times \{-1\}$ 
 deforming the map attaching $B$ to $M[0]_{\mathrm{int}}$ by an ambient isotopy if necessary.
Thus we have obtained a piecewise Riemannian metric on $M[0]$, which we call the {\em model metric} on $M[0]$.
By our construction, each component $C$ of $\part M[0]$ is either a Euclidean torus or a Euclidean cylinder which 
has a foliation $\cf_C$ whose leaves consist of closed geodesics  of length $\ve_1$.

\subsection{Meridian coefficients}
For a complex number $z$ with $\mathrm{Im}(z)>0$ and a real number $\eta>0$, we denote the covering map 
$\mathbf{C}\rightarrow \mathbf{C}/\eta(\zz+z\zz)$ by $\pi_{z,\eta}$. 
For any component $V$ of $\cv[0]$, its boundary $\part V$ has a Euclidean metric induced from 
the model metric on $M[0]$ as above.
Then there is a unique $\omega\in \mathbf{C}$ with 
$\mathrm{Im}(\omega)>0$ for which we have an 
orientation-preserving isometry from the quotient space $\mathbf{C}/\ve_1(\zz+\omega\zz)$ to $\part V$ taking $\pi_{\omega,\ve_1}(\rr)$ (resp.\ $\pi_{\omega,\ve_1}(\omega\rr)$) to a longitude (resp.\ a meridian) of $V$.
(Here a longitude of $V$ is defined to be a horizontal essential simple closed curve on $\partial V$.)
We denote this $\omega$ by $\omega_M(V)$ and call it the \emph{meridian coefficient} of $\part V$.

For any integer $k>0$, consider the union $\cv[k]$ of components $V$ of $\cv[0]$ with $|\omega_M(V)|\geq k$ and 
set
$$M[k]=M[0]\cup (\cv[0]\setminus \cv[k]).$$
By definition, we have $M^0 =M[k] \cup \cv[k]$.
We put each component $V$ of $\cv[0]$ a hyperbolic metric induced from the Margulis tube whose boundary has  exactly the Euclidean metric  induced from  the model metric on $M[0]$.
Here abusing the terminology, we always call hyperbolic equidistance neighbourhoods of simple closed geodesics Margulis tubes even when the lengths of the core curves are not less than the Margulis constant.
See Lemma 8.5  in \cite{bcm} or Lemma 5.8 in \cite{bow3}.
In this way, we extend the model metric metric on $M[0]$ to a metric on $M^0$ whose restriction on 
$\cv[0]$ is hyperbolic.
The brick manifold $M$ has the metric induced from that on $M^0$ via $\eta_M$.
We also call these metrics on $M^0$ and $M$ the model metrics.

\section{The bi-Lipschitz model theorem for brick manifolds}\label{S_BLM}
Minsky constructed  in \cite{mi2} model manifolds for hyperbolic $3$-manifolds homeomorphic to $S \times (0,1)$ and proved that for any such hyperbolic manifold, there is a Lipschitz map, called a model map, from its model manifold, whose Lipschitz constant is uniformly bounded.
Furthermore, in Brock-Canary-Minsky \cite{bcm}, it was shown that such a model map can be taken to be a bi-Lipschitz homeomorphism with uniformly bounded bi-Lipschitz constant.
Using and generalising these results, we shall show that a homeomorphism from a labelled brick manifold satisfying the conditions A-(1)--(5) and (EL) to a hyperbolic $3$-manifold preserving end invariants can be homotoped to a bi-Lipschitz homeomorphism with uniformly bounded bi-Lipschitz constant.
Let us recall that for any hyperbolic 3-manifold $N$ and a constant $\ve_1>0$ less than the Margulis constant, 
$N_0=N_0^{\ve_1}$ denotes the $\ve_1$-non-cuspidal part, \ie the union 
of $N_{[\ve_1,\infty)}$ and all Margulis tube components of $N_{(0,\ve_1]}$ as defined in Subsection \ref{hyp}.

\begin{theorem}[Bi-Lipschitz Model Theorem]\label{blm}
Let $M$ be a labelled brick manifold satisfying the conditions A-(1)--(5) and (EL), 
and  $N$   a hyperbolic $3$-manifold with  a homeomorphism $f:M\rightarrow N_0$ preserving 
the end invariants.
Then $f$ is properly homotopic to a homeomorphism $g:M\rightarrow N_0=N_0^{\ve_1}$ satisfying 
the following conditions, where $k\in \nn$, $K\geq 1$ and $\ve_1$ less than the Margulis constant depend only on $\xi(S)$.
\begin{enumerate}[\rm (i)]
\item
The image $g(\cv[k])=\mathrm{T}[k]$ is a union of  $\varepsilon_1$-Margulis tubes of $N_0$.
%\varepsilon_1?????????????????`?F?b?N?D
% and 
%$g(M[k])$ is contained in $N_{[\ve_2,\infty)}$ for some uniform constant $\ve_2$ with $0<\ve_2\leq \ve_1$.
\item$g(M[k])=N_0\setminus \Int \mathrm T[k]$.
\item
The restriction $g|_{M[k]}:M[k]\rightarrow N_0\setminus \Int \mathrm{T}[k]$ is a $K$-bi-Lipschitz map.
\item
The homeomorphism $g$ extends continuously to a conformal map from $\part_\infty M$ to $\part_\infty N$.
\end{enumerate}
\end{theorem}

The whole of the present section is devoted to the proof of Theorem \ref{blm}.
We should note that by Lemma \ref{l_21}, there is a proper embedding $\iota_M$ of our model manifold into $S \times (0,1)$.
Accordingly, we have an embedding $\iota_N: N_0 \rightarrow S \times (0,1)$ such that $\iota_N \circ f= \iota_M$.
As in the previous section, we modify the brick decomposition of $M$ so that Assumption \ref{modification} holds.

Applying the argument in Minsky \cite[Subsections 3.4 and 8.3]{mi2}, we can deform $f$ to a map $f_1$ by a 
proper homotopy so that  for any 
geometrically finite half-open brick $B'\in \ck_{\mathrm{gf}}$, the restriction $f_1|_{B'}:B'\rightarrow f_1(B')$ is a uniformly bi-Lipschitz homeomorphism 
which extends continuously to a conformal map from $\part_\infty B'$ to $\part_\infty f_1(B')$ and its real front is mapped into the boundary of the convex core of $N$.

We shall first show that $f_1$ can be properly homotoped to a  $K$-Lipschitz map with a constant $K$ depending only on $\xi(S)$.
For that, we shall follow the line of  Minsky's  argument in \cite{mi2}.
Recall that we have a union of tubes $\cv$ in $M$ which we constructed in \S\ref{SS_block} inducing a decomposition of  $M$ into blocks, and that for each tube $V$ in $\cv$, its meridian coefficient $\omega_M(V)$ is defined.
The first step is to prove the following lemma.

\begin{lemma}
\label{upper bound}
There is a universal constant $L$ depending only on $\xi(S)$ such that for the core curve $v$  of each tube $V$ in $\cv[0]$, the length of the closed geodesic in $N$ homotopic to $f(v)$ is less than $L$.
\end{lemma}
\begin{proof}
This lemma corresponds to Lemma 7.9 in Minsky \cite{mi2}.
We shall use its generalisation by Bowditch, Theorem 1.3 in \cite{bow2}.

Recall that we constructed a block decomposition of $M$ repeating the process of putting tight tube unions in bricks, starting from the decomposition of $M$ into bricks of $\mathcal{K}$.
At the first stage, for each connectable internal brick $B=F\times J$, we connected  a component $\partial_- B \cap \boldsymbol{l}(\mathcal{K})$ with a component of $\partial_+ B \cap \boldsymbol{l}(\mathcal{K})$ by a tight geodesic.
Since $f$ takes $\boldsymbol{l}(\mathcal{K})$  to either an ending lamination or a parabolic element in $N$, by applying Lemma 7.9 in Minsky \cite{mi2} or Theorem 1.3 in Bowditch \cite{bow2} to the covering of $N$ associated to $f_\# \pi_1(B)$, we see that there is a constant $L_0$ depending only on $\xi(S)$ such that each curve in the simplices constituting the tight geodesic has length in $N$ bounded by $L_0$.

At the $n$-th stage, we have bricks $\mathcal{K}^{(n)}_\mathrm{int}$ constituting $M_{\mathrm{int}}^{(n)}$ which is the complement of $\cv_n=\bigcup_{m=1}^n \cv^{(m)}$ in $M_{\mathrm{int}}$.
Let $\boldsymbol{l}_n$ be the union of $\boldsymbol{l}(\mathcal{K})$ and the core curves of $\cv_n$ that are not homotopic to a simple closed curve in $\boldsymbol{l}(\mathcal{K})$.
In each brick $B^{(n)}$ of $\mathcal{K}^{(n)}_\mathrm{int}$, we constructed a tight tube union connecting $\partial_- B^{(n)} \cap \boldsymbol{l}_{n}$ and $\partial_+ B \cap \boldsymbol{l}_n$.
Therefore using Bowditch's Theorem 1.3 inductively, we see that if the geodesic lengths in $N$ of curves in $\boldsymbol{l}_{n}$ are  bounded by $L_{n}$, then there is $L_{n+1}$ depending only on $L_n$ bounding the lengths in $N$ of $\boldsymbol{l}_{n+1}$.
Since we reached the block decomposition within $\xi(S)-3$ steps, we see that there is a constant $L$ depending only on $\xi(S)$ which bounds the lengths of the closed geodesics corresponding to the core curves of $\cv$.
\end{proof}

\subsection{Homotoping $f$ to a Lipschitz map preserving the thin part}
Moving $\cv$  by an ambient isotopy of $M_{\mathrm{int}}$ without changing the structure of block decomposition, we may assume that  for any $B\in \ck_{\mathrm{int}}$, every component of  $\part_+ B\setminus \cv$ and $\part_- B\setminus \cv$ is homeomorphic to a thrice-punctured sphere.
Let $F$ be a compact essential subsurface of $S$ such that $B$ is homeomorphic to $F \times J$ for an interval $J$.
If $\partial_+ B$ is a real front, then $\partial_+ B \cap \cv$ determines a simplex  in $\mathcal C(F)$ inducing a pants decomposition of $F$.
We now homotope $f_1$ so that each core curve of $\cv[0]$ is mapped to a closed geodesic.
By Lemma \ref{upper bound}, all of such closed geodesics have length bounded by $L$.
In this situation, we can apply Minsky's construction in Steps 0-6 of \cite[Section 10]{mi2} to get a map $f_2: M\rightarrow N$ for which the following hold.
Recall that we have fixed a constant $\varepsilon_1$ less than the three-dimensional Margulis constant.
\begin{enumerate}
\label{conditions}
\item
We have $f_2|B'=f_1|B'$ for every $B' \in \mathcal{K}_\mathrm{gf}$.
\item
For each block $B$ of $M[0]_\mathrm{int}$, the $f_2| \partial_\pm B$  lies on a pleated surface with totally geodesic boundary each of whose components is a closed geodesic homotopic to $f_2(v)$ for a core curve $v$ of some $V \in \cv$.
\item
There exists a constant $\varepsilon_0>0$ depending only on $\xi(S)$ such that  for a core curve $v$ of a solid torus component $V$ of $\cv$, if the geodesic length of $f_2(v)$ is less than $\varepsilon_0$, then $f_2(V)$ is contained in the $\varepsilon_1$-Margulis tube with core curve $f_2(v)$.
\item
The image of $f_2$ is contained in the union of the $1$-neighbourhood of the convex core of $N$ and the $\varepsilon_1$-Margulis tubes of $N$.
\item
For any $k$, there exists a positive number $\epsilon(k) < \varepsilon_1$ such that $f_2(M[k])$ is disjoint from the $\varepsilon_1$-Margulis tubes of $N$ whose core curves have length less than $\epsilon(k)$.
\item For any $k$, there exists a constant $L(k)$ such that $f_2|M[k]$ is $L(k)$-Lipschitz.
\end{enumerate}

To modify $f_2$ further to get a Lipschitz map, we need the following lemma.

\begin{lemma}
\label{large k}
Let $V$ be a tube in $\cv[0]$, and $v$ its core curve.
For any $\delta>0$, there exists $k$ which depends on $\delta$ and $\xi(S)$ but is independent of $M$ and $N$ such that if $|\omega_M(V)|> k$ then the closed geodesic homotopic to $f(v)$ has length smaller than $\delta$.
\end{lemma}
\begin{proof}
This lemma corresponds to Lemma 10.1 in Minsky \cite{mi2}.
In our situation, $V$ may be shared by blocks contained in distinct bricks.
Therefore, we cannot apply Minsky's result directly.
Instead, we use an argument which can also be found in Soma \cite{so}.
Our argument is by contradiction.
Suppose that there exist $\delta >0$ and tubes $V_j$ with core curves $v_j$ such that $|\omega_M(V_j)| \rightarrow \infty$ whereas the closed geodesics homotopic to $f_2(v_j)$ have length greater than $\delta$.

Since $|\omega_M(V_j)| \rightarrow \infty$, by passing to a subsequence, we can assume that either $\Im \omega_M(V_j) \rightarrow \infty$ or $\Re \omega_M(V_j) \rightarrow \infty$ holds.
We shall first consider the case when $\Im \omega_M(V_j) \rightarrow \infty$.
By the definition of $\omega_M(V_j)$, there are $(\Im \omega_M(V_j)-2)$ blocks which intersect $\partial V_j$ along their vertical sides.
This implies that there are at least $\Im \omega_M(V_j)$ gluing surfaces, which are homeomorphic to $\pants$, having boundary components lying on $\partial V_j$.
We should also note that no two distinct gluing surfaces are homotopic in $M$.
Since we assumed that $\Im \omega_M(V_j)$ goes to $\infty$, there are $k_j$ pairwise non-homotopic gluing surfaces with boundary components on $\partial V_j$ with $k_j \rightarrow \infty$.
The image of each gluing surface lies on a pleated surface with totally geodesic boundary one of whose components is the closed geodesic $\gamma_j$ homotopic to $f(v_j)$.
Therefore, there are $k_j$ pairwise non-homotopic pleated surfaces from $\pants$ which have $\gamma_j$ as a boundary component.

%Let $\epsilon_0$ be a constant less than the three-dimensional Margulis constant as before.
Now, we put a basepoint $x_j$ on $\gamma_j$, and consider the geometric limit $(N_\infty, x_\infty)$ of $(N, x_j)$, passing to a subsequence if necessary.
Since the length of $\gamma_j$ is bounded from above by Lemma \ref{upper bound} and from below by $\delta>0$ by our assumption,  the geometric limit exists (as a hyperbolic $3$-manifold) if we take a subsequence, and does not depend on the choice of $x_j$ as long as it lies on $\gamma_j$ once we fix some geometrically convergent subsequence.
Let $\rho_i : B_{R_i}(N, x_j) \rightarrow B_{K_i R_i}(N_\infty, x_\infty)$ be an approximate isometry associated to the geometric convergence with $R_i \rightarrow \infty$ and $K_i \rightarrow 1$.
In the geometric limit $N_\infty$, we have  the limit $\gamma_\infty$ of $\gamma_j$, which is a closed geodesic since  the lengths of the $\gamma_j$ are bounded away from $0$.
The geometric limit of pleated surfaces with boundary components on $\gamma_j$ are pleated surfaces with a boundary component on $\gamma_\infty$.
We should also note that if we fix a positive constant $\epsilon$ smaller than $\delta$ and $\varepsilon_1$, then  all the pleated surfaces intersect the $\epsilon$-thin part of $N$  only at near their boundary components other than $\gamma_j$, and hence that the limit pleated surfaces can intersect the $\epsilon$-thin part only near their boundary components other than $\gamma_\infty$.
Since $k_j \rightarrow \infty$, we can find among the limit pleated surfaces,  two limit pleated surfaces $F_1, F_2$ such that $F_2$ is homotopic to $F_1$ in a small regular neighbourhood $F_1$ whereas $\rho_i^{-1}(F_1)$ and $\rho_i^{-1}(F_2)$ are not homotopic.
This is a contradiction.

It remains to deal with the case when $\Re \omega_M(V_j) \rightarrow \infty$ whereas $\Im \omega_M(V_j)$ is bounded.
Fix a horizontal simple closed curve $c_i$ on $\partial V_j$.
We let $d_i$ be a simple closed curve on $\partial V_j$ intersecting $c_i$ at one point and having shortest length among all  simple closed curves intersecting $c_i$ at one point.
Let $m_j$ be a meridian of $V_j$.
Since $d_j$ intersects $c_j$ at one point, as elements of the first homology group of $\partial V_j$, we have $[d_j]=[m_j]+\alpha_j[c_j]$ with $\alpha_j \in \integers$ if we fix orientations on $c_j, m_j$ and $d_j$.
Since we assumed that $\Re \omega_M(V_j) \rightarrow \infty$, we have $|\alpha_j| \rightarrow \infty$, and in particular, we can assume that $\alpha_j \neq 0$ by taking a subsequence.
Since the length of $d_j$ is shortest among the simple closed curves intersecting $c_j$ at one point, we have $\length_{\partial V_j}(d_j) \leq (\Im\omega_M(V_j)+1)\epsilon_1$.
Now, since $\partial V_j$ is contained in $M[0]$, by the condition (6) above, we have $\length(f_2(d_j)) \leq L(0) (\Im\omega_M(V_j)+1)\epsilon_1$.
The right hand side is bounded above since we have already proved $\Im \omega_M(V_j)$ is bounded as $j \rightarrow \infty$.
Since $[d_j]=[m_j]+\alpha_j[c_j]$, the curve $f_2(d_j)$ with an appropriate orientation is homotopic to the $|\alpha_j|$-time iteration of $\gamma_j$ in $N$.
This implies $\length f_2(d_j) \geq |\alpha_j| \length(\gamma_j)$.
The right hand side goes to $\infty$, whereas  the left hand side is bounded as we have seen above.
This is a contradiction.
\end{proof}

Having proved Lemma \ref{large k}, the rest of modification as in Minsky \cite{mi2} to get {\em a proper, degree-$1$ map} $f_3: M \rightarrow N_0$ such that  $f_3|M[k]$ is $K_3$-Lipschitz with $K_3$ depending only on $\xi(S)$ works without changes.

We state one more property of $f_3$. 
\begin{enumerate}
\setcounter{enumi}{6}
\item
Since $f_3$ has degree $1$,  there exist constants  $k_2$ and $\epsilon(k_2)$ as in the condition (5) depending only on $\xi(S)$ such that any $\varepsilon_1$-Margulis tube in $N$ whose core curve has length less than $\epsilon(k_2)$ is contained in the image of a component of $\cv[k_2]$.
\end{enumerate}

\subsection{Preliminary steps to homotope $f_3$ to a bi-Lipschitz map}
We now turn to modify $f_3$ to a bi-Lipschitz homeomorphism, which was done in Brock-Canary-Minsky \cite{bcm} for the case of surface Kleinian groups.
Recall that we moved $\cv$ so that for each brick $B$ in $\mathcal{K}_\mathrm{int}$,  if its upper or lower front $\partial_\pm B$ is real, then every component of $\partial_\pm B \setminus \cv$ is a thrice-punctured sphere.
We parametrise $B$ as $F \times J$ with a closed or half-open interval $J$.
We define $\boldsymbol{i}(B)$ to be a simplex in $\cc(F)$ with empty transversals such that $\boldsymbol{i}(B) \times \{\min J\}$ is homotopic to $\partial_- B \cap \cv$ if $\partial_-B$ is real, and to be the ending lamination of the end corresponding to $F \times \{\inf J\}$ if $\partial_-B$ is ideal.
Similarly we define  $\boldsymbol{t}(B)$ for the upper boundary of $B$.
We shall first show that  in this setting, the block decomposition of $B$ induced by $\cv$ corresponds to a hierarchy in the sense of Masur-Minsky \cite{mm2}.

\begin{lemma}
\label{hierarchy}
Let $B$ be a brick in $\mathcal{K}_\mathrm{int}$, which is homeomorphic to $F \times J$ with a closed or half-open interval $J$.
Then there is a $4$-complete hierarchy $h$ of tight geodesics on $F$ with $I(h)=\boldsymbol{i}(B)$ and $T(h)=\boldsymbol{t}(B)$ whose $4$-sub-hierarchy gives rise to the same  block decomposition of $B$ as the one induced by $\cv$ converted as in Remark \ref{two ways of block} to Minsky's decomposition.
\end{lemma}
\begin{proof}
In the construction of  $\cv$ in the previous section, we began with putting tight tube unions in all connectable bricks in $M_\mathrm{int}$ whose initial and terminal vertices are  in $\boldsymbol{l}(\mathcal{K})$.
After that, we merged homotopic tubes into one and let the obtained tube union be $\cv^{(1)}$.
Then we considered the brick manifold $M^{(1)}$ which is the complement of $\cv^{(1)}$ and repeated the same procedure until we got a block decomposition.
Now, we shall look more closely how tubes are put (and merged) in $B$ during this construction and define  tight geodesics which constitute $h$.
We define $I(B)=\boldsymbol{i}(B) \times \inf J$ and $T(B)=\boldsymbol{t}(B) \times \sup J$.
(These may be larger than $\boldsymbol{I}(B)$ and $\boldsymbol{T}(B)$ defined in the previous section.)

If $B$ is connectable in the first step of the construction of $\cv$, then we get a tube union $\cv_B$ on $B$ in the first step, which corresponds to a tight geodesic $g_B$ in $\cc(F)$ connecting a component of $\boldsymbol{l}(\ck) \cap \partial_- B$ with a component of $\boldsymbol{l}(\ck) \cap \partial_+ B$.
(If one of them is an ending lamination, the geodesic $g_B$ refers to a tight geodesic ray tending to it.) 
Since $\boldsymbol{l}(\ck) \cap \partial_- B \subset \boldsymbol{i}(B)$ and $\boldsymbol{l}(\ck) \cap \partial_+ B \subset \boldsymbol{t}(B)$, we can assume that $g_B$ has $\boldsymbol{i}(B)$ as the initial marking, and $\boldsymbol{t}(B)$ as the terminal marking.

We next consider how  the merging of tubes is reflected in the construction of geodesics in the hierarchy still under the assumption that $B$ is connectable at the first step.
If there is a tube $V$ in $B$ which is merged with another homotopic tube $V'$ in another brick $B'$, then a core curve $v$ of $V$ must be in either $\boldsymbol{i}(B)$ or $\boldsymbol{t}(B)$ since $\partial_\pm B \setminus \cv$ consists of thrice-punctured spheres.
This can occur only when the core curve is contained in either the first, or the second, or the second but last, or the last simplex of the geodesic $g_B$, for its core curve regarded as a curve on $\partial_\pm B$ cannot have non-zero intersection number with $\mathbf I(B)$ or $\mathbf T(B)$.
If $v$ is contained in the first or the last vertex of $g_B$, this procedure of merging does not affect tubes in $B$.
Otherwise, $v$ is  contained in either the second or the second to the last simplex of $g_B$.
In this case, we regard the procedure as corresponding to putting a geodesic consisting of only one vertex, \ie of length $0$, which is subordinate to $g_B$ at the first or the last vertex.

Next, we shall consider the case when $B$ is not connectable at the first step.
In the second step, either $B$ is contained in another brick $\bar B$ constituting $M^{(1)}$ or $B$ is split into two (or more) in the process of merging two homotopic tubes of $\cv^{(1)}$, one lying above $B$ and the other below $B$.
In the latter case, let $V_1, \dots , V_p$ be tubes in $\cv^{(1)}$ which split $B$.
We should note that these tubes have core curves which are homotopic to curves both in $I(B)$ and $T(B)$ then.
Let $v_1, \dots , v_p$ be curves on $F$ corresponding to their core curves.
Then we define geodesics $g_1, \dots, g_p$  each of which consists of only one vertex, such that $D(g_1)=F$, $D(g_j)$ is a component of $F \setminus \cup_{s=1}^{j-1} v_s$ for $j=2, \dots p$, $I(g_j)= \boldsymbol{i}(B) \cap D(g_j), T(g_j)=\boldsymbol{t}(B) \cap D(g_j)$, and $g_{j-1} \supord g_j \subord g_{j-1}$, and let them be geodesics contained in $h$.

In the former case, if $\bar B$ is connectable in $M^{(1)}$, then we consider $\cv_{\bar B} \cap B$, where as explained above $B$ is assumed to be in a position such that $\cv \cap \partial_-B $ is a regular neighbourhood of $I(B)$ and $\cv \cap \partial_+ B$ is that of $T(B)$, and define $g_B$ to be the tight geodesic in $\cc(F)$ corresponding to $\cv_{\bar B} \cap B$.
As before, we define $I(g_B)=\boldsymbol{i}(B)$ and $T(g_B)=\boldsymbol{t}(B)$.
If $\bar B$ is not connectable, we proceed to the following step and repeat the same procedure depending on either there is a brick containing $\bar B$ or $\bar B$ is split by merging of homotopic tubes.
Thus we have defined $g_B$, together with some more geodesics in $h$ in the case when $B$ is split.
We shall now turn to subsequent steps.

In subsequent steps, we put a tight tube union $\cv_{B'}$ into a brick $B'$ constituting a  brick decomposition of $M_\mathrm{int} \setminus \cv^{(k)}$.
We shall show that the intersection with $B$ of  each tube union in a connectable brick $B'$ in the $(k+1)$-th step gives rise to a tight geodesic on $F$ which is subordinate to the ones obtained up to the $k$-th step.
This implies that at the final step, we shall get a hierarchy on $F$ connecting $\boldsymbol{i}(B)$ and $\boldsymbol{t}(B)$.
To show  that, we shall  analyse what a tube union in $B'$ brings about to $B$, dividing the argument into subcases depending on the location of $B'$ with regard to $B$.
(Again, $B$ is in a position where $\partial_- B \cap \cv $ is a regular neighbourhood of $I(B)$ and $\partial_+ B \cap \cv $ is that of $T(B)$.)
We parametrise $B'$ as $F' \times J'$ with $F' \subset F$, in such a way that horizontal leaves and vertical leaves are contained in those of bricks in $\ck_\mathrm{int}$.
Since $F' \times \{x\}$ for $x \in \Int J'$ is a horizontal leaf whose boundary lies on $\partial \cv_k$, the surface $F'$ is a component domain of a geodesic corresponding to a tube union which was already put into $M$ up to the $k$-th step.
Now we divide our argument into three, depending on the inclusive relation between $J$ and $J'$.

First, suppose that $B'$ is contained in $B$, which means that both $\partial_-B'$ and $\partial_+ B'$ lie in $B$ and $J'$ is contained in $J$.
We define $I(B')= \partial_- B' \cap (\cv^{(k)} \cup (\boldsymbol{i}(B) \times \inf J))$ and $T(B')= \partial_+ B' \cap (\cv^{(k)} \cup (\boldsymbol{t}(B) \times \sup J))$.
%, which we can assume to define pants decompositions on $\partial_- B'$ and $\partial_+B'$.
In this definition, we need to add $\boldsymbol{i}(B) \times \inf J$ and $\boldsymbol{t}(B) \times \sup J$ to deal with the case when $\inf B'=\inf B$ or $\sup B'= \sup B$.
Note that $\boldsymbol{I}(B')$ is the union of core curves in $\cv^{(k)} \cap \partial_-B'$ and $\boldsymbol{T}(B')$ is that of $\cv^{(k)}\cap \partial_+B'$, which are contained in $I(B')$ and $T(B')$ respectively.
%We can also see that  $\boldsymbol{I}(B')$ consists of   solid tori in  $B' \cap (\cv^{(k)}\cup\boldsymbol{I}(B))$.
%Similarly, $\boldsymbol{T}(B')$ consists of  solid tori in  $B' \cap (\cv^{(k)}\cup \boldsymbol{T}(B))$.
By our construction of $\cv^{(k+1)}$, the tube union $\cv_{B'}$ in $B'$ connects a component of $\boldsymbol{I}(B')$ to that of $\boldsymbol{T}(B')$.
We define $g_{B'}$ to be the tight geodesic corresponding to $\cv_{B'}$ whose initial and terminal markings are simplices corresponding to $I(B')$ and $T(B')$ respectively.
%Tubes in $\cv^{(k)}$ are vertices in simplices of geodesics which we constructed up to the $k$-th step, and $\boldsymbol{I}(B)$ and $\boldsymbol{T}(B)$ correspond to initial and terminal markings of $g_B$.
%Also since $B'$ is a brick of $M_\mathrm{int} \setminus \cv^{(k)}$, we see that $B' \cap \cv^{(k)} \subset \boldsymbol{I}(B') \cup \boldsymbol{T}(B')$.
The geodesic $g_{B'}$ corresponds to a tube union connecting a tube in $\boldsymbol{I}(B')$ to that  in $\boldsymbol{T}(B')$, which are contained in $(\cv^{(k)} \cap B) \cup (\boldsymbol{i}(B) \times \inf J)$ and $(\cv^{(k)} \cap B) \cup (\boldsymbol{t}(B) \times \sup J)$ respectively.
This shows that the tight geodesic $g_{B'}$  is both forward and backward subordinate to a geodesic in $h$ which was obtained up to the $k$-th step.

Next suppose that, one of $\partial_-B'$ and $\partial_+B'$ is contained in $B$ whereas the other is not.
This means that one of the endpoints of $J'$ lies in $J$ whereas the other does not.
Now, we  assume that $\partial_-B'$ is the one contained in $B$: for the other case,  we can argue in the same way, just interchanging the directions.
In this situation, $\boldsymbol{I}(B')$ consists of core curves of   $\partial_-B  \cap \cv^{(k)}$ which is contained in $\cv^{(k)} \cap B$.
On the other hand, $\boldsymbol{T}(B')$ may not lie in $B$.
Now, by our definition of $\cv$, the tube union $\cv_{B'}$ is contained in $\cv$.
Therefore, $\cv_{B'}$ intersects $\partial_+ B$ by components of $\partial_+ B\cap \cv$ since we moved $\cv$ so that every component of $\partial_+B \setminus \cv$ is a thrice-punctured sphere.
(Recall that unless $\xi(B') = 4$, the upper front of each tube of $\cv_{B'}$ lies on the same horizontal level as the lower front of the subsequent tube.
In the case when $\xi(B') =4$, there is a gap between them, but we moved $\cv$ so that $\partial_\pm B$ avoid such gaps.)
Therefore, if we consider a sub-tube union $\bar \cv_{B'}$ of $\cv_{B'}$ starting from the first tube and ending at a tube in $\partial_+B \cap \cv$, then it is exactly what $\cv_{B'}$ brings about to $B$.
If $\cv_{B'} \cap \partial _+ B$ consists of only one component, then we let $g_{\bar B'}$ be the tight geodesic corresponding to $\bar \cv_{B'}$ defining $I(g_{\bar B'})$  to be a simplex consisting of curves corresponding to core curves of $\partial_-B' \cap (\cv^{(k)} \cup (\boldsymbol{i}(B) \times \inf J)) $ and $T(g_{\bar B'})$ to be $\boldsymbol{t}(B) \cap F'$.
Otherwise, we choose one component  of $\cv_{B'}$, denoted by $V_{B'}^0$ and remove the others, denoted by $V_{B'}^1, \dots, V_{B'}^u$, from $\bar \cv_{B'}$, and then define $g_{\bar B'}$ in the same way.
Since the last tube of $\bar \cv_{B'}$ intersects $\partial_+ B$, it has a core curve contained in $T(B)$.
This implies that $g_{\bar B'}$ is forward subordinate to one of the geodesics obtained up to the $k$-th step.
Since we  know that $g_{\bar B'}$ is also backward subordinate to such a geodesic by the argument in the previous case.

In the case when $\cv_{B'} \cap \partial _+ B$ is not connected, we further define tight geodesics $g_{\bar B'}^1, \dots , g_{B'}^u$ inductively as follows.
Let $v_{B'}^j$ be a core curve of $V_{B'}^j$ for $j=0, \dots , u$.
Let $D$ be a component of $F \setminus v_{B'}^0$ containing $v_{B'}^1$, and let $v_{B'}^{-1}$ be the simplex of $g_{\bar B'}$ which precedes $v_{B'}^0$.
Then we define $g_{B'}^1$ to be a tight geodesic of length $1$ supported on $D$ with $I(g_{B'}^1)$ equal to $v_{B'}^{-1} \cap D$ and $T(g_{B'}^1)$ equal to $\boldsymbol{t}(B) \cap D$.
In the same way, we define $D^k$ to be the component of $F \setminus (v_{B'}^0 \cup \dots \cup v_{B'}^{k-1})$ containing $v_{B'}^k$ and $g_{B'}^k$ to be a tight geodesic of length $1$ supported on $D^k$ with $I(g_{B'}^k)$ equal to $v_{B'} \cap D^k$ and $T(g_{b'}^k)$ equal to $\boldsymbol{t}(B) \cap D^k$.
Then all these geodesics $g_{B'}^0, \dots , g_{B'}^u$ are subordinate to $g_{\bar B'}$ in both directions.
%We define $I(g_{\bar B'})$ to be curves corresponding to core curves of $\partial_- B' \cap \cv$ and $T(g_{\bar B'})$ those corresponding to the core curves of $\partial_+B' \cap \cv$.
Thus in either case, we get tight geodesics which are both backward and forward subordinate to geodesics in $h$ obtained up to the previous step.

Finally suppose that neither $\partial_-B'$ nor $\partial_+B'$ is contained in $B$.
Then  $\cv_{B'}$ intersects $\partial_-B$ by a union of  solid tori $V_1$ contained in $\partial_-B \cap \cv$ and $\partial_+B$ by a union of solid tori $V_2$ contained  in $\partial_+B \cap \cv$.
We define $\bar \cv_{B'}$ to be the sub-tube union of $\cv_{B'}$ starting from a component of $V_1$ and ending at  a component of $V_2$.
Then $\bar \cv_{B'}$ is the union of tubes which $\cv_{B'}$ brings about to $B$. 
We define $g_{\bar B'}$ to be the corresponding tight geodesic supported on $F'$, setting its initial and terminal markings to be $\boldsymbol{i}(B) \cap F'$ and $\boldsymbol{t}(B)\cap F'$ respectively.
In the same way as in the previous paragraph, we define geodesics of length $1$ corresponding to the components of $V_1$ and $V_2$ which are not chosen.
These geodesics are both forward and backward subordinate to geodesics which are obtained up to the $k$-th step by the same reason as in the previous case.

Now, recall that in each step, we also merge homotopic tubes into one.
As was analysed in the first step, this procedure corresponds to putting a geodesic consisting of only one vertex which is subordinate to a geodesic which was already constructed in the previous steps.
Thus we have shown that at each step we get a geodesic subordinate to those which we have already had and at the final step, we get non-annular geodesics in  $h$.

We shall next define annular geodesics of $h$.
Let $V$ be a tube in $\cv$ intersecting $B$ along its entire boundary.
We parametrise $V$ as $A \times [0,1]$ preserving leaves.
Since $A \times \{0\}$ lies on $\partial_+ b$ for some block $b$ of the form $\Sg_{0,4} \times J$ or $\Sg_{1,1} \times J$, the core curve of the annulus on $\partial_-b$ which is the complement of the other blocs intersecting $\partial_- b$ defines an arc on $A \times \{0\}$ which is regarded as a vertex $v_-$ in $\cc(A)$.
Similarly we can define an arc on $A \times \{1\}$ regarded as a vertex $v_+$ in $\cc(A)$.
We define a geodesic  $g_B$ supported on $A$ connecting $v_-$ and $v_+$, and let it be contained in $h$.
By our construction, this geodesic is both forward and backward subordinate to $4$-geodesics.

%%%%%$4$-hierarchy etc ?????`???????D
It remains to show that the block decomposition of $B$ induced from $\cv$ is compatible with $h$.
This means that we have a resolution of the $4$-sub-hierarchy of $h$ which gives rise to the block decomposition induced from $\cv$.
We consider the family of horizontal surfaces $F \times \{t\}$ in $B$.
Then outside countably many intervals corresponding to gaps which we introduced for $4$-geodesics,  $F \times \{t\} \cap \cv$ induces a pants decomposition of $F$.
We should also note that if $t$ is contained in a gap interval, then $F \times \{t\}$ passes a block of the form either $\Sg_{0,4} \times J$ or $\Sg_{1,1} \times J$.
Passing each interval of gap, the configuration of pants decomposition changes by elementary moves which may take place at finitely many disjoint places at the same time.
This must come from stepping forward on $4$-geodesics which we defined above.
Therefore, we can define a resolution of the $4$-sub-hierarchy of $h$.
Since each elementary move also corresponds to a block of the form $\Sg_{0,4} \times J$ or $\Sg_{1,1} \times J$ in our decomposition, the block decomposition induced from this resolution is obtained by converting the one induced from $\cv$ as in Remark \ref{two ways of block}.
\end{proof}

In a hierarchy, a curve can appear at most once.
Since our tube union $\cv$ itself does not correspond to a hierarchy, (only its restriction to a brick corresponds to a hierarchy) we need to show the same kind of property for $\cv$.

\begin{lemma}
\label{no homotopic tubes}
There are no two distinct tubes in $\cv$ which are homotopic in $M$.
\end{lemma}
\begin{proof}
Suppose, seeking a contradiction, that there are tubes $V_1, V_2$ in $\cv$ which are homotopic to each other in $M$.
Let $k_1, k_2$ be the numbers such that $V_1 \in \cv^{(k_1)}$ and $V_2 \in \cv^{(k_2)}$, and set $k= \max\{k_1, k_2\}$.
(When we say $V_s \in \cv^{(k_s)}$ for $s=1,2$, we take the smallest $k_s$ such that $\cv^{(k_s)}$ contains $V_s$.
We follow the same convention throughout the proof.)
Then longitudes of $V_1$ and $V_2$ pushed out to their boundaries are not homotopic in $M^{(k)}$, for otherwise they should have been merged into one in our construction.
Let $\cu$ be the union of  tubes in $\cup_{j=1}^k \cv_k$ which intersects essentially an embedded annulus $A$ bounded by the pushed-out longitudes of $V_1$ and $V_2$ in $M$.
(These are determined independently of the choice of an annulus since $M$ is atoroidal.)

Let $U \in \cu$ be a tube which appears in the earliest step among the tubes in $\cu$, and suppose that $U \in \cv^{(l)}$.
Note that we have $l \leq k$ by our definition of $\cu$.
Let $B \cong F \times J$ be a brick in $\ck^{(l-1)}$ where $U$ appears as a tube in the tight tube union.
If either  $V_1$ or $V_2$, say $V_1$,  intersects a front of $B$, then by replacing $V_1 \cap B$ with $V_1$, we can assume that both $V_1$ and $V_2$ are disjoint from the fronts of $B$.
First suppose that both $V_1$ and $V_2$ are contained in $B$.
In the following argument, for two tubes $U,V \in \cv$, we write $U \approx V$ if $U=A_1 \times J_1$, $V=A_2 \times J_2$ and $\Int J_1 \cap \Int J_2 \neq \emptyset$ for the parametrisation on $S \times (0,1)$ in which $M$ is embedded by $\iota_M$.
In the $k$-th step, a tight tube union $\cv_B$ corresponding to a tight geodesic $g_B$ on $\cc(F)$ is given.
Then there are tubes $U_1, U_2$ in the tight tube union of $B$ such that $V_1 \approx U_1$ and $V_2 \approx U_2$.
Since $U$ intersects $A$, we have $U_1 < U < U_2$ or $U_2 < U < U_1$ with respect to the ordering on the simplices of $g_B$.
Let $u_1, u_2, u, v$ be vertices of $\cc(F)$, which correspond to $U_1, U_2, U, V_1$.
Then, we have $u_1, u_2\in \phi_{g_B}(v)$ whereas $u\not\in \phi_{g_B}(v)$.
This contradicts the fact that $\phi_{g_B}(u)$ consists of contiguous, which was proved in Lemma 4.10 in Masur-Minsky \cite{mm2}.

 Next suppose that one of $V_1, V_2$, say $V_1$, lies outside $B$ whereas $V_2$ is contained in $B$.
 In this case, $A$ passes through a joint contained in the upper or the lower front of $B$.
 We only consider the case when $A$ passes through a joint in the upper front.
 The other case can be dealt with in the same way just by turning the figure upside down.
 Since $A$ passes through a joint in the upper front, we have the vertical projection of the core curve of $\boldsymbol{T}(B)$ is disjoint from that of the core curve of $V_1$, which implies that the last vertex $u_\infty$ of $g_B$  is contained in $\phi_{g_B}(v)$.
As in the previous paragraph, we have $u_1 < u < u_\infty$ and $u_1, u_\infty \in \phi_{g_B}(v)$ whereas $u \not\in \phi_{g_B}(v)$, which contradicts Lemma 4.10 of \cite{mm2} as before.
Also in the case when both $V_1$ and $V_2$ lie outside $B$, we can argue in the same way considering joints which $A$ passes contained in the upper and the lower fronts.
Then we see that the first and the last vertices are contained in $\phi_{g_B}(v)$ whereas $u$ is not.
This again contradicts Lemma 4.10 in \cite{mm2}, which completes the proof.
%Suppose that there is a tube among $U_1, \dots , U_p$, say $U_1$, which is contained in $\cv^{(k)}$.
%We can assume that $k=k_1$ by interchanging $V_1, V_2$ if necessary.
%Let $B$ be a brick of $\ck^{(k-1)}$ whose tube union contains $V_1$.
%We first consider the case when $V_2$ is also contained in $B$.
%Since $A$ cannot be contained in $M^{(k-1)}$, there must be a tube $U$ in $\cu$ 
\end{proof}

The next lemma is obtained from Otal \cite{otal} for  hyperbolic 3-manifolds  homeomorphic to $S \times \reals$ for a hyperbolic surface $S$.
Since the only condition that is necessary for the proof is the fact that the manifold can be filled up by incompressible pleated surfaces (with bounded genus), his argument also works in our settings.
\begin{lemma}
\label{Otal}
There is a constant $k_0$ depending only on $\chi(S)$ such that for any $k \geq k_0$ and tubes $V \in \cv[k]$, the core curves $c$ of the $V$ are mapped by $f_3$ to   unknotted and unlinked closed geodesics, \ie there is a isotopy of $S \times(0,1)$ which takes $\iota_N(c)$ to disjoint collection of simple closed curves lying on horizontal surfaces.
\end{lemma}

Take $k_2$ in the condition (7) so that $\epsilon(k_2)$ is less than our fixed $\varepsilon_1$ (less than the Margulis constant).
By Lemma \ref{large k}. there exists $k_1$ such that if $|\omega(V)| \geq k_1$, then $f(v)$ has length less than $\epsilon(k_2)$.
We define $k_u = \max\{k_0, k_1, k_2\}$ for $k_0$ in the above lemma, and let $\epsilon_u$ be $\epsilon(k_u)$.
%\varepsilon_1??\varepsilon_0???????????D
%The next lemma shows that every short geodesic comes from a brick.

%\begin{lemma}
%\label{short geodesic}
%There exists a positive constant $\eta$ for which the following holds.
%For every closed geodesic $\gamma$ in $N$ whose length is less than $\eta$, there is a closed curve $c$ in a brick $B$ such that $f(c)$ is homotopic to $\gamma$.
%\end{lemma}
%\begin{proof}
%We shall show that there is a positive constant bounding from below the lengths of closed geodesics which are not homotopic to the images of closed curves in bricks.

%\end{proof}

Now, we recall the following definition of topological order introduced in Brock-Canary-Minsky \cite{bcm}, which we shall apply for  surfaces in $M$ or $N_0$.

\begin{definition}[Brock-Canary-Minsky \cite{bcm}]
\label{def:topological order}
Let $f_1: F_1 \rightarrow M$ and $f_2 : F_2 \rightarrow M$ be maps from essential subsurfaces $F_1, F_2 \subset S$ such that $\iota_M \circ F_j$ is homotopic to the inclusion $F_j \rightarrow F_j \times \{t\}$.
We say $f_1 \prec_\mathrm{top} f_2$ if and only if $\iota_M \circ f_1$ can be homotoped to $S \times \{0\}$ without touching $\iota_M \circ f_2(F_2)$ and $\iota_M \circ f_2$ can be homotoped to $S \times \{1\}$ without touching $\iota_M \circ f_1(F_1)$.
We define the topological order on maps from surfaces to $N_0$ in the same way just replacing $M$ with $N_0$ and $\iota_M$ with $\iota_N$.
\end{definition}

We should also recall that two embedded surfaces $F_1, F_2$ in $S\times (0,1)$ are said to overlap if their projections to $S$ have essential intersection.
We use this term also for surfaces in $M$ or $N_0$, for they can be embedded in $S\times (0,1)$ by $\iota_M$ and $\iota_N$.

\subsection{Homotoping $f_3$ to a homeomorphism}
\label{check defined}
We shall next consider to homotope $f_3$ so that its restriction to the union of the joints of the bricks is an embedding.
Let $F$ be a joint of $B$ with another brick.
Recall that $F$  intersects $\cv$ in such  a way that each component of $F \setminus \cv$ is a thrice-punctured sphere.
We define $\check F[k]$ to be an embedded surface in $B \cap M[k]$ obtained from $F$ by isotoping annuli in $F \cap \cv[k]$ to those on  $\partial \cv[k]$.
There are two choices for an annulus for each component of $F \cap \cv[k]$.
We take an annulus on $\partial \cv[k]$ which is situated lower than the other with respect to the embedding $\iota_M$.

Recall that  the images of $\cv[k_u]$ are unknotted and unlinked $\varepsilon_1$-Margulis tubes whose core curves have lengths less than $\epsilon(k_2)$.
Conversely, every $\varepsilon_1$-Margulis tube whose core curve has length less than $\epsilon(k_2)$ is the image of a component of $\cv[k_2]$ by $f_3$.
Recall that we denote the union of the  Margulis tubes which are images of tubes in $\cv[k_u]$ by $T[k_u]$.
We denote $N_0 \setminus \Int T[k_u]$ by $N[k_u]$.
By Lemma \ref{no homotopic tubes}, $f_3$ induces a bijection between the components of $\cv[k_u]$ and those of $T[k_u]$.
Moreover, the image of $M[k_2]$ is disjoint from $T[k_u]$ by the condition (5).
Again by Lemma \ref{no homotopic tubes}, no tubes in $\cv[k_2] \setminus \cv[k_u]$ are mapped to $T[k_u]$.
Therefore $f_3$ induces a Lipschitz map $f_3: M[k_u] \rightarrow N[k_u]$.

%Recall that each gluing surface $\Sigma$ has its boundary on tubes $V_1, V_2, V_3 \in \cv$.
%Let $c_1, c_2, c_3$ be the core curves of $V_1, V_2, V_3$, and $\gamma_1, \gamma_2, \gamma_3$ the closed geodesics homotopic to $f(c_1), f(c_2), f(c_3)$ respectively.
%We can extend $\Sigma$ to $\bar \Sigma$ whose boundary components are $c_1, c_2$ and $c_3$, and the  map $f_2|\Sigma$ is a restriction of a pleated surface defined on $\bar \Sigma$, which we denote by $\tilde f_\Sigma: \bar \Sigma \rightarrow N$.
\begin{prop}
\label{disjoint embeddings}
The Lipschitz map $f_3 : M[k_u] \rightarrow N[k_u]$ can be properly homotoped to a homeomorphism $f_4: M[k_u] \rightarrow N[k_u]$, which extends to a homeomorphism between $M$ and $N_0$.
This map $f_4$ may not be Lipschitz.
\end{prop}
\begin{proof}
Let $B$ be a brick of $M_\mathrm{int}$.
We denote by $F^+_1, \dots , F^+_\mu$ its joints contained in the upper front, and by $F^-_1, \dots , F^-_\nu$ those contained in the lower front.
(One of the fronts may be ideal; hence one of these families may be empty.)
We consider $\check F^+_1[k_u], \dots , \check F^+_\mu[k_u]$ and $\check F^-_1[k_u], \dots , \check F^-_\nu[k_u]$ as defined above, and denote their unions by $\check F_+$ and $\check F_-$.
Note that both $\check F_+$ and $\check F_-$ are incompressible in $M$.
By changing each joint $F$ to $\check F$, we get a brick decomposition of $M$ which is isotopic to the original one.
From now on until the end of the proof of this proposition, when we refer to a brick $B$, we mean the one in this new decomposition, which is isotopic to the original $B$.
Let $p_B: M_B \rightarrow M$ be the covering  associated to the image of $\pi_1(B)$ in $\pi_1(M)$.
Similarly, we consider the covering $N_B$ of $N_0$ associated to $(f_3)_\# \pi_1(B)$.
Let $\tilde f_3 : M_B \rightarrow N_B$ be the lift of $f_3$ which is uniformly Lipschitz outside the preimages of $\cv[k_u]$, and $\tilde{f}: M_B \rightarrow N_B$ that of $f$, which is a homeomorphism.
The surfaces $\check F_+, \check F_-$ lift homeomorphically to surfaces $\tilde F_+, \tilde F_-$ lying on the boundary of a submanifold $\tilde B$ homeomorphic to $B$ under $p_B$.
We use the symbols $\partial_- \tilde B$ and $\partial_+ \tilde B$ to denote the fronts of $\tilde B$ corresponding to $\partial_- B$ and $\partial_+B$ respectively.
Let $\tilde \cv[k_2]$ and $\tilde T[k_u]$ be the preimages of $\cv[k_2]$ and $T[k_u]$ respectively.
We denote  by $M_B[k_u]$ the complement of $\Int \tilde \cv[k_u]$ in $M_B$, and by $N_B[k_u]$ the complement of $\Int \tilde T[k_u]$ in $N_B$.

Note that  $\tilde f_3|(\tilde F_+ \sqcup \tilde F_-)$ is properly homotopic to $\tilde f|(\tilde F_+ \sqcup \tilde F_-)$ which is an embedding.
We can assume that $\tilde f_3|(\tilde F_+ \sqcup \tilde F_-)$ is an immersion from the start by perturbing it.
Then, as was shown in Freedman-Hass-Scott \cite{fhs}, $\tilde f_3|(\tilde F_+ \sqcup \tilde F_-)$ can be properly homotoped to an embedding by a homotopy which passes through only relatively compact components of $N_B \setminus \tilde f_3(\tilde F_+ \sqcup \tilde F_-)$.
%By a simple argument using the Euler characteristic, we see that each of such relatively compact components must be either an open solid torus or an open ball.
We note that each of such relatively compact components is homeomorphic to a trivial $I$-bundle whose  associated $\partial I$-bundle can be identified with a compact subsurface of $\tilde F_+ \sqcup \tilde F_-$.

Suppose that  a component $W$ of $N_B \setminus \tilde f_3(\tilde F_+ \sqcup \tilde F_-)$ intersects a component $T$ of $\tilde T[k_u]$.
This means that $W$ contains $T$ since $f_3(\check F)$ is disjoint from $T[k_u]$.
%We should first observe that  at most one component of $\tilde T[k_u]$ can be contained in $W$.
%If there were more than one components of $\tilde T[k_u]$ contained in $W$, by projecting them into $N_0$, we should have at least two components of $f(\cv)$ whose core curves are homotopic in $N[0]$.
%This implies that there would be two core curves of $\cv$ which are homotopic in $M$.
%This contradicts Lemma \ref{no homotopic tubes}.
%Thus we have shown that $W$ contains only one component $T$ of $\tilde T[k_u]$.
We shall now prove the following claim.

\begin{claim}
The surfaces $\tilde f_3(\tilde F_+)$ and $\tilde f_3(\tilde F_-)$ are homotopic to disjoint embeddings by proper homotopies which do not touch $T$.
\end{claim}
\begin{proof}
Because $f_3$ satisfies the conditions (3), (5) and (7), there is a unique component $V$ of $\tilde \cv[k_u]$, which is  a solid torus, such that $\tilde f_3(V)=T$.
Since $M[k_u]$ is mapped to $N[k_u]$ and $\cv[k_u]$ bijects to $T[k_u]$, we see that  $\tilde f_3(M_B \setminus V) \subset N_B \setminus T$.
Since every Kleinian surface group is tame, the interior of  $N_B$ is homeomorphic to $\partial_-  B \times (0,1)$, and hence so is $M_B$.
Let $V_1, \dots, V_p$ be all components of $\partial M$, which are open annuli or tori, whose longitudes or core curves are homotopic into $\check F^+_1 \cup \dots \cup \check F^+_\mu$ in $M \setminus \Int B$.
We renumber them in such a way that $V_1, \dots , V_r$ are disjoint from $B$ whereas $V_{r+1}, \dots , V_p$ intersect $B$ along annuli.
By the annulus theorem and a standard cut-and-paste technique, we see that there are disjoint embedded annuli $\alpha_1, \dots , \alpha_r$ realising homotopies between longitudes or core curves on $V_1, \dots , V_p$ and simple closed curves on $\check F^+_1 \cup \dots \cup \check F^+_p$ with  $\partial \alpha_j \subset V_j \cup \check F^+_1 \cup \dots \cup \check F^+_p$.
We can lift $V_1, \dots , V_p$ and $\alpha_1, \dots , \alpha_r$ to open annuli  $A_1, \dots , A_p$ on $\partial M_B$ and annuli $\tilde \alpha_1, \dots , \tilde \alpha_r$ such that    $A_j$ and $\tilde \alpha_j$ intersect at  a core curve of $A_j$ for $j =1, \dots , r$. 
Similarly, we consider  components $V_1', \dots , V'_q$ of $\partial M$ whose longitudes or core curves are homotopic into $\check F^-_1 \cup \dots \cup \check F^-_\nu$ in $M \setminus \Int B$, among which $V_1', \dots , V'_s$ lie outside $B$, and take annuli $\alpha'_1, \dots , \alpha'_s$ realising homotopies between longitudes or core curves to $\check F^-_1 \cup \dots \cup \check F^-_\nu$.
We lift $V_1', \dots , V_q'$  to  open annuli $A'_1, \dots , A'_q$ and $\alpha'_1, \dots , \alpha'_s$ to annuli $\tilde \alpha_1', \dots , \tilde \alpha_s'$ in the same way as before.

Let $\bar A_1, \dots , \bar A_p$ and $\bar A'_1, \dots , \bar A'_q$ be core annuli of $A_1, \dots , A_p$ and $A'_1, \dots , A'_q$ such that $\bar A_j$ contains $ \tilde \alpha_j \cap A_j$ for $j\leq r$ whereas $\bar A_j= \tilde B \cap A_j$ for $j >r$, and $\bar A'_j$ contains $\alpha'_j \cap A'_j$ for $j \leq s$ whereas $\bar A_j'=\tilde B \cap A'_j$ for $j > s$.
By identifying $\partial_-\tilde B$ and $\partial_+\tilde B$ by vertical translation and $\partial_- \tilde B$ with $\partial_- B$ by $p_B$, we regard $\bar A_1, \dots , \bar A_p; \bar A'_1, \dots, \bar A'_q$ as  lying on $\partial_-  B$.
To construct parts corresponding to the $\integers$-cusps in $\partial_-  B \times (0,1)$,  we set $U_+=(\bar A_1 \cup \dots \cup \bar A_r) \times (7/8,1)$,  $U'_+=(\bar A_{r+1} \cup \dots \cup \bar A_p) \times (3/4,1)$, $U_-=(\bar A'_1 \cup \dots \cup \bar A'_s) \times (0,1/8)$, and $U'_-=(\bar A'_{s+1} \cup \dots \cup \bar A'_q) \times (0,1/4)$ and denote the union of these four parts by $U$.
We parametrise $M_B$ by a proper homeomorphism $I_M: M_B \rightarrow \partial_- B \times (0,1) \setminus U$, in such a way that $\check F_-$ is identified with the horizontal surface $\partial_- B\times \{1/4\} \setminus \Int U'_-$ whereas $\tilde F_+$ is identified with the horizontal surface $\partial_+ B\times \{3/4\} \setminus \Int U'_+$.

Similarly, we parametrise $N_B$ by a homeomorphism $I_N : N_B \rightarrow \partial_- B \times (0,1) \setminus U$ in such a way that  $I_N(\tilde f_3(\tilde B))$ lies in $\partial_- B \times [1/4,3/4]$ and $I_N(W)$ lies in $\partial_- B \times [1/8, 7/8]$.
Note that each component of $\partial U$ corresponds to the boundary  of a $\integers$-cusp neighbourhood of  $N_B$.
Since $N_B$ is the covering of the non-cuspidal part $N_0$, we can extend $N_B$ to a hyperbolic $3$-manifold $\hat N_B$ which is the covering of $N$ associated to $\pi_1(B)$ by attaching cusp neighbourhoods.
Therefore, the parametrisation $I_N$ extends to a homeomorphism $\hat I_N : \hat N_B \rightarrow \partial_- B \times (0,1)$.

Since both $\check F_+$ and $\check F_-$ are disjoint from $\tilde \cv[k_u]$, the solid torus $I_M(V)$ is contained in either $\partial_- B \times (0, 1/4)$ or $\partial_- B \times (1/4,3/4)$ or $\partial_- B \times (3/4,1)$.
We shall first consider the case when $I_M(V)$ is contained in $\partial_- B \times (1/4,3/4)$.
Take a sufficiently small number $s_0$ so that both $\hat I_N^{-1}(\partial_- B \times (1-s_0, 1))$ and $\hat I_N^{-1}(\partial_- B \times (0,s_0))$ are disjoint from the $1$-neighbourhood of $W$.
Since $f_3$ is proper and has degree $1$, for sufficiently small $t_0>0$, the surfaces $I_N \circ\tilde f_3 \circ I_M^{-1}(\partial_- B \times \{ t_0\} \setminus U)$ and $I_N \circ\tilde f_3 \circ I_M^{-1}(\partial_- B \times \{1-t_0\} \setminus U)$ are contained in $\partial_- B \times (0,s_0)$ and  $\partial_-B \times (1-s_0,1)$ respectively.
Denote $I_M^{-1}(\tilde F_+ \times \{1-t_0\} \setminus U)$ by $F'_+$, and $I_M^{-1}(\tilde F_-  \times \{t_0\} \setminus U)$ by $F'_-$.

We can enlarge $F'_-$ and $F'_+$ to surfaces $\check F'_-$ and $\check F'_+$ homeomorphic to $\check F_-$ and $\check F_+$ respectively by joining pairs of parallel boundary components of $F'_-$ lying on $\partial U_-$ by annuli on $\partial U_-$ bounded by them, and those of $F'_+$ lying on $\partial U_+$  by annuli on  $\partial U_+$ bounded by them.
On the other hand, since $\tilde f_3(F'_-)$ and $\tilde f_3(F'_+)$ are disjoint from the $1$-neighbourhood of $W$, we can enlarge $\tilde f_3(F'_-)$ and $\tilde f_3(F'_+)$ by joining each pair of parallel boundary component on $I_N \circ \tilde f_3 (\partial U_- \cup \partial U_+) \subset \partial N_0$ by an annulus embedded in the closure of an $\ve$-cusp neighbourhood which is a component of $N \setminus \Int N_0$ so that their images under $\hat I_N$ are contained in $\partial_- B \times (0, s_0)$ and $\partial_- B \times (1-s_0, 1)$ respectively.
These surfaces, which are homeomorphic to $\check F_-$ and $\check F_+$, are homotopic to embeddings $\bar F_-$ and $\bar F_+$  respectively outside the $1$-neighbourhood of $W$ by our choice of $s_0$, again using the result of Freedman-Hass-Scott.
Then by our choice of $t_0$, we see that  $\tilde f_3 (\check F'_+)$ and $\tilde f_3 (\check F'_-)$ are homotopic to $\bar F_-$ and $\bar F_+$ respectively by homotopies disjoint from $W \supset T$.
Since $\tilde f_3(\tilde F_-)$ is homotopic to $\tilde f_3 (\check F'_-)$ outside $T$ and $\tilde f_3(\tilde F_+)$ is homotopic to $\tilde f_3 (\check F'_+)$ outside $T$ (for $I_M(V)$ is contained in $\partial_-B \times (1/4,3/4)$ and $\tilde f_3(M_B \setminus V) \subset N_B \setminus T$), the surfaces $\tilde f_3(\tilde F_-)$ and $\tilde f_3(\tilde F_+)$ are homotopic to disjoint embeddings by homotopies disjoint from $T$.

Next suppose that $I_M(V)$ is contained in $\partial_-B \times (0,1/4)$.
In this case, we shall consider to move both $\tilde F_-$ and $\tilde F_+ $ in the +-direction.
As in the previous case, there are sufficiently small  $s_0, t_0>0$ such that $\hat I_N^{-1}(\partial_-B \times (1-s_0))$ is disjoint from the $1$-neighbourhood of $W$, and such that $I_N \circ \tilde f_3\circ I_M^{-1}(\partial_- B \times \{1-t_0\})$ is contained in $\partial_- B \times (1-s_0, 1)$.
Then, by the same argument as in the previous case, we can see that both $\tilde f_3(\tilde F_-)$ and $\tilde f_3 (\tilde F_+)$ are homotopic to an embedding contained in $\hat I_N^{-1}(\partial_- B \times (1-s_0, 1))$ by a homotopy outside $T$.
They can be homotoped to disjoint embeddings just by considering parallel copies of the embedding.
Thus we are done also in this case.
The argument for the case when $I_M(V)$ is contained in $\partial_- B \times (3/4,1)$ is the same way just by changing the $+$-direction to the $-$-direction.
\end{proof}

The above claim says that a homotopy from $f_3(\tilde  F_+ \sqcup \tilde F_-)$ can be taken to be disjoint from $W$ since any homotopy passing through $W$ must intersect $T$.
We can repeat the same argument for every relatively compact component of $N_0 \setminus \tilde f_3(\tilde F_+ \sqcup \tilde F_-)$ containing a component of $\tilde T[k_u]$ and show that $\tilde f_3(\tilde F_+ \sqcup \tilde F_-)$ can be homotoped to an embedding by a homotopy within $N_B[k_u]$.

Now, we consider a new  hyperbolic metric $m_N$ on $\Int N[k_u]$ which makes every component of $T[k_u]$ a torus cusp preserving the original cusps of $N$.
Pull back this metric to $\Int N_B[k_u]$ and denote it by $m_B$.
We consider a least area map $h_3: \tilde F_- \sqcup \tilde F_+ \rightarrow (\Int N_B[k_u], m_B)$ homotopic to $\tilde f_3| \tilde F_- \sqcup \tilde F_+$.
By the main result of Freedman-Hass-Scott \cite{fhs}, $h_3$ is an embedding.

In the following argument, we shall use the notion of topological order due to Brock-Canary-Minsky \cite{bcm} which we explained in Definition \ref{def:topological order}.
%We refer the reader to \S3 of \cite{bcm} for its definition and basic properties.

\begin{claim}
\label{extended embedding}
Let $B$ be a brick in $M_\mathrm{int}$ neither of whose fronts lies on the boundary of $M_\mathrm{int}$.
Then the embedding $h_3$ can be extended to an orientation-preserving embedding of $\tilde B \cap \Int M_B[k_u]$ to $(\Int N_B[k_u], m_B)$ taking $\tilde B \cap \tilde \cv[k_u]$ to cusps corresponding to $\tilde T[k_u]$ and the homotopy classes of meridians of tube components of $\tilde B \cap \tilde \cv[k_u]$  to those of $\tilde T[k_u]$.
\end{claim} 
\begin{proof}
Recall that there is a homeomorphism $I_N: N_B \rightarrow \partial_- B \times (0,1) \setminus U$.
By our definition of $k_u$, the images of the tube components of $\tilde T[k_u]$ under $I_N$ are unknotted and unlinked in $\partial_- B \times (0,1)$.
Since  ends of $h_3(\tilde F_- \sqcup \tilde F_+)$ other than those tending to cusps of $N_B$ tend to $I_N(\partial \tilde T[k_u])$, the surfaces $I_N \circ h_3(\tilde F_-) \sqcup I_N \circ h_3(\tilde F_+)$ together with annuli on $I_N(\partial \tilde T[k_u])$ bound a submanifold  homeomorphic to $\partial_-B \times [1/4,3/4] \cong \tilde B$.
We shall first prove that $I_N \circ h_3(\tilde F_+)$ is situated above $I_N \circ h_3(\tilde F_-)$.
This trivially holds when one of $\tilde F_+$ and $\tilde F_-$ is empty.
Therefore, we assume that neither of them is empty.
Since we assumed that neither $\partial_- B$ nor $\partial_+B$ lie on the boundary of $M_\mathrm{int}$, both $\partial_- \tilde B \cap \partial M_B$ and $\partial_+  \tilde B \cap \partial M_B$ are non-empty.

By Assumption \ref{modification}, every component of $\partial_-  B \cap \partial M$ overlaps some component of $ \partial_+  B \cap \partial M$.
Therefore, we can take components $X$ and $X'$ of  $\tilde \cv[k_u]$ on which boundary components of $\tilde F_+$ and of $\tilde F_-$ lie respectively such that $X \cap \tilde B$ and $X' \cap \tilde B$ overlap.
It follows that we have $X \cap \tilde B \prec_\mathrm{top} X' \cap \tilde B$.
Since $\tilde f_3$ is a proper degree $1$-map and $\tilde f_3|\tilde \cv[k_u]$ is a homeomorphism to its image, this implies that $\tilde f_3(X \cap \tilde B) \pretop \tilde f_3(X \cap \tilde B)$.
On the other hand, if $I_N \circ h_3(\tilde F_+)$ is situated under $I_N \circ h_3(\tilde F_-)$, then we should have $\tilde f_3(X'  \cap \tilde B) \pretop \tilde f_3(X \cap \tilde B)$, which is a contradiction.
Thus we have proved that $I_N \circ h_3(\tilde F_+)$ is situated above $I_N \circ h_3(\tilde F_-)$ and $h_3$ extends to an orientation-preserving homeomorphism from $\tilde B$ to a submanifold $B_N$ bounded by $h_3 (\tilde F_- \cup \tilde F_+)$.

We shall next show that this homeomorphism induces one between $\tilde B \cap \Int M_B[k_u]$ to $B_N \cap \Int N_B[k_u]$.
For that, it suffices to show that for the components of $\tilde \cv[k_u]$ in $\tilde B$, the corresponding components of $\tilde T[k_u]$ are contained in $B_N$ preserving the topological order since all such components in $B_N$ are unknotted and unlinked. 
Let $V$ be a component of $\tilde \cv[k_u]$ contained in $\tilde B$.
Then we have $\tilde F_- \pretop V \pretop \tilde F_+$.
Let $T$ be a component of $\tilde T[k_u]$ with $T=\tilde f_3(V)$.
Since $\tilde f_3$ is a proper degree-$1$ map and takes $M_B \setminus V$ to $N_B  \setminus T$, we see that $\tilde f_3(\tilde F_-) \pretop T \pretop \tilde f_3(\tilde F_+)$.
Since $h_3$, defined on $\tilde F_- \sqcup \tilde F_+$, is homotopic to $f_3|(\tilde F_- \sqcup \tilde F_+)$ in $N_B[k_u]$, we also have $h_3(\tilde F_-) \pretop T \pretop h_3(\tilde F_+)$.
Therefore any tube component of $\cv[k_u]$ in $\tilde B$ has its corresponding Margulis tube in $B_N$.
Now suppose that we have two such tube components $V_1, V_2$ with $V_1 \pretop V_2$.
Let $T_1, T_2$ be the components of $\tilde T[k_u]$ with $\tilde f_3(V_1)=T_1$ and $\tilde f_3(V_2)=T_2$.
Then by the same argument as above using the bijective correspondence between the components of $\tilde \cv[k_u]$ and $\tilde T[k_u]$, we have $T_1 \pretop T_2$.
Thus we have shown that we have a homeomorphism $\bar h_3$ from $\tilde B \cap \Int M_B[k_u]$ to $B_N \cap \Int N_B[k_u]$ which is an extension of $h_3$.

It remains to show that a meridian of a solid torus component of $\tilde \cv[k_u]$ contained in $\tilde B$ is taken to a  meridian of $\tilde T[k_u]$ by $\bar h_3$.
This is rather obvious from our construction: for $\tilde f_3$ takes  meridians of   solid torus components of $\tilde \cv[k_u]$ to those of $\tilde T[k_u]$.
%
%$\bar h_3$ takes a longitude $l$ of a solid torus component $V$ of $\tilde \cv[k_u] \cap \tilde B$ to a longitude of a component $T$of $\tilde T[k_u] \cap B_N$.
%If a meridian $m$ of $V$ is not mapped to that of  $T$, then it should be mapped to a composite of a meridian and a power of  a longitude.
%
\end{proof}

Now, for each brick $B$ of $\ck_\mathrm{int}$ neither of whose fronts lies on the boundary of $M_\mathrm{int}$, we consider $B \cap \Int M[k_u]$,  its lift $\tilde B$ in $\Int M_B[k_u]$, and its  embedding of  $\Int N_B[k_u]$ by an extension of $h_3$ as above, which we denote by $B_N$ as above.
We denote the map taking $B \cap \Int M[k_u]$ to $B_N$ in this way by $f_B$.
We regard $B_N$ as a hyperbolic $3$-manifold with boundary by restricting the metric $m_B$, and call $B_N$ with this metric the {\em least-area realisation of B}.
When $B$ is a brick one of whose front lies on the boundary of $M_\mathrm{int}$, we define $B_N$ to be a submanifold of $\Int N_B[k_u]$ homeomorphic to $\partial_- B \times (0,1)$ obtained by cutting $\Int N_B[k_u]$ along the embedding of one of the boundary components of $\tilde B$ whose projection in $M$ does not lie on the boundary of $M_\mathrm{int}$, which is defined using the least area map in the same way as above.

Suppose that two bricks $B^1$ and $B^2$ share a joint $F$.
We can assume $F$ is a component of $\partial_+B^1$ and $\partial _- B^2$ by interchanging $B^1$ and $B^2$ if necessary.
Construct the least-area realisations $B^1_N$ and $B^2_N$ as above.
Then both of their boundaries contain a least area surface corresponding to $F$ as components.
We denote by $F^j$ the one contained in $\partial B^j_N$ for $j=1,2$.
Since the projections of $F^1$ and $F^2$ in $(\Int N[k_u], m_N)$ are least-area surfaces homotopic to $f_3(F)$ (which might not be embeddings), they must coincide.
Therefore, $F^1$ is isometric to $F^2$.
Then we can consider the hyperbolic $3$-manifold homeomorphic to $(\Int B^1 \cup \Int B^2 \cup F) \cap \Int M[k_u]$ by pasting $B^1_N$ and $B^2_N$ along $F^1$ and $F^2$ by an isometry.

Repeating this procedure for every joint on $B^1$ and $B^2$, then again for all the  bricks, we get a hyperbolic 3-manifold $N'[k_u]$ homeomorphic to $\Int M[k_u]$.
We denote the homeomorphism obtained by identifying $B \cap \Int M[k_u]$ with $B_N$ in $N'[k_u]$ by $h: \Int M[k_u] \rightarrow N'[k_u]$.
We shall show that this manifold is isometric to $(\Int N[k_u], m_N)$.

\begin{claim}
\label{isometric N}
There is an isometry $f' : N'[k_u] \rightarrow (\Int N[k_u], m_N)$, whose restriction to $B_N$ for each brick $B$ is an isometric embedding homotopic to $f_3 \circ f_B^{-1}$ in $N_0$.
\end{claim}
\begin{proof}
For each brick $B$, by Claim \ref{extended embedding}, there is an embedding $h_3: \tilde B \cap \Int M_B[k_u]\rightarrow \Int N_B[k_u]$ homotopic to $\tilde f_3|B \cap \Int M_B[k_u]$.
If we  lift $f_B^{-1}(B_N)$ to $\Int M_B[k_u]$, and embed it by $h_3$ into $\Int N_B[k_u]$, then the map is isometric by our definition of the metric on $N'[k_u]$.
By projecting it to $N[k_u]$, we get a locally isometric map from $B_N$, which was defined above and is bounded by least area surfaces, into $\Int N[k_u]$.
Since for two bricks sharing a joint, such maps induce the same map on the joint, we can glue this map at joints and get a local isometry $f' : N'[k_u] \rightarrow \Int N[k_u]$.
(Note that if two bricks share a joint, then their images by $h_3$ lie on the opposite sides of the image of the joint by our way of extending $h_3$ in Claim \ref{extended embedding}, which guarantees that the map is also local isometry at joints.)
Since $h_3$ is homotopic to $f_3|B$, we see that $f' \circ h_3$ is homotopic to $f_3$.

Since $f'$ induces an isomorphism between the fundamental groups, to show that it is an isometry, it is sufficient to show that $f'$ is proper.
Suppose, seeking a contradiction, that $f'$ is not proper.
Then, there exists a sequence of distinct bricks $B^i$ and points $x_i \in B^i_N$ such that $f'(x_i)$ converges in $\Int N[k_u]$.
Since $f'(x_i)$ converges, the injectivity radius at $f'(x_i)$ is bounded below by a positive constant independent of $i$, hence so is the injectivity radius at $x_i$.
We divide our argument depending on the distance between $x_i$ and $\partial B^i_N$ is bounded or not.

First we consider the case when the distance from $x_i$ and $\partial B^i_N$ is bounded as $i \rightarrow \infty$.
Let $F^i$ be a least-area boundary component of $B^i_N$ from which $x_i$ is within uniformly bounded distance.
Since $\xi(F^i) \leq \xi(S)$, the diameter of the thick part of $F^i$ is uniformly bounded.
Since $x_i$ lies in the thick part, it is within uniformly bounded distance from either an $\varepsilon_1$-Margulis tube or an $\varepsilon$-cusp neighbourhood touching $F^i$ which corresponds to a component $\tilde V^i$ of $\tilde T[k_u]$.
We denote by $V^i$   a component of $T[k_u]$ which is the projection of $\tilde V^i$.

We can show that in $(\Int N[k_u], m_N)$, for each component $V$ of $T[k_u]$ there are only finitely many images of joints by $f'$ touching $V$ as follows.
For any $R>0$, there is a finitely many components of $T[k_u]$ and original cusp neighbourhoods of $N$ which can be reached from $V$ within the distance $R$ modulo the $\epsilon_0$-thin part.
Since joints are subsurfaces of $S$ and the boundaries of their images in $N[k_u]$ lie in $T[k_u] \cup \partial N_0$, there are only finitely many possibilities for the boundaries of their images in $N[k_u]$.
This implies there are only finitely many joints up to homotopy whose images can touch $V$ since there are at most two kinds of homotopy classes of horizontal surfaces if we fix a boundary.
Since no two distinct joints are homotopic as we removed inessential joints, it follows that there are only finitely many joints whose images touch $V$.

Since our joints $F^i$ are all distinct, we can assume that all the $V^i$ are distinct by taking a subsequence.
Since $f_3$ takes the components of $\cv[k_u]$ to those of $T[k_u]$ one-to-one, and no other part is mapped to $T[k_u]$, we see that $f'$ takes the $V^i$ to distinct components of $T[k_u]$.
Therefore $f'(x_i)$ is within bounded distance from infinitely many distinct components of $T[k_u]$.
Since the $f'(x_i)$ are assumed to  converge, this contradicts the fact that there are only finitely many components of $T[k_u]$ within a bounded distance.

Thus, it only remains to consider the case when the distance from $x_i$ to the boundary of $B^i_N$ goes to $\infty$ as $i \rightarrow \infty$.
Recall that $B^i_N$ was originally a submanifold in $\Int N_{B^i}[k_u]$.
Therefore, we can regard  $x_i$  also  as a point in $\Int N_{B^i}[k_u]$.
Since $B^i_N$ is bounded by least-area surfaces, it is contained in the convex core of $(\Int N_{B^i}[k_u], m_{B^i})$.
Therefore, there is a pleated surface $k_i : \partial_- B^i \rightarrow \Int N_{B^i}[k_u]$ which is within bounded distance from $x_i$ and is homotopic to the inclusion of $\partial_-B^i$ as $\partial_- B^i \times \{t\}$ with respect the parametrisation $N_{B^i} \cong \partial_- B^i \times (0,1)$.
Since the distance from $x_i$ to $\partial B^i_N$ goes to $\infty$, we can assume that the image of $k_i$ is contained in $B^i_N$.
Hence we can regard $k_i$ as a pleated surface in $N'[k_u]$.
Also since the cuspidal part of $N'[k_u]$ consists of those of $N_0$ and rank-$2$ cusps corresponding to $T[k_u]$, we can take cusp neighbourhoods which are disjoint from all the images of $k_i$.

We consider the pleated surfaces $f' \circ k_i$.
Since $f'(x_i)$ converges and $f'\circ k_i$ is disjoint from the  cusp neighbourhoods which are images of those taken above, the pleated surface $f' \circ k_i$ converges geometrically, passing to a subsequence.
This implies in particular that there are distinct $i_1, i_2$ such that $f' \circ k_{i_1}$ and $f' \circ k_{i_2}$ are properly homotopic.
Since $f'$ induces an isomorphism between the fundamental groups, it follows that $k_{i_1}$ and $k_{i_2}$ are properly homotopic.
This is a contradiction since no two horizontal surfaces of distinct bricks are properly homotopic.
(Recall that $N'[k_u]$ and $\Int M[k_u]$ are homeomorphic.)
Thus we have established that $f'$ is an isometry.
By our construction, it is evident that $f'|B_N$ is homotopic to $f_3 \circ f_B^{-1}$ in $N_0$.
\end{proof}

Thus $\Int N[k_u]$ is isometric to $N'[k_u]$ which is the union of the $B_N$ each of which is homeomorphic to $B \cap \Int M[k_u]$.
This shows that there is a homeomorphism $h: \Int M[k_u] \rightarrow N'[k_u]$ such that $f' \circ h$ is homotopic to $f_3|\Int M[k_u]$.
By setting $f_4$ to be the natural extension of $f' \circ h$ to $M[k_u]$, we get a homeomorphism as we wanted.

It only remains to show that $f_4$ extends to a homeomorphism between $M$ and $N_0$.
To show this, it suffices to show that for each component $V$ of $\cv[k_n]$, its meridian is sent to a meridian of a component of $T[k_u]$.
If $V$ is contained in some brick $B$, then this follows from Claim \ref{extended embedding}.
Since we isotoped the original brick decomposition to a new one by moving each joint $F$ to $\check F$, we see that every component of $\cv[k_u]$ is contained in some brick.

This completes the proof of Proposition \ref{disjoint embeddings}.
\end{proof}
%????????????$B$?????????u???????Cextended horizontal split surface??¢??????`??????????????D
%Recall that in the proof of the previous proposition, we moved bricks $B$ of $M$ so that $B$ is disjoint from tubes in $\cv[k_u]$.
%We now bring them back to the original position, so that $\partial_\pm B \setminus \cv$ consists of thrice-punctured spheres.
%For a joint $F$, we define  $\hat F$ to be the surface  $F \setminus \cv[k_u]$.

Having proved that $M[k_u]$ and $N[k_u]$ are homeomorphic, we shall next show that the Lipschitz map $f_3$ can be homotoped so as to embed the joints preserving the Lipschitzity.
For that, it is more convenient to consider a brick decomposition of $M[k_u]$ rather than that of $M$.
As in \S \ref{SS_block}, we define a brick of $M[k_u]$ to be a maximal union of vertically parallel horizontal leaves which are inherited from the horizontal foliation of $M$.
By the same argument as in \S \ref{SS_block}, we can check the conditions A-(1)-(5) are satisfied.
(In reality, only A-(2) and A-(3) need to be checked.)
{\em We denote this brick decomposition of $M[k_u]$ by $\ck[k_u]$.}

Before that, we shall first move $f_3$ so that it preserves the order of joints on the boundary except for parallel ones.
 Let $\mathcal F$ be the union of joints of pairs of bricks in $\ck[k_u]$.
 We introduce an equivalence relation $\sim$ in the set of  components of $\mathcal F$ such that $F_1 \sim F_2$ if they are parallel.
 By our definition of brick decomposition, there are not three distinct joints in $\mathcal F$ which are all parallel.
 Therefore each equivalence class consists of at most two joints.
 We define the reduced union of joints to be the union of joints taken one from each equivalence class, and denote it by $\hat{ \mathcal F}$.
\begin{lemma}
\label{homeo on boundary}
There is a uniform constant $K_3'$ 
%depending only on $K_3$ and $\epsilon_u$
 as follows.
%We define $\hat{\mathcal{F}}$ to be the union of all $\hat F$ for $F \in \mathcal F$.
We can homotope $f_3$ to a proper, degree-$1$  map $f_3' : M[k_u] \rightarrow N[k_u]$ with the following properties.
\begin{enumerate}[\rm(i)]
\item $f_3'$ coincides with $f_3$ outside small pairwise disjoint neighbourhoods of $\partial M[k_u]$.
\item $f_3'$ is $K_3'$-Lipschitz.
\item On each component $T$ of $\partial M[k_u]$,  distinct components of $\mathcal F \cap T$ have disjoint images under $f_3'$.
\item On each component $T$ of $\partial M[k_u]$, the restriction $f_3'|T$ maps the components of $\hat{\mathcal F} \cap T$ disjointly preserving the orientation of $\hat{\mathcal F}\cap T$ and the order of $\{F \cap T \mid F \text{ is a component of  } \hat{\mathcal F}\}$.
(When $T$ is a torus the order means the cyclic order.)
\item
For a component $F$ of $\mathcal F \setminus \hat{\mathcal F}$, let $\hat F$ be the other component of $\mathcal F$ equivalent to $F$ and contained in $\hat{\mathcal F}$.
Then $f_3'$ also preserves the order of $((\hat{\mathcal F} \setminus \hat F) \cup F)\cap T$ for any component $F$ of $\mathcal F \setminus \hat{\mathcal F}$.
\item
The order of $F\cap T$ and $\hat F \cap T$ may be reversed only when $f_3'(F)\cap f_3'(\hat F)=\emptyset$.
\item For each small $\delta>0$, there is an universal number $n_0$ such that for any component $F$ of $\mathcal F$, there are at most $n_0$ joints $F_i$ such that $f_3'(F_i \cap T)$ are within the distance $\delta$ from $f_3'(F \cap T)$.
%a $K_3'$-bi-Lipschitz homeomorphism to $\partial T[k_u]$.
\end{enumerate}
\end{lemma}

%Before starting the proof of Lemma \ref{homeo on boundary}, we need the following lemma.

%\begin{lemma}
%\label{quasi-geodesic tori}
%For any $\delta$, there exists $\epsilon(\delta)$ which goes to $0$ as $\delta \rightarrow 0$ with the following properties.
%Let $N$ be  a hyperbolic manifold $N$ and $T$ a boundary component $T$ of the $\varepsilon_0$-thin part.
%Then for two points $x, y \in T$ with $d_N(x,y) \leq \delta$, the distance between $x$ and $y$ with respect to the Euclidean metric on $T$ induced from $N$ is less than $\epsilon(\delta)$. 
%\end{lemma}
\begin{proof}
Let $T$ be a component of  $\partial M[k_u]$, which is either a torus or an open annulus.
%We shall first consider the case when $T$ is a torus.
%By our construction of tube unions, $T$ consists of  two vertical annuli and two horizontal annuli.
As was shown before, $T$ consists of horizontal annuli and vertical annuli, and the joints can intersect only vertical annuli.
We shall show that we can homotope $f_3|T$ to a uniformly Lipschitz map by a homotopy moving each point at a uniformly bounded distance.
We should note that $f_3|T$ is a degree-$1$ map to a boundary component $T'$ of $N[k_u]$.
%, and in the case when $T$ is a torus, the order of $T \cap \mathcal F$ is a cyclic order.
The foliation of $M$ by horizontal leaves induces  a foliation on $T$ whose leaves are parallel horizontal circles.
By our definition of the model metric, each leaf has length $\varepsilon_1$.
We can extend this foliation also to horizontal annuli so that they are also foliated by parallel circles with length $\varepsilon_1$.
%We fix a simple closed curve $\gamma$ on $\partial V$ consisting of geodesic arcs of vertical annuli and horizontal annuli transverse to leaves, taken one arc from each annulus.
We let $\gamma$ be a  simple closed geodesic with respect to the induced metric intersecting each leaf at one point when $T$ is a torus, and a geodesic ray intersecting each leaf at one point when $T$ is an open annulus.

Since $f_3$ is $K_3$-Lipschitz, the homotopy classes in $T'$ of the images of the leaves have (Euclidean) geodesic lengths bounded by $K_3 \varepsilon_1$.
We also note the their lengths are also bounded below by $\varepsilon_1$ since $T'$ lies on the boundary of an $\varepsilon_1$-Margulis tube.
%Let $\lambda$ denote this homotopy class of simple closed curves on $T$.
We first homotope $f_3|A$ to $\bar f_3$ fixing $f_3|\gamma$ so that for each leaf $l$ of the foliation on $A$, the simple closed curve $\bar f_3(l)$ is a closed geodesic with respect to the Euclidean metric on $T'$.
If there are distinct components of $\mathcal F \cap T$ which have the same images, we can perturb them to be disjoint moving them within a very small distance.
We can take a homotopy $H_3: A \times [0,1] \rightarrow T'$ from $f_3$ to $\bar f_3$ as a $\bar K_3$-Lipschitz map, where $\bar K_3$ depends only on $\varepsilon_1$ and $K$, since the length of each closed curve $f_3(l)$ is between $\varepsilon_1$ and $K \varepsilon_1$ and the perturbation moves the images at uniformly bounded distances.

Now, the map from $\gamma$ to $f_3(\gamma)=\bar f_3(\gamma)$ may not proceed in the positive direction monotonously.
(As we shall see below, since $f_3|T$ has degree $1$, the orientations of $T$ and $T'$ determine the positive direction to which $\bar f_3(\gamma)$ should proceed.)
This may cause a permutation of the order of $\hat{\mathcal F} \cap T$ by $\bar f_3$.
%The distinction between the positive and the negative directions in $f_3|\gamma$ makes sense since $f_3|T$ has degree $1$ and all the  leaves are embedded on $T'$ are homotopic to each other.
We fix an orientation of the foliation on $T$, which, together with the orientation of $T$, induces a transverse orientation of the leaves and an orientation of $\gamma$.
This also defines a transverse orientation of the foliation on $T'$ induced by the closed geodesics which are images of the leaves on $T$, since $f_3'|T$ has degree $1$.
%%remove one from each parallel pair.
We number the simple closed curves constituting $\hat{\mathcal F} \cap T$  as $F_1, F_2, \dots $ in accordance with the order determined by the orientation of $\gamma$.
In the case when $T$ is a torus, we fix a leaf on the lower horizontal annulus, and let its intersection  with $\gamma$, which we denote by $a_0$, be the starting point.
The transverse orientation of the foliation on $T'$ gives an order on the images $\bar f_3(F_1 \cap T), \dots $, which may be different from the order on $T$.
(We allow some of them go beyond $\bar f_3(a_0)$ in the negative direction.
As long as $\bar f_3(\gamma)$ moves in the negative direction, we regard it as receding with respect to the order on $T'$.)
Let $\sigma$ be a permutation such that $\bar f_3(F_{\sigma(1)}), \dots$ is the right order on $T'$, in other words $F_i$ is mapped to the $\sigma^{-1}(i)$-th curve with respect to the order on $T'$.
Now, we first look at $\bar f_3(F_1 \cap T)$ which is the $\sigma^{-1}(1)$-th curve on $T'$, and consider the curves $\bar f_3(F_{\sigma(1)} \cap T), \dots , \bar f_3(F_{\sigma(\sigma^{-1}(1)-1)} \cap T)$ which are those situated before $\bar f_3(F_1 \cap T)$ on $T'$.
Set $j=\max\{\sigma(1), \dots , \sigma(\sigma^{-1}(1)-1)\}$.
We shall consider to move $\bar f_3(F_1 \cap T) , \dots , \bar f_3(F_j \cap T)$ to correct their order.
The point in the following argument is that this can be done by a homotopy with  bounded displacement.

% %(The cyclic order does not make sense if there are only two components of $\mathcal F$ intersecting $T$.
%%In this case, the order is always regarded as preserved, and we can see that the condition (iv) of our lemma holds since the length of the geodesic homotopic to $\bar f_3(\gamma)$ on $T'$ is bounded uniformly from below by a positive constant.)
%%We call a block of cyclic permutation $F_{i_1}, \dots , F_{i_p}$  {\em maximal} if there is no other block of cyclic permutation $F_{j_1}, \dots , F_{j_q}$ with $j_1 < i_1 < i_p < j_q$.
%%The components of $\mathcal F \cap T$ are decomposed into disjoint maximal blocks of cyclic permutations if we allow a block  to consist of single component.

Using the theory of Freedman-Hass-Scott \cite{fhs}, we shall bound uniformly the distance between any two of $\bar f_3(F_{1} \cap T), \dots , \bar f_3(F_j \cap T)$.
Let $k$ be a number among $2, \dots , j$.
First consider the case when $\bar f_3(F_k\cap T)$ comes before $\bar f_3(F_1 \cap T)$ on $T'$. 
Recall that $f_3$ is homotopic in $M[k_u]$ to a homeomorphism $f_4: M[k_u] \rightarrow N[k_u]$.
By the same procedure as we used to construct $\bar f_3$ from $f_3$, we can assume that $f_4$ also maps each leaf on $T$ to a closed geodesic with respect to the induced Euclidean metric on $T'$.
Then, since both $F_1$ and $F_k$ are incompressible, by the theory of  Freedman-Hass-Scott, we can homotope $\bar f_3|F_1$ and $\bar f_3| F_k$ fixing the boundaries to embeddings $g_3^1$ and $g_3^k$ in $N[k_u]$ which are contained in  small regular neighbourhoods of $\bar f_3(F_1)$ and $\bar f_3(F_k)$ respectively.
By perturbing $g_3^1$ and $g_3^k$, we can assume that they are transverse to $f_4( F_1)$ and $f_4(F_k)$ at their interiors.
Then $(g_3^1(F_1) \cup f_4( F_1)) \cap T'$ bounds an annulus $A'_1$ which may degenerate to a circle.
When $T$ is a torus, there are two choices for $A'_1$.
We choose one which bounds a compact region with $g_3^1(F_1)$ and $f_4(F_1)$ (possibly together with other components of $\partial N[k_u]$) which is disjoint from the components which $f_4(F_1)$ does not touch.
Similarly, we define an annulus $A'_k$ for $g_3^k(F_k)$ and $f_4( F_k)$.
Since  $g_3^1(F_1 \cap T)$ comes after $g_3^k(F_k \cap T)$ whereas $f_4(F_1\cap T)$ is situated before $f_4(F_k \cap T)$, we see that $A'_1$ and $A'_k$ must intersect.
Since  $f_4(F_1)\cap f_4( F_k)=\emptyset$, both $F_1$ and $F_k$  are connected, and $F_1$ and $F_k$ are not parallel by our definition of $\hat{\mathcal F}$, we see that $g_3^1(F_1)$ and $g_3^k(F_k)$ must intersect at their interiors.
By our construction of $g_3^1$ and $g_3^k$, this implies that $\bar f_3( F_1)$ and $\bar f_3( F_k)$ also intersect at their interiors.
Next suppose that $\bar f_3(F_k\cap T)$ comes after $\bar f_3(F_1\cap T)$.
By our definition of $j$, we see that $\bar f_3(F_j\cap T)$ comes before $\bar f_3(F_1\cap T)$, hence also before $\bar f_3(F_k\cap T)$.
Since $k < j$,  the order of $F_j \cap T$ and $F_k \cap T$ is reversed under $\bar f_3$, and we can argue in the same way as above to conclude that $\bar f_3( F_j)$ and $\bar f_3( F_k)$ intersect at their interiors.

Recall that the diameters of the joints $F_1, \dots $ are uniformly bounded from above by a constant depending only on $\xi(S)$.
Since $\bar f_3$ is uniformly Lipschitz, their images $\bar f_3(F_1), \dots $ also have diameters bounded from above by a constant $\lambda$ depending only on $\xi(S)$.
This implies that for any $k =2, \dots , j$, the distance between   $\bar f_3(F_k \cap T)$ and either $\bar f_3(F_1 \cap T)$ 
%(including the case when $k=j$) 
or $\bar f_3(F_j \cap T)$ is  bounded by $2\lambda$.
Therefore the distance between any two of $\bar f_3(F_1 \cap T) , \dots , \bar f_3(F_j \cap T)$ is bounded by $4\lambda$.
Recall that $\bar f_3(F_1 \cap T), \dots , \bar f_3(F_p \cap T)$ are parallel closed geodesics on $T'$ with respect to the induced Euclidean metric.
By the uniform quasi-convexity of horoballs, we see that there is a number $\lambda_0$ depending only on $\chi(S)$ which bounds the distance between any two of $\bar f_3(F_1 \cap T), \dots , \bar f_3(F_j \cap T)$ with respect to the Euclidean metric on $T'$.
Then we can homotope $\bar f_3|T$ so that $\bar f_3(F_1 \cap T), \dots ,\bar f_3(F_j \cap T)$ lie in the right order on $T'$ and near the original position of $\bar f_3(F_{\sigma(1)} \cap T)$ so that all $\bar f_3(F_i \cap T)$ with $i >j$ come after them, without changing  the condition that every leaf is mapped to a closed geodesic preserved, by moving the image by $\bar f_3$ of  thin neighbourhoods of $F_1\cap T, \dots , F_j\cap T$ only at the distance at most $\lambda_0+1$.
The map which we get after this homotopy is also uniformly Lipschitz since the displacement of the points by the homotopy is uniformly bounded.

%

%%Having shown this, we now turn to bound the number $p$ uniformly.
%%By the above argument, we see that all of $f_3(F_2), \dots , f_3(F_p)$ intersect $f_3(F_1)$ at their interiors.
%%Since the diameters of $F_1, \dots , F_p$ are uniformly bounded from above by some constant $L$, so are their images by $f_3$, for $f_3$ is Lipschitz.
%%Therefore in our situation, there are at least $p$ components of $\partial N[k_u]$ within the distance $2L$.
%%Since each component of $\partial N[k_u]$ bounds an $\varepsilon_1$-Margulis tube, such number $p$ is bounded from above uniformly.

%%Now, we shall consider to homotope $f_3|T$ to rectify the order of $f_3(F_1), \dots , f_3(F_p)$.
%%Note that $f_3(F_1 \cap T), \dots , f_3(F_p \cap T)$ are parallel closed geodesics.
%%Since each two of them are within distance $2L$ in $N_0$, by the quasi-convexity of horoballs, we see that there is a number $L_0$ depending only on $\chi(S)$ which bounds the distance of each two of these geodesics on $T'$.

We now forget $F_1, \dots , F_j$ and only consider $F_{j+1}, \dots $.
If $\sigma(j+1)=j+1$, we also forget $F_{j+1}$ and proceed to the first $j_0 >j$ with $\sigma(j_0) \neq j_0$.
Otherwise we let $j_0$ be $j+1$.
Regarding $\bar f_3(F_{j_0} \cap T)$ instead of $\bar f_3(F_1 \cap T)$ as the first one, we repeat the same argument.
Then we can correct the order of $\bar f_3(F_{j_0} \cap T), \dots , \bar f_3(F_{j_1}\cap T)$ for some $j_1 > j_0$ and make them come after $\bar f_3(F_{j_0-1})$ by moving $\bar f_3$ in thin neighbourhoods of $F_{j_0} \cap T, \dots , F_{j_1} \cap T$ only at the distance less than $\lambda_0+1$.
We note that we do not touch the components $F_k \cap T$ with $k < j_0$ at this stage.
We repeat the same process, and eventually we can homotope $\bar f_3|T$ to a uniformly Lipschitz map $f_3^T: T \rightarrow T'$ which preserves the order of $F_1 \cap T, \dots$ by a homotopy moving every point within the distance $\lambda_0+1$.
(To be more precise, we need to define the homotopy inductively in the case when there are infinitely many components of $\mathcal F \cap T$.)

Having moved $\bar f_3|T$ to $f_3^T$ which preserves the order of $\hat{\mathcal F} \cap T$, we now turn to consider a component $F$ of $\mathcal F \setminus \hat{\mathcal F}$.
Suppose that  $f_3^T$ does not preserve the order of $\hat{\mathcal F} \cap T$ with $\hat F$ replaced with $F$.
Then for each component $F'$ of $\hat{\mathcal F} \setminus \hat F$ such that the order between $F$ and $F'$ is reversed by $f_3^T$, we see that $F$ must intersect $F'$ by the same argument as above, and we can move $\bar f_3$ in a thin neighbourhood of $F \cap T$ within the distance $\lambda_0+1$ to correct the order.
Moreover, in the same way as above, we can correct the order of the images $F\cap T$ and $\hat F \cap T$ under $f_3^T$ by moving $f_3^T$ within the distance $\lambda_0+1$ if $\bar f_3(F)$ and $\bar f_3(\hat F)$ intersect.
We note that during this homotopy, each component of $\mathcal F$ is moved at most twice; hence the displacement is bounded independently of the number of the components of $\mathcal F$.
Thus we have shown that if we construct a uniform Lipschitz map whose restriction to $T$ is $f_3^T$, then the conditions (iii), (iv) and (v) in the statement are satisfied.
We denote  a homotopy on $T$ by $H_3'$.
%By repeating the same operation for every maximal order-reversing family of $\mathcal F \cap T$, we can homotope $f_3|T$ to $f_3^T: T \rightarrow T'$ passing only through uniform Lipschitz maps so that it preserves the order of $\mathcal F \cap T$ by moving the image by $f_3^T$ of each point of $T$ only at the distance at most $2L+1$.
%We define $H_3': \partial N[k_u] \times [0,1] \rightarrow \partial N[k_u]$ to be a homotopy which is a homotopy from $f_3|T$ to $f_3^T$ constructed as above for each component $T$ of $\partial M[k_u]$.
This homotopy $H_3'$ is uniformly Lipschitz since the homotopy only passes through uniformly Lipschitz maps and its displacement function is uniformly bounded.

We shall next show that thus homotoped $f_3^T$ satisfies the condition (iv) (with $f_3'$ in the statement replaced by $f_3^T$).
%The argument is the same as the one used for bounding $p$.
Fix some component $F$ of $\mathcal F$ and we shall bound the number of components $F'$ such that $f_3^T(F \cap T)$ and $f_3^T(F'\cap T)$ are within the distance $\delta$.
By our construction of brick decomposition of $M[k_u]$, there are at most two joints whose boundary lie on exactly the same boundary components of $M[k_u]$.
Therefore, for any natural number $\nu$ there exists $n$ such that if there are $n$ distinct joints, then there are at least  $\nu$ boundary components of $M[k_u]$ which these joints intersect.
As was shown above, the diameter of the image of each component of $\mathcal F$ under $\bar f_3$ is bounded by $\lambda$. 
Now, since there is a bound for the number of components of $\partial N[k_u]$ which can be reached from $F$ within the distance $\delta + 2\lambda$, we get $n_0$ bounding the number of components of $\mathcal F \cap T$ whose images by $\bar f_3$ are within the distance $\delta$ from $\bar f_3(F \cap T)$.
%%Let $\delta_u$ be a positive number bounding from below the lengths of Euclidean closed geodesics on the boundaries of $\varepsilon_1$-Margulis tubes whose core curves have lengths less than $\epsilon_u$.
%%%(Such a number can be taken to depend only on $\varepsilon_1$ and $\epsilon_u$.)
%%The length of $f_3(\gamma)$ is at most $K_3\length(\gamma)$.
%%Therefore, there is a $\hat K_3$-Lipschitz homotopy from $f_3|\gamma$ to a monotone $\hat K_3$-Lipschitz homeomorphism from $\gamma$ to $f_3(\gamma)$, where $\hat K_3$ depends only on $K_3$,  $\epsilon_u$ and $\varepsilon_1$.
%%We can extend this to a $\hat K_3$-Lipschitz homotopy $H_3' : \partial V \times [0,1] \rightarrow T$ such that $H_3'(\ ,t)$ takes each leaf to a (Euclidean) closed geodesic.

We shall finally show that $f_3^T$ as defined above can be extended to a uniform Lipschitz map $f_3'$.
We can take $r>0$ depending only on $\varepsilon_1$ and $\chi(S)$ such that the boundary components of $M[k_u]$ have product $r$-neighbourhoods in $M[k_u]$ which are pairwise disjoint.
Let $\mathcal N_r(T)$ denote the $r$-neighbourhood in $M[k_u]$ of a boundary component $T$ of $M[k_u]$, and we parametrise $\mathcal N_r(T)$ by $T \times [0,r]$ so that $T \times \{t\}$ is at the distance $t$ from $T$.
We modify $f_3$ only inside $\cup \mathcal N_r(T)$ to get $f_3'$.
We first define $f_3'| T\times [2r/3, r]$ to be rescaled $f_3|\mathcal N_r(T)$ so that $f_3'|\partial V \times \{2r/3\}$ is naturally identified with $f_3|T$.
Next we define $f_3'| T\times [r/3, 2r/3]$ to realise the homotopy $H_3$ so that  $f_3'|T \times \{t\}$ corresponds to $H_3(\ ,2-3t/r)$.
Finally we define $f_3'|T \times [0,r/3]$ to realise the homotopy $H_3'$, so that $f_3'| T \times \{t\}$ corresponds to $H_3'(\ , 1-3t/r)$.
Since  $H_3$ and $H_3'$ are  uniformly Lipschitz, we see that there is a uniform constant $K_3'$ such that $f_3'$ is $K_3'$-Lipschitz.
\end{proof}

\begin{lemma}
\label{uniformly homotoped}
There exists a constant  $K'$ depending only on $\xi(S)$ as follows.
Let $\mathcal F$ be the union of the joints of pairs of bricks in $\ck[k_u]$ as defined above.
%We define $\hat{\mathcal{F}}$ to be the union of all $\hat F$ for $F \in \mathcal F$.
Then, there exists a $K'$-Lipschitz homotopy $H: \mathcal{F} \times [0,1] \rightarrow N[k_u]$ fixing the boundary of $\mathcal F$ as follows.
\begin{enumerate}[\rm(i)]
\item  $H|\mathcal F \times \{0\}$ coincides with $f_3'|\mathcal F$.
\item $H(x,t)=f_3'(x)$ for every $x \in \partial M[k_u]\cap \mathcal F$.
\item $H|\mathcal F \times [1/2,1]$ is a  $C^2$-embedding.
\item For each component $F$ of $\mathcal F$, the restriction $H| F \times [1/2,1]$ is $K'$-bi-Lipschitz.
%\item The image of $H$ is disjoint from the $\epsilon$-Margulis tubes whose core curves have lengths less than $\epsilon_u$.
\end{enumerate}
\end{lemma}
\begin{proof}
Let $F$ be a component of $\mathcal F$.
Since the geodesic lengths of core curves in $\cv[0] \setminus \cv[k_u]$  are bounded below by $\epsilon_u$ by the condition (5) in \S \ref{conditions}, and $F \setminus \cv$ consists of thrice-punctured spheres, the modulus of $F$ is uniformly bounded.
%Recall that for each $F \in \mathcal F$, the surface $\check F$ is constructed from $F$ by replacing the annuli $F \cap \cv[k_u]$ with those lying on $\partial \cv[k_u]$.
%We shall denote the union of $\check F$ by $\check{\mathcal{F}}$.
%Since the moduli of the annuli on $\partial \cv[k_u]$ are uniformly bounded, we see that the modulus of $\check F$ is also uniformly bounded.
By the condition (7) in \S \ref{conditions} and our choice of $k_u$, we see that there is no essential closed curve with length less than $\epsilon_u$ in $N[k_u]$.
This shows that the map $f_3'|F$ is a uniformly bi-Lipschitz map to its image.
(We should note that $f_3'|F$ may not be injective.
The bi-Lipschitzity here means that the metric on $F$ induced from $M[k_u]$ and the one induced from $N[k_u]$ by $f_3'$ are bi-Lipschitz equivalent.)
We can approximate $f_3'|\mathcal F$ by an immersion fixing the boundary and preserving the uniform bi-Lipschitzity.
Now, by Proposition \ref{disjoint embeddings}, $f_3'|\mathcal F$ is properly homotopic to an embedding in $N[k_u]$ (without fixing the boundary).
%The main theorem of Freedman-Hass-Scott \cite{fhs} implies that $f_3'|\hat{\mathcal F}$ can be homotoped in $N[k_u]$ fixing the boundary to an embedding $g_3$ which is contained in the $\delta$-neighbourhood of $f_3'(\hat{\mathcal F})$ for any $\delta >0$.
%We shall show that there is a universal constant $\bar K$ such that $g_3$ is $\bar K$-bi-Lipschitz by an argument involving geometric limits.
   
We shall first show that  each component $ F$ of $\mathcal F$ can be homotoped fixing the boundary to a uniformly bi-Lipschitz embedding.
Suppose, seeking a contradiction, that this is not the case.
Then there exist   sequences of labelled brick manifolds $M^i$, homeomorphisms $f^i: M^i \rightarrow N^i$,  Lipschitz maps $f^i_3: M^i[k_u] \rightarrow N^i[k_u]$ corresponding to $f_3'$ constructed above, and joints $F^i$ in $M^i[k_u]$ such that an embedding $g_3^i$ as above within the  $\delta$-neighbourhood of $f^i_3(F^i)$ cannot be made $K_i$-bi-Lipschitz, with $K_i \rightarrow \infty$.
We put the superscript $i$ for all the symbols related to $M^i$ and $N^i$.
%, for instance, $\cv^i[k_u]$, $T^i[k_u]$ etc.
By taking a subsequence we can assume that all the $\hat F^i$ are homeomorphic to some fixed surface $F$.
As was shown before, by our definition of the model metric, the moduli of the $F^i$  are bounded.
Therefore, we can choose a homeomorphism $\kappa_i : F \rightarrow F^i$ so that the pullback of the metric on $F^i$ by $\kappa_i$ converges as $i \rightarrow \infty$.
%We can also assume that the annuli $k_i^{-1}(F^i \cap \cv^i[k_n])$ converge.
Take a basepoint $x$ on $F$, 
%which is outside $k_i^{-1}(F^i \cap \cv[k_n])$ for large $i$, 
and consider geometric limits of $(F_i, \kappa_i(x))$, $(M^i[k_u], \kappa_i(x))$, and $(N^i[k_u], f^i_3 \circ \kappa_i(x))$.
Since $f^i_3$ is uniformly Lipschitz, it converges to a Lipschitz map $f^\infty_3: M^\infty[k_u] \rightarrow N^\infty[k_u]$, where $M^\infty[k_u]$ and $N^\infty[k_u]$ are the geometric limits of $(M^i[k_u], \kappa_i(x))$ and $(N^i[k_u], f_3 \circ \kappa_i(x))$ respectively.
%Let $\check k_i : F \rightarrow \check F_i$ be a homeomorphism induced from $k_i$.
%Since the $F^i$ have bounded moduli and so do the annuli on $\partial \cv^i[k_u]$, the pull-back on $F$ of the metric on $\check F_i$ induced from $M^i[k_u]$ converges as $i \rightarrow \infty$.
Since the metrics induced from  the $F^i$ on $F$ are bounded, the homeomorphism $k_i$ converges to a homeomorphism $\kappa_\infty: F \rightarrow F^\infty$, where $F^\infty$ is embedded in $M^\infty[k_u]$.

As before, we can assume that both $f^i_3 \circ \kappa_i$ and $f^\infty_3 \circ \kappa_\infty$ are immersions.
By a result of Freedman-Hass-Scott as was used in the proof of Proposition \ref{disjoint embeddings}, $f^i_3 \circ \kappa_i$ is homotopic to an embedding relative to the boundary by a homotopy passing through only relatively compact components of $N^i[k_u] \setminus f_3^i \circ \kappa_i(F)$.
%These relatively compact components must be either open solid tori or open 3-balls as we can see by the additivity of the Euler number.
%Le numero du th?or?me clamant ceci doit ?tre ajout?.
Since $N[k_u]$ contains no Margulis tubes whose core curves have lengths less than $\epsilon_u$, these components have uniformly bounded diameters and   converge geometrically to relatively compact components of $N^\infty[k_u] \setminus f_3^\infty \circ \kappa_\infty(F)$ through which $f_3^\infty \circ \kappa_\infty$ can be homotoped to an embedding (after a perturbation if necessary).
Therefore, the geometric limit $f_3^\infty \circ \kappa_\infty$ can be homotoped to a bi-Lipschitz embedding in $N^\infty[k_u]$.
By pulling back this embedding and a homotopy, we can homotope $f_3^i \circ \kappa_i$ to a uniformly bi-Lipschitz embedding, contradicting our assumption.
Thus we have shown that $f_3'|F$ can be homotoped to a uniformly bi-Lipschitz embedding, which we shall let be $H(\ ,3/4)$.
The above argument also shows that we can choose a homotopy $H$ between $H(\ , 3/4)$ to $H(\ , 0)=f_3'|F$ to be uniformly Lipschitz.

%We can argue in the same way to show that $f_3|\hat{\mathcal F}$ can be homotoped to a uniformly bi-Lipschitz embedding using the fact that for each component $\hat F$, its image $f_3(\hat F)$ can intersects only finitely many $f_3(\hat F')$ for $F' \in \mathcal F$ which follows from the property that $f_3$ is proper.
Since $f_3'$ preserves the order of $\hat{\mathcal F} \cap T$ as was shown in Lemma \ref{homeo on boundary}-(iii), $f_3'$ is homotopic to a homeomorphism from $M[k_u]$ to $N[k_u]$ fixing $\hat{\mathcal F} \cap T$. 
Therefore  the least area surfaces homotopic to the restrictions of $f_3'$ to  the components of $\hat{\mathcal F}$ fixing the boundary must be pairwise disjoint.
The same holds even if we put  $F$ of $\mathcal{F} \setminus \hat{\mathcal F}$ into $\hat{\mathcal F}$ removing its counterpart  $\hat F$ instead.
Therefore, to show the disjointness of  the least-area images of the components of $\mathcal F$, it suffices to show that the least area surfaces homotopic to $f_3'(F)$ and $f_3'(\hat F)$ are disjoint for each component $F$ of $\mathcal F \setminus \hat{\mathcal F}$.
This follows immediately from Freedman-Hass-Scott when $f_3'(F)$ and $f_3'(\hat F)$ are already disjoint.
If $f_3'(F)$ and $f_3'(\hat F)$ intersect, then the condition (vi) of Lemma \ref{homeo on boundary} implies that the order of $F \cap T$ and $\hat F \cap T$ is preserved under $f_3'|T$.
Therefore, by considering $\hat{\mathcal F} \cup F$ instead of $\hat{\mathcal F}$, we see that the least area surfaces are disjoint.

It remains to show that we can take disjoint regular neighbourhoods of the components.
(Since the restriction of $H(\ ,3/4)$ to each component of $\mathcal F$ is uniformly bi-Lipschitz, the uniform bi-Lipschitzity on $\mathcal F \times [1/2,1]$ follows immediately once we prove that we can take regular neighbourhoods to be disjoint.
Combined with the fact shown above that a homotopy between $f_3'$ and $H(\ , 3/4)$ can be made uniformly Lipschitz, the uniform Lipschitzity of $H$ also follows.)
Recall that by Lemma \ref{homeo on boundary}-(vii),  we can assume that there is a uniform positive lower bound for the distances between the images of distinct boundary components of $\mathcal F$ under $f_3'$, hence also under $H(\ , 3/4)$.
To get disjoint regular neighbourhoods, what we need is a lower bound for the distances between the images of distinct components of $\mathcal F$ under $H(\ , 3/4)$, not only for their boundaries but for the entire surfaces.
Suppose that such a lower bound does not exist.
Then by taking a geometric limit, we get two minimal surfaces which are tangent to each other at their interiors.
This contradicts the maximal principle of minimal surfaces.
Thus, we have shown that there is a lower bound, and that we can take disjoint regular neighbourhoods.
\end{proof}

\subsection{Topological ordering of joints}
\label{topological order}
Next we shall show thus obtained embedding $H(\ , 1): \mathcal F \rightarrow N[k_u]$ preserves the topological order of joints.

\begin{lemma}
\label{non-homotopic joints}
Let $F_1$ and $F_2$ be joints in $\mathcal F$ such that $\iota_M(F_1)$ and $\iota_M(F_2)$ are not homotopic in $S \times (0,1)$.
If $F_1 \pretop F_2$,
 then $H(F_1,1) \pretop H(F_2,1)$.
\end{lemma}
\begin{proof}
Suppose that $F_1 \pretop F_2$ whereas $\iota_M(F_1)$ is not homotopic to $\iota_M(F_2)$.
%We first consider the case when $F_1$ and $F_2$ are not homotopic as subsurfaces of $S$.
%In this case, since $F_1$ and $F_2$ overlap, there is a component $c$ of $\partial F_j$ ($j=1,2$) which overlaps $\partial F_{j'}$ where $j'$ is one of $1,2$ other than $j$.
%We have only to consider the case when $c \sqsubset \partial F_2$ overlaps $F_1$ because the other case can be dealt with just by changing the direction of the ordering.
Let $c$ be a boundary component of $F_2$ which overlaps $F_1$ if there are any.
%(Such a component exists because $F_1$ and $F_2$ overlap and $\iota_M(\hat F_1)$ and $\iota_M(\hat F_2)$ are not homotopic.)
There is a component $T$ of $\partial M[k_u]$ on which $c$ lies.
Then, Lemma 3.3 in Brock-Canary-Minsky \cite{bcm} implies that $F_1 \pretop c$.
Recall from Proposition \ref{disjoint embeddings} that $f_4$ extends to a homeomorphism $\hat f_4: M \rightarrow N_0$ properly homotopic to $f$.
Since $ F_1 \pretop c$, the surface $\iota_M(F_1)$ can be homotoped to $S \times \{0\}$ without touching $\iota_M(c)$.
Since $\iota_M=\iota_N \circ f$, we see that $\iota_N \circ f ( F_1)$ can be homotoped to $S \times \{1\}$ without touching $\iota_N \circ f(c)$, which implies $\iota_N \circ f (F_1) \pretop \iota_N \circ f(c)$ by Lemma 3.18 in \cite{bcm}.
Because $f$ is properly homotopic to $\hat f_4$,
%and $\hat f_4(M\setminus V)$ is contained in $N \setminus \hat f_4(V)$ ( for $\hat f_4$ is an extension of $f_4$), 
we also have $f_4(F_1) \pretop  f_4(c)$.
Since $c$ lies on a component $T$ of $\partial M[k_u]$,  the homeomorphism $f_4$ is homotopic to $f_3'$ in $N[k_u]$, and $H$ is a proper homotopy in $N[k_u]$,  this topological order is preserved by $H(\ ,1)$, and we have $H(F_1, 1) \pretop H(c,1)$.
By the same argument by changing the direction of order, we see that for any boundary component $c'$ of $ F_1$ that overlaps $F_2$, we have $H(c',1) \pretop H(F_2,1)$.
Since $F_1$ and $F_2$ are assumed not to be homotopic, by Lemma 3.17 in \cite{bcm}, this implies that $H(F_1, 1) \pretop H(F_2,1)$. 
%
%Next suppose that $F_1$ and $F_2$ are properly homotopic.
\end{proof}

We next consider the case when two joints $F_1$ and $F_2$ are properly homotopic.

\begin{lemma}
\label{homotopic joints}
Suppose that $F_1$ and $F_2$ are joints in $\mathcal F$ such that $\iota_M(F_1)$ is homotopic to $\iota_M(F_2)$.
We further assume that  $F_1 \cup F_2$ does not bound a brick in $M[k_u]$.
If $F_1 \pretop  F_2$, then we have $H(F_1, 1) \pretop H(F_2, 1)$.
\end{lemma}
\begin{proof}
%Let $B_1$ be a brick such that $F_1 \subset \partial_- B_1$, and let $B_2$ be a brick such that $F_2 \subset \partial_+ B_2$.
%Suppose first that $B_1=B_2$.
%We assumed that $\iota_M(F_1)$ is homotopic to $\iota_M(F_2)$.
%We first consider the case when there is a component $\cv[k_u]$ intersecting only one of $F_1, F_2$.
%This is equivalent to saying  that  $\iota_M(\hat F_1)$ is {\sl not} homotopic to $\iota_M(\hat F_2)$ in our setting.
%Let $c$ be a boundary component of $F_1 \setminus \hat F_1$.
%Then either $c$ is homotopic to a boundary component of $F_2 \setminus \hat F_2$ or $c$ overlaps $\hat F_2$ and $c \pretop \hat F_2$.
%In the latter case, we have $H(c,1) \pretop H(\hat F_2, 1)$ by the same argument as in the proof of the previous lemma.
%Similarly, for any boundary component $c$ of $F_2 \setminus \hat F_2$ that is not homotopic to a boundary component of $F_1 \setminus \hat F_1$, we have $H(\hat F_1, 1) \pretop H(c,1)$.
%By Lemma 3.15 in \cite{bcm} as before, we see that this implies $H(\hat F_1,1) \pretop H(\hat F_2, 1)$.
%
%Now we assume that $\iota_M(\hat F_1)$ and $\iota_M(\hat F_2)$ are also homotopic.
Since $\iota_M(F_1)$ is homotopic to $\iota_M(F_1)$, for each component $c$ of $\partial  F_1$, there is a unique component of $c'$  of $\partial  F_2$ such that $\iota_M(c)\simeq \iota_M(c')$ in $S \times (0,1)$.
%Since both $F_1$ and $F_2$ are contained in the same brick $B_1=B_2$, this implies that $c$ and $c'$ are homotopic in $B_1=B_2$.
%We shall show that for at least one of the boundary components $c$ of $F_1$, the two curves $c$ and $c'$ are not homotopic in $M[k_u]$.
Suppose first that $c$ and $c'$ are homotopic in $M[k_u]$ for all components $c$ of $\partial F_1$.
Then $c$ and $c'$ lie on the same boundary component of $\partial M[k_u]$ by the condition A-(2) in \S\ref{SS_Conditions} and the definition of $\cv[k_u]$.
Since this holds for every boundary component of $F_1$, we see that  $F_1 \cup F_2$ bounds a submanifold $W$ in $M[k_u]$.
If $W$ is homeomorphic to $F_1 \times [0,1]$, then by our definition of the brick decomposition of $M[k_u]$, we see that $W$ consists of only one brick.
This contradicts our assumption that $F_1 \cup F_2$ does not bound a brick.

Therefore, $F_1 \cup F_2$ bounds a submanifold $W$ in $M[k_u]$, which  is not homeomorphic to $F_1 \times [0,1]$.
Then there is  a component $T$ of $\partial M[k_u]$ which is contained in $W$.
We take a horizontal curve $c$ contained in $T$.
Then $c$ overlaps both $F_1$ and $F_2$ and $F_1 \pretop c \pretop F_2$.
%\hat removed
This implies that $f_4(F_1) \pretop f_4(c) \pretop  f_4(F_2)$.
Since $\iota_N \circ \hat f_4$ is homotopic to $\iota_M$, we see that $\iota_N \circ f_4(F_1)$ is homotopic to $\iota_N \circ f_4(F_2)$, and the boundary of $\iota_N \circ (F_1)$, which lies on $\cv[k_u]$,  is unknotted and unlinked.
Therefore, applying  Lemma 3.16 in \cite{bcm}, we have $f_4(F_1) \pretop f_4(F_2)$, which implies that $H(F_1, 1) \pretop H(F_2,1)$ as before.
Thus it only remains to consider the case when there is a component $c$ of $\partial F_1$ which is not homotopic to $c'$ in $M[k_u]$.

Since $\iota_M(c)$ and $\iota_M(c')$ are homotopic, and $\iota_M(c)$ and $\iota_M(c')$ are horizontal, there is an embedded annulus $A$ bounded by $\iota(c) \cup \iota(c')$ in $S \times (0,1)$.
Since $c$ and $c'$ are not homotopic, it follows that there is a boundary component $T$ of $M[k_u]$ such that $\iota_M(T)$ intersects $A$ essentially.
Take a longitude or a core curve $c''$ of $T$. 
Then we have  $F_1 \pretop c'' \pretop  F_2$.
Now since $f_4$ is a homeomorphism from $M$ to $N_0$, we see that $f_4(F_1) \pretop f_4(c'') \pretop f_4(F_2)$, and as before, we have $H(F_1, 1) \pretop H(c'', 1) \pretop H(F_2,1)$.
Since $\iota_N \circ H(F_1,1)$ and $\iota_N \circ H(F_2,1)$ are homotopic, Lemma 3.16 in \cite{bcm} again implies that $H(F_1,1) \pretop H(F_2,1)$.
Thus we have completed the proof.
%Consider the set of bricks $B$ with either $\iota_M(F_1) \pretop \iota_M(\partial_+B) \pretop \iota_M(F_2)$ or $\iota_M, and denote it by $\mathcal B$.
%(Since $F_1$ and $F_2$ are properly homotopic, there is a region in $S \times \reals$ bounded by $\iota_M(F_1)$ and $\iota_M(F_2)$ together with annuli on $\iota_M(\partial M)$.)
%By definition, it is evident that $B_1, B_2 \in \mathcal B$.
%We divide our argument into two cases, depending on whether there is a brick in $\mathcal B$ other than $B_1, B_2$.
%
%First, suppose that $\mathcal B=\{B_1, B_2\}$.
%By Assumption \ref{modification}, we assume that there is no inessential joint.
%Therefore, $\partial_+B_1 = \partial_- B_2$ is not properly homotopic to $F_1$.
%This implies that there is a torus component 
\end{proof}

The remaining case is when $F_1$ and $F_2$ are homotopic in $M[k_u]$ and cobound a brick in $M[k_u]$.
Let $B$ be a brick bounded by $F_1 \cup F_2$, and $h(B)$ the hierarchy on $B$ which we obtain by applying  Lemma \ref{hierarchy} to $M[k_u]$.
We say that a tight geodesic $g \in h(B)$ is {\em deep-seated} if there is a component of $\Fr D(g)$ whose corresponding tube in $\cv$ is disjoint from either $\partial_+B$ or $\partial_-B$.
In the case when $D(g)$ is an annulus, we regard a core curve of $D(g)$ as a component of $\Fr D(g)$.

We shall first show that $h(B)$ cannot have a long deep-seated geodesic.

\begin{lemma}
\label{no long deep-seated}
There exists a constant  $A$ depending only on $\xi(S)$ as follows.
Let $B$ be a brick in $M[k_u]_\mathrm{int}$.
% such that both $\partial_- B$ and $\partial_+B$ are joints.
Then every deep-seated geodesic in $h(B)$ has length less than $A$.
\end{lemma}
\begin{proof}
By Theorem 9.1 in \cite{mi2}, we can take a constant $A$ such that if $g \in h(B)$ has length at least $A$, then for every component $c$ of $\Fr D(g)$, either $c$ lies on $\partial M$ or a boundary component $\partial V$ for $V \in \cv$ such that  $|\omega_M(V)|> k_u$.
(Since we are considering geodesics in $h(B)$ whose geodesics consist of simplices on the curve complex of $\cc(\partial_-B)$, we can apply Minsky's result on  Kleinian surface groups.)
Therefore every component of $\Fr D(g)$ lies on $\partial M[k_u]$ in this situation.
If  $g$ is deep-seated, then some component $c$ of  $\Fr D(g)$ must lie on $\partial V$ which is disjoint either from $\partial_+B$ or $\partial_- B$.
Thus we see that if $h(B)$ has deep-seated geodesic with length at least $A$, then there is a component of $\partial M[k_u]$ which intersects $B$ but not at least one of $\partial_-B$ and $\partial_+B$.
This contradicts the assumption that $\partial_-B$ and $\partial_+B$ are homotopic and bound $B$ in $M[k_u]$.
\end{proof}

Suppose that all the deep-seated geodesics in $h(B)$ have length less than $A$.
We further divide our argument into two cases: the first is when the number of blocks constituting $B$ is large and the other is when it is small.
It will turn out later that we do not need to show that the topological order is preserved in the latter case for the proof of Theorem \ref{blm}.

\begin{lemma}
\label{long main geodesic}
Fix a constant $A$ as in Lemma \ref{no long deep-seated}.
%Suppose that $\partial_-B$ and $\partial_+ B$ are joints and that every component of $\cv[k_u]$ intersecting one of $\partial_+B$ and $\partial_-B$ also intersects the other.
There exists a constant $C$ depending only on $\xi(S)$ (and $A$) as follows.
If $|h(B)|> C$, then we have $H(\partial_- B, 1) \pretop H(\partial_+ B,1)$.
(Recall that $|h(B)|$ denotes the sum of the lengths of all geodesics constituting $h(B)$.)
\end{lemma}
\begin{proof}
We can take a constant $C$ so that if $|h(B)|>C$, then there must be a geodesic $g$ in $h(B)$ whose length is greater than $A$.
By Lemma \ref{no long deep-seated}, then $g$ cannot be deep-seated.
If $g$ is not deep-seated, then since every frontier  component of $D(g)$ lies in $\partial M[k_u]$, the only possibility is $g$ is the main geodesic of $h(B)$.
%If there is a geodesic $g$ in $h(B)$ with length greater than $A$ which is not the main geodesic, as was shown in the proof of Lemma \ref{bounding a brick}, every boundary component of $D(g)$ corresponds to a tube in $\cv[k_u]$.
%Since we are considering the case when every component of $\cv[k_u]$ touching $B$ intersects both $\partial_-B$ and $\partial_+B$, we see that $D(g) \times [0,1]$ is a component $B_j$ of $B \setminus \Int \cv[k_u]$.
%It follows that the geodesics in $h(B)$ that are subordinate to $g$ constitute a hierarchy on $D(g)$, which we denote by $h(g)$.
%Then we can apply Theorem 7.1 in \cite{bcm} to our hierarchy $h(g)$.
%The argument  as the case 1b in the proof of Lemma 8.4 in \cite{bcm} implies that $f_4(\partial_- B_j) \pretop f_4(\partial_+ B_j)$.
%On the other hand, by Lemma 3.14 in \cite{bcm}, $f_4(\widehat{\partial_-B})$ and $f_4(\widehat{\partial_+B})$ are topologically ordered.
%The only possibility is $f_4(\widehat{\partial_-B}) \pretop f_4(\widehat{\partial_+B})$ then because of the topological order of  $f_4(\partial_- B_j)$ and $f_4(\partial_+ B)$.
%This implies that $H(\widehat{\partial_-B},1) \pretop H(\widehat{\partial_+B},1)$ as before.
Then we can apply Theorem 7.1 in \cite{bcm} to our hierarchy $h(B)$.
The same argument  as in the case 1b of the proof of Lemma 8.4 in \cite{bcm} implies that $H(\partial_-B,1) \pretop H(\partial_+B,1)$.
%If only the main geodesic $g_B$ has length greater than $A$, we can use  the argument of Brock-Canary-Minsky cited above directly to $h(B)$, and we get the topological order $H(\widehat{\partial_-B},1) \pretop H(\widehat{\partial_+B},1)$.
\end{proof}

%The case for which we do not know as yet if the right topological order of the images of the joints holds is when $\partial_-B$ and $\partial_+B$ are joints, every component of $\cv[k_u]$ that intersects one of the fronts also intersects the other, and $|h(B)|< C$.
For the remaining case, we make the following definition.
We say that a brick $B$ in $M[k_u]_\mathrm{int}$ is {\em short} in the remaining case: \ie if 
%$B$ satisfies the following:
%\begin{enumerate}
%\item
%Their fronts $\partial_- B$ and $\partial_+B$ are homotopic.
%%\item
%%Every component of $\cv[k_u]$ that is not disjoint from $B$ intersects both $\partial_-B$ and $\partial_+B$.
%%\item
%%There is no deep-seated geodesic in $h(B)$ with length greater than $A%$.
%\item
$|h(B)| \leq C$.

\subsection{Deformation to a bi-Lipschitz map}
\label{deforming to bi-Lipschitz}
Having obtained the results in the previous subsection, we are now in a position to show that we can further homotope $H(\ ,1)$ to make it bi-Lipschitz on the region between joints, applying  arguments of \S\S 8.2-8.4 in Brock-Canary-Minsky \cite{bcm}.
For a brick $B$ of $M[k_u]_\mathrm{int}$ which is not short, we shall construct a cut system,
following \S 4 and \S 8.2 in Brock-Canary-Minsky \cite{bcm}.
%For the reader who is familiar with the argument of \cite{bcm}, we remark that the only requirement of our cut system is to contain the slice corresponding to $\partial_-B$.
Our cut system $C_B$ is a set of slices of $h(B)$ having the following properties with a constant $d_1>5$ which will be specified later.
\begin{enumerate}
\item For a geodesic $g \in h(B)$, let $C_B|g$ denote the subset of $C_B$ consisting of slices with bottom geodesic $g$.
Then, for any geodesic $g \in h(B)$, the bottom simplices $\{v_\tau \mid \tau \in C_B|g\}$ cut $g$ into intervals whose lengths are between $d_1$ and $3d_1$.
\item Two distinct slices in $C_B|g$ cannot have the same bottom simplex.
\item For each $\tau \in C_B$ and any $(k,v)$ in $\tau$ other than the bottom one, $v$ is the first vertex of $k$.
\item For non-annular $g$, any slice in $C_B|g$ is a non-annular saturated slice.
\item For annular $g$, there is at most one slice in $C_B|g$.
%\item The first slice, $\{(g, v)| g \in h, v \text{is the first vertex of }g\}$ is contained in $C_B$. 
\end{enumerate}

%The last condition is the added requirement: the slice corresponding to the union of the joints on $\partial_-B$ should be contained in $C_B$.

\medskip
%In Lemma 2.12 in \cite{bcm}, it was proved that there exists a function $\mathcal L : \reals_+ \rightarrow \reals_+$ depending only on $\xi(S)$ such that for any hierarchy $H$ associated to a faithful discrete representation $\rho$ of $\pi_1(S)$ (preserving the parabolicity), if a geodesic $g \in H$ has length at least $\mathcal L(\epsilon)$, then for any component $c$ of $\Fr D(g)$, the length of $\rho(c)$ is less than $\epsilon$.
%
%In our situation for a brick $B$, the hierarchy $h_B$ corresponds to 

We take a constant $d_1$ so that 
for any geodesic $g$ in the hierarchy $h(B)$, if $g$ has length at least $d_1$,  the geodesic length of each component of $f_3(\partial D(g))$ in $N$ is less than $\epsilon_u$. 
(Such  $d_1$ exists by Lemma \ref{large k}.)
Furthermore, we consider the constant $C$ which appeared in Lemma \ref{long main geodesic} depending only on $S$.
%, and denote it $C(S)$.
%We define $C'$ to be $\max\{C(F) | F \text{ is an essential subsurface of } S \text{ including } S \text{ itself}\}$.
We can take $d_1$ which is greater than $C$.

For each slice $\tau$ in $C_B$, we define  {\em extended split level surfaces} as below following \cite{bcm}.
Suppose that the bottom pair $(g_\tau, v_\tau)$ of $\tau$ is not supported in an annulus.
Since $\tau$ is a non-annular saturated slice and $h(B)$ is $4$-complete,  $\base(\tau)$ defines a pants decomposition of $D(g_\tau)$.
For each pair of pants $Y$ constituting the pants decomposition, there is horizontal boundary of two adjacent blocks \`{a} la Minsky in the form of $Y \times \{t\}$ with respect to the parametrisation of $S \times (0,1)$, along which the two are glued.
(This lies at the middle of a block of the form $\Sg_{0,3} \times J$ in our block decomposition in \S 3.)
This horizontal surface is denoted by $F_Y$.
We consider the union $F_\tau=\cup F_Y$ for all $Y$ constituting the pants decomposition, and call it the split level surface corresponding to $\tau$.
%For each solid torus $V$ in $\cv \setminus \cv[k_u]$, the intersection $F_\tau \cap V$ consists of two disjoint parallel simple closed curves on $\partial V$.
%By connecting these two simple closed curves by a horizontal annulus embedded in $V$ for each $V$ in $\cv \setminus \cv[k_u]$, we get the extended horizontal split level surface corresponding to $\tau$, which we denote by $\hat F_\tau$.
For a cut system $C_B$ as above, the split level surface $F_\tau$ for $\tau \in C_B$ is called a {\em cut} in $B$.
%We denote by $\check F_+(B)$   the union of $\check F[0]$, defined in \S\ref{check defined}, for the joints $F$ on $\partial_+B$.
\smallskip
Let $\mathcal F_B$ be the union of $F_\tau$ for all $\tau \in C_B$ and $\mathcal F_b$ the union of $\mathcal F_B$ for all bricks $B \in \ck[k_u]_\mathrm{int}$.
Let $V$ be a component of $\cv$ on which a boundary component of $F_\tau$ lies.
By the condition (1) of the definitions of $C_B$ and $d_1$, we see that $\omega(V) > k_u$, hence $V \in \cv[k_u]$.
Therefore, by adding $\mathcal F_b$ to the joints of $M[k_u]$, we get a subdivision of $M[k_u]$ into smaller bricks, which may have inessential joints.
We denote this refined brick manifold by $M'[k_u]$.
(Note $M[k_u]$ and $M'[k_u]$ are the same as manifolds, only their brick structures differ.)
%This brick decomposition is our $\cv'$, for which we have proved the preservation of topological order of joints except for short bricks.

We shall show that $H(\ ,1)$ can be homotoped so that each $F_\tau$ for $\tau \in C_B$ is a  uniform bi-Lipschitz embedding.
\begin{lemma}
\label{embedding cuts}
There exists a constant $K''$ depending only on $\xi(S)$ as follows.
There exists a $K''$-Lipschitz homotopy $H': (\mathcal F_b \cup \mathcal F) \times [0,1] \rightarrow N[k_u]$, such that 
\begin{enumerate}[\rm(i)]
\item $H'|(\mathcal F_b \cup \mathcal F) \times \{0\}$ coincides with $H(\mathcal F_b \cup \mathcal F,1)$.
\item $H' | (\mathcal F_b \cup \mathcal F) \times [1/2,1]$ is a $K''$-bi-Lipschitz $C^2$-embedding.
\end{enumerate}
\end{lemma}
\begin{proof}
Our argument is similar to the proof of Lemma \ref{uniformly homotoped}.
Let $T$ be a component of $\partial M[k_u]$ intersecting $B$ and $T'$ its image in $N[k_u]$ under $f_4$.
We first need to show that the $H(\ ,1)$ can be moved to a uniformly Lipschitz map which preserves the order of $T \cap (\mathcal F_b \cup \mathcal F)$  except for the fronts of short bricks by a homotopy whose displacement function is bounded from above by a uniform constant.
Our situation is a little different from that of Lemma \ref{homeo on boundary} since among our surfaces in $\mathcal F_b \cup \mathcal F$, there might be more than two components which are all homotopic to each other.
Still, we can argue as in the proof of (vi) in Lemma \ref{homeo on boundary}, and see that the order of components $F, F'$ of $\mathcal F_b \cup \mathcal F$ can be reversed only when they are homotopic in $M[k_u]$ and  $H(F,1)\cap H(F',1) =\emptyset$.
Now, by applying Lemma \ref{long main geodesic} to our refined brick manifold $M'[k_u]$, we see that the order between $F \cap T$ and $F' \cap T$ is preserved by $H(\ ,1)$ unless $F\cup F'$ bounds a short brick $B$ in $M'[k_u]$, which must be also a short brick of $M[k_u]$ since we did not introduce new short bricks in our subdivision of bricks.
Thus we have shown that $H(\ , 1)$ can be homotoped with  uniformly bounded displacement of points so that the order of $(\mathcal F_b \cup \mathcal F) \cap T$ is preserved for any component $T$ of $\partial M[k_u]$ except for the order between the two fronts of the same short bricks.
Let $f_3''$ be thus obtained uniform Lipschitz map from $M'[k_u]$ to $N[k_u]$.

Next we shall show that the same property as (vii) in Lemma \ref{homeo on boundary} holds for $\mathcal F_b$ and $f_3''$; that is, 
for any $\delta$, there is a number $n_0$ bounding the number of components of $f_3''(\mathcal F_b \cap T)$ which are within distance $\delta$ from $f_3''(F \cap T)$ for any component $F$ of $\mathcal F_b$.
Let $F_1, \dots ,F_n$ be distinct components of $\mathcal F_b$ such that $f_3''(F_1 \cap T), \dots , f_3''(F_n \cap T)$ are within the distance $\delta$ from $f_3''(F \cap T)$.
Then $H(F_1,1), \dots , H(F_n,1)$ are within distance $3\lambda_0+\delta$ from $H(F,1)$, where $\lambda_0$ is the constant which we defined in the proof of Lemma \ref{homeo on boundary}.
Recall that for each slice $\tau$ of $C_B$, each component of $F_\tau \setminus \cv$ is a thrice-punctured sphere.
By Lemma \ref{no homotopic tubes}, for distinct slices $F_{\tau_1}, \dots, F_{\tau_n}$, there are at least $\nu$ non-homotopic tubes in $\cv$ which at least one of $F_{\tau_1}, \dots , F_{\tau_n}$ intersects with $\nu$ going to $\infty$ as $n \rightarrow \infty$.
By Lemma \ref{upper bound}, each tube has a core curve with length less than $L$.
Since $H(\ , 1)$ is uniformly Lipschitz, the lengths of the images of the core curves are universally  bounded.
Suppose that there is not a universal bound for $n$.
By the usual argument using a geometric limit of model maps for which there are at least $i$ slices as above, we are lead to a contradiction to the fact that for a hyperbolic 3-manifold, a constant $R$ and a base point $x$, there are only finitely many homotopy classes which are represented by a closed curve of length less than $L$ contained in the $R$-ball centred at $x$.
Thus we have shown that $f_3''|\mathcal F_B$ has the same property as (vii) in Lemma \ref{homeo on boundary}.
Combining this with Lemma \ref{homeo on boundary}, we see that for any $\delta$, there exists $n_0$ bounding the number of components of $f_3''((\mathcal F_b \cup \mathcal F)\cap T)$  within the distance $\delta$ from $f_3''(F\cap T)$ for any component of $F$ of $\mathcal F_b \cup \mathcal F$.
%%%the lengths of curves are uniformly bounded: geometric limit : refer to previous lemma

%Repeating this argument for each brick $B$ that is not short, 
%and using the same argument as in the latter half of  the proof of  Lemma \ref{homeo on boundary}, we can construct a uniform Lipschitz map $f_3'': M[k_u] \rightarrow N[k_u]$, coinciding with $H(\ , 1)$ outside a small neighbourhood of $\partial M[k_u]$, which preserves the order of $\hat{\mathcal F} \cup \hat{\mathcal F}_b \cap T$ for each component $T$ of $\partial M[k_u]$, such that for any $\delta>0$ the number of components of $\hat{\mathcal F} \cup \hat{\mathcal F}_b \cap T$ whose images are   within the distance $\delta$ from $f_3''(F \cap T)$ is bounded by a uniform constant..
%%%%and for any \delta.... this should hold for  \mathcal F_b...
%We consider a complete hyperbolic metric $m_N$ on $\Int N[k_u]$ which makes each component of $T[k_u]$ a torus cusp, which we defined in Proof of Proposition \ref{disjoint embeddings}.
%Since $f_4$ is homotopic to $f_3'|M[k_u]: M[k_u] \rightarrow N[k_u]$ fixing $\partial M[k_u]$, where $f_3'$ is a homeomorphism to $\partial N[k_u]$, and embeds $\hat{\mathcal F}_b \cup \hat{\mathcal F}$, a least area surface  homotopic to $f_3|(\hat{\mathcal F}_b \cup \hat{\mathcal F})$ fixing the boundary with respect to any Riemannian metric is an embedding.
%Furthermore, we can make the least area surface disjoint from the least area surface homotopic to 

%Since the only short bricks of $M'[k_u]$ are those already contained in $M[k_u]$, 
By the same argument as  Lemma \ref{uniformly homotoped}, we see that $H(\ ,1)$ can be homotoped to a uniform Lipschitz map which embeds $\mathcal F_b \cup \mathcal F$ in such a way two distinct components have disjoint $\delta$-regular neighbourhoods.
%
%
%Then the same argument as 
% in the proof of Lemma \ref{uniformly homotoped} using the least area theory shows that 
% we can homotope $f_3''$ relative to $\partial M[k_u]$ by a Lipschitz homotopy as we wanted to embed $\hat{\mathcal F} \cup \hat{\mathcal F}_b$.
   \end{proof}

%We can identify $H'(\ , 1)|\hat{\mathcal{F}}$ with $H(\ ,1)$ in Lemma \ref{uniformly homotoped} and assume that the results in \S \ref{topological order} hold also for $H'(\ ,1)| \hat{\mathcal{F}}$; for these results do not depend on the choice of the embedding $H(\ ,1)$.
%Next we shall show that $H'( , 1)$ preserves the topological orderings of the cut system $C_B$.

%\begin{lemma}
%\label{topological order of cuts}
%Suppose that $B$ is not short.
%For any $c, c' \in C_B$ such that $\hat{F}_c \pretop \hat{F}_{c'}$, we have $H'(\hat{F}_c, 1) \pretop H'(\hat{F}_{c'},1)$.
%Moreover, for any $c \in C_B$, we have $H'(\widehat{\partial_- B},1) \pretop H'(\hat{F}_c,1) \pretop H'(\widehat{\partial_+B},1)$.
%\end{lemma}
%\begin{proof}
%We consider the family of surfaces consisting of $\hat{\mathcal{F}}_B$, $\widehat{\partial_-B}$ and $\widehat{\partial_+B}$.
%We can argue in the same way as Lemmata \ref{non-homotopic joints} -\ref{long main geodesic} (or follow the original argument of \cite{bcm}) and see that the only possibility that the images of two cuts $\hat{F}_c, \hat{F}_{c'}$ do not have the right topological order is when they are homotopic and the submanifold bounded by them has less than $C'$ blocks.
%This case is excluded by our choice of $d_1$ in the case when $C_B$ is non-empty.
%If $C_B$ is empty, this condition contradicts the assumption that $B$ is not short.
%\end{proof}

By this homotopy $H'$, we can homotope $f_3'$ to $f_5$ which  is a uniform bi-Lipschitz map on each component of $\mathcal{F}_b \cup \mathcal{F}$ and embeds  its regular neighbourhood.
Recall that $f_5$ preserves the topological order of $\mathcal F_b \cup \mathcal{F}$ except for the fronts of short brick by the results in \S \ref{topological order} and Lemmata \ref{no long deep-seated}, \ref{long main geodesic}.
If $B$ is short, then $B$ consists of less than $C$ blocks, hence the diameter of $B$, which can be bounded by a linear function of the number of blocks, is bounded a constant depending only on $\xi(S)$. 
Therefore,  we can isotope $f_5(\partial_- B)$ into a regular neighbourhood of $f_5(\partial_+B)$ so that $f_5(\partial_-B) \pretop f_5(\partial_+B)$ preserving the condition on the bi-Lipschitzity.
We should note that short bricks of $M'[k_u]$ come from those of $M[k_u]$ and that by Assumption \ref{modification}, two short bricks cannot be adjacent to each other.
Therefore, we can perform this deformation for all short bricks so that $f_5(\partial_-B)$ and $f_5(\partial_+B)$ have regular neighbourhoods with uniform width.
Since the embedding of each cut by $f_5$ has a regular neighbourhood with uniform width, $f_5$ is bi-Lipschitz not only on each cut or joint but also with respect to the induced metric on the entire $\mathcal F_b \cup \mathcal F$.

To complete the proof of Theorem \ref{blm}, it remains to deform $f_5$ in the complement of $\mathcal F_b \cup \mathcal F$ in $M[k_u]_\mathrm{int}$ to make it bi-Lipschitz without changing the map on the geometrically finite bricks.
This can be done by the same argument as \S8.4 in \cite{bcm} without any modification.
Thus we have completed the proof of Theorem \ref{blm} by setting $k$ to be $k_u$.

\section{Proofs of Theorems}\label{S_Proofs}

\subsection{Geometric limits of geometrically finite bricks}\label{geom_finite}
Let $G_n$ be a Kleinian surface group, and set $N_n$ to be $\hyperbolic^3/G_n$.
Let $g_n: M_n \rightarrow (N_n)_0$ be a model map constructed in \cite{bcm} which induces a bi-Lipschitz homeomorphism $g_n[k_u]: M_n[k_u] \rightarrow N_n[k_u]$.
Suppose that $M_n$ has a geometrically finite brick $B_n \cong F_n \times [0,1)$ or $F_n \times (0,1]$.
We shall consider only the case when $B_n \cong F_n \times [0,1)$, for the other case can be dealt with in the same way just by turning everything upside down.

\begin{lemma}
\label{geometrically finite end}
Let $x_n$ be a point in $B_n$ in the above situation.
Suppose that with respect to the metric $d_{B_n}$ on $B_n$ defined in \S\ref{SS_block_metric}, we have $d_{B_n}(F_n \times \{0\}, x_n) \rightarrow \infty$.
Then the geometric limit of (a subsequence of) $\{(M_n, g_n(x_n))\}$ is elementary: \ie isomorphic to $\hyperbolic^3/\Gamma$ for an elementary Kleinian group $\Gamma$.
\end{lemma}
\begin{proof}
Let $C(N_n)$ be the convex core of $N_n=\hyperbolic^3/G_n$.
By the definition of the model map, we see that $d_{M_n}(C(M_n), g_n(x_n))\rightarrow \infty$.
Let $\Gamma$ be a Kleinian group such that $(\hyperbolic^3/\Gamma, x_\infty)$ is the Gromov limit of  $\{(M_n, g_n(x_n))\}$ after passing to a subsequence.
Suppose, seeking a contradiction, that there are non-commuting elements $g, h$ in $\Gamma$.
Then, there exist elements $g_n, h_n$ in $G_n$ such that $\lim g_n=g$ and $\lim h_n=h$.
Consider the action of $G_n$ on $\hyperbolic^3$.
Then $g_n$ and $h_n$ act on $\hyperbolic^3$ as loxodromic or parabolic transformations.
Let $l_n$ be a geodesic  in $\hyperbolic^3$ which is a common perpendicular of the axes of $g_n$ and $h_n$ if they are loxodromic, or a geodesic ray perpendicular to the axis of the loxodromic one to tending to the fixed point at infinity of the parabolic one when only one of them is loxodromic, or a geodesic connecting the fixed points at infinity of the two elements if both of them are parabolic.

We claim that the function $t(g_n, h_n)(x)=\max\{d(x, g_n(x)), d(x,h_n(x))\}$ takes minimum at a point $c_n$ on $l_n$.
This can be seen by considering sets $V_{g_n}(r)$ and $V_{h_n}(r)$ consisting of points whose translation distances are less than or equal to $r$ under $g_n$ and  $h_n$ respectively.
The smallest $r$ for which $V_{g_n}(r) \cap V_{h_n}(r) \neq \emptyset$ realises the minimum of $t(g_n, h_n)$.
(If $V_{g_n}(r)$ (resp. $V_{h_n}(r)$) reaches the axis of $h_n$ (resp. $g_n$) while $V_{h_n}(r)$ (resp. $V_{g_n}(r)$) is empty, we take such $r$ as the smallest.)
By the convexity of these sets, we see that the intersection consists of one point $c_n$, and that it lies on $l_n$.
Since $\{g_n\}$ and $\{h_n\}$ converge, the smallest $r$ is bounded from above independently of $n$.
Since the configurations of $V_{g_n}(r)$, $V_{h_n}(r)$ up to isometries are compact, we see that $|t(g_n, h_n)(y)- 2d(y,c_n)|$ is bounded from above independently of $n$.
(This follows from the fact that the displacement of a point can be approximated by twice the distance from the point to the axis if the translation length on the axis is bounded above.)

Obviously, $l_n$ is contained in the Nielsen convex hull of $G_n$.
Take a lift $\tilde x_n$ of $x_n$ which converges to a lift $\tilde x_\infty$ of $x_\infty$.
Since $d_{M_n}(C(M_n), g_n(x_n)) \rightarrow \infty$,  the distance of $l_n$ from a lift $\tilde x_n$ of $x_n$ in $\hyperbolic^3$ goes to $\infty$ as $n\rightarrow \infty$; hence $d(\tilde x_n, c_n) \rightarrow \infty$.
From the above observation, this implies that $t(g_n, h_n)(\tilde x_n) \rightarrow \infty$.
This contradicts the facts that $g=\lim g_n$ and $h=\lim h_n$ translate $\tilde x_\infty$ within a finite distance.
\end{proof}

%Let $\{M^{(m)}\}_{m=1}^\infty$ be a sequence of brick manifolds satisfying the conditions (2.a)-(2.d) and (EL), 
%and let $M^{(m)}[0]$ be a block manifold associated to $M^{(m)}$.
%Note that $\{M^{(m)}\}$ is not necessarily an expanding sequence and possibly its entries are infinite 
%brick manifolds.
%Let $\{B_m\}$ be a sequence of bricks $B_m=F_m\times [0,1)$ each of which is isometric to some geometrically finite brick of $M^{(m)}[0]$.
We now consider  geometric limits of geometrically finite bricks.
Since we are only interested in non-elementary geometric limits, by virtue of the previous lemma, we have only to consider the case when the basepoint lies on the real front along which the brick is pasted to other bricks.
Let $x_m$ be a point in $B_m$ lying on $F_m\times \{0\}$.
Since each $\hat F_m=\hat F_m\times \{0\}$ has the cylinder-$\Sg_{0,3}$ metric $\tau_m$, 
if $\{(\hat F_m,x_m)\}$ converges geometrically to $(\hat F_\infty,x_\infty)$ passing to a subsequence, then $\hat F_\infty$ also has such a metric $\tau_\infty$.
Moreover, since $B_m$ is uniformly bi-Lipschitz to the brick $\hat F_m\times [0,\infty)$ with metric $e^{2r}\tau_m +dr^2$ $(r\in [0,\infty))$, 
$\{B_m\}$ converges geometrically to a brick $B_\infty$ uniformly bi-Lipschitz to $F_\infty\times [0,\infty)$ 
with metric of the form $ e^{2r}\tau_\infty+dr^2$ $(r\in [0,\infty))$ passing to a subsequence.
In particular, $B_\infty$ is also a brick homeomorphic to $F_\infty \times [0,1)$.

\subsection{Proofs of Theorem \ref{thm_a} and Corollary \ref{cor_b}}

\begin{proof}[Proof of Theorem \ref{thm_a}]
Let $\{G_n\}$ be a sequence of Kleinian surface groups which 
converges geometrically to a non-elementary Kleinian group $G$.
Since $\{G_n\}$ converges geometrically to $G$, fixing a basepoint in $\hyperbolic^3$, and projecting it to $\hyperbolic^3/G_n$ and $\hyperbolic^3/G$ as basepoints $y_n$ and $y_\infty$, we get a geometric convergence $(\hyperbolic^3/G_n, y_n) \rightarrow (\hyperbolic^3/G, y_\infty)$.
By the original bi-Lipschitz model theorem \cite{bcm}, 
for each $n\in \nn$, there exist a model manifold $M_n$ and a model map $g_n: M_n \rightarrow (N_n)_0$ inducing 
a $K$-bi-Lipschitz homeomorphism  $g_n:M_n[k_u]\rightarrow N_n[k_u]$, where $N_n=\hh^3/G_n$.
We let $x_n$ be a point in $M_n$ which is taken to $y_n$ by $g_n$.
The model manifold $M_n$ consists of $M_n[0]$, which is decomposed into internal blocks and boundary blocks, and Margulis tubes.
Since we assumed that $G$ is non-elementary, $x_n$ cannot go deeper and deeper into Margulis tubes as $n \rightarrow \infty$.
Therefore, moving $x_n$ and $y_n$ within uniformly bounded distance without changing $G$ up to conjugacy, we can assume that $x_n$ lies in $M_n[0]$.

Since $G_n$ is a Kleinian surface group, $M_n$ is embedded in $S_0 \times (0,1)$ for a compact core $S_0$ of $S$ so that the boundary of a cusp neighbourhood which does not correspond to a boundary component of $S_0$  is a properly embedded  open annulus both of whose ends go to the same end of $S_0 \times (0,1)$, either to the $+$-direction or to the $-$-direction.
We   put $M_n$ the structure of a brick manifold compatible with the block decomposition as follows.
We first consider a proper embedding $\eta_n : M_n \rightarrow S \times (0,1)$ with the following properties, which is obtained by isotoping blocks within $S \times (0,1)$.
\begin{enumerate}
\item 
The embedding $\eta_n$ preserves the horizontal and the vertical leaves of each block.
(Here for a block with the form $\Sigma \times J$, each $\Sigma \times \{t\}$ is a horizonal leaf and $\{x\} \times J$ is a vertical leaf.)
\item Each Margulis tube in $M_n$  is mapped to $A \times [t_1,t_2]$  for some essential annulus $A$ on $S$ and $t_1 < t_2$, and each torus boundary of $M_n$ is mapped to the boundary of $A \times [t_1, t_2]$.
\item 
Each open annulus boundary component of $M_n$ is mapped to the boundary of either $A\times [t, 1)$ or $A \times (0,t]$ for an essential annulus $A$ on $S$ and $t \in (0,1)$.
\item The geometrically finite ends of $M_n$ are peripheral, \ie lie on $S \times \{0,1\}$.
\end{enumerate}
This is exactly the situation as in the construction of brick decomposition for $M_\mathrm{int}^{(1)}$ in \S \ref{SS_block}.
Therefore, we can endow a brick decomposition with $M_n$ by defining  each to be   a maximal family of parallel leaves.

We now consider the geometric limit $(M[0], x_\infty)$ of $(M_n[0], x_n)$, possibly passing to a subsequence.
Note that any internal block of  $M_n[0]$ is isometric to either $\Sg_{(0,4)}\times [0,1]$ or 
$\Sg_{(1,1)}\times [0,1]$, or $\Sg_{(0,3)}\times [0,1]$, each with a standard metric.
(We can consider a block decomposition in our sense or Minsky's. Either will do.)
Therefore a geometric limit of internal blocks can also regarded as blocks.
On the other hand, as was seen in Subsection \ref{geom_finite}, any sequence of geometrically finite bricks in $M_n[0]$ converges geometrically to a geometrically finite brick in after taking a subsequence if we put a basepoint on the real front.
Since $G$ is non-elementary, by Lemma \ref{geometrically finite end}, if the $x_n$ lie in geometrically finite bricks, we can assume that they lie on the real fronts of the bricks.
These imply  that the geometric limit $M[0]$  consists of geometrically finite bricks and the remaining part which is decomposed into blocks.
(Here we are not considering the brick decomposition of $M[0]$ yet.)

We denote by $M[0]_\mathrm{int}$ the part of $M[0]$ consisting of the limits of internal bricks.
The complement of $M[0]_\mathrm{int}$ in $M[0]$ consists of geometrically finite bricks as was seen above.
As before, we denote by $\cv_n$ the union of tubes in the tight tube unions giving a block decomposition of $M_n^0$.
(Recall that $M_n^0$ is the complement of tubes in $\cv_n$ intersecting $M_n$ along annuli and is naturally identified with $M_n$.)
For any $k$, we denote by $\cv_n[k]$ the subset of $\cv_n$ consisting of tubes $V$ with $|\omega_{M_n}(V)| \geq k$.
Recall that $M_n[k]=(M_n)^0 \setminus \cv_n[k]$.
We denote by $T_n[k]$ the union of Margulis tubes which is  the image of $\cv_n[k]$ by $g_n$. 

Each torus component $T$ of $\part M[0]$ is a geometric limit of torus components $T_n$ of $\part M_n[0]$.
Since $T_n$ converges geometrically,  either $\{\omega_{M_n}(T_n)\}$ converges or goes to $\infty$.
If it converges, then $T_n$ bounds a hyperbolic tube $V_n$ converging geometrically to a hyperbolic tube $V$ bounded by $T$.
We denote by $\cv_\infty$ the union of such tubes $V$.
The gluing map of $V_n$ to $M_n[0]$ converges to a gluing map of $V$ to $M[0]$.
We define the union of $M_n[0]$ and such tubes glued by the limit gluing maps to be $M$.
Then it follows immediately that the geometric convergence of $(M_n[0], x_n)$ to $(M[0], x_\infty)$ extends that of $(M_n, x_n)$ to $(M, \infty)$.
We denote by $M[k]$ the union of $M[0]$ and tubes in $\cv_\infty$ for which $\lim_{n \rightarrow \infty} |\omega_{M_n}(T_n)| \leq k$.
The argument above also implies in particular that $g_n$ with base point $x_n$ converges to a $K$-bi-Lipschitz homeomorphism $g: M[k_u] \rightarrow N[k_u]$.
Since we put the metric on $V_n$ inherited from a Margulis tube determined by $\omega_{M_n}(V_n)$, each $g_n$ is extended to a $K$-bi-Lipschitz map in $V_n$.
Therefore $g$ is also extended to a $K$-bi-Lipschitz homeomorphism from $M$ to $N_0$.
We use the symbol $M_\mathrm{int}$ to denote the union of $M[0]_\mathrm{int}$ and $\cv_\infty$.

If $\lim_{n\to \infty}{\omega_{M_n}(V_n)}=\infty$, then $g(T)$ is the boundary of a torus cusp neighbourhood of $N$ not contained in $N_0$.
If we put a basepoint on $\partial V_n$, then the geometric limit of $V_n$ is also a $\integers \times \integers$-cusp which is $K$-bi-Lipschitz to the cusp neighbourhood bounded by $g(T)$ since $\omega_{M_n}(V_n)$ controls the modulus of the Margulis tube bounded by $g_n(T_n)$.
(See \S8.5 in \cite{bcm}.)
Note that by our definition of $M$, the limit cusp neighbourhood is not contained in our model manifold $M$.
%Combining this with the extension of $g_n$ to solid tori described in the previous paragraph, we see that we have a  $K$-bi-Lipschitz homeomorphism $\bar g:M \rightarrow N_0$ extending $g$.
%For later use, we define $\mathcal T$ to be the union of cusp neighbourhoods as above.

The properties (ii) that $M$ is acylindrical and (i) that $\partial M$ consists of tori and annuli in the statement of Theorem A are derived from the same properties for $N_0$.
We shall next show that $M$ is a brick manifold.
Recall that $M[0]_\mathrm{int}$ admits  a decomposition into blocks.
Let $\rho_n: B_{r_n}(M_n ,x_n) \rightarrow B_{K_nr_n}(M, x_\infty)$ be a $(K_n,r_n)$-approximate isometry associated to the geometric convergence of $\{(M_n, x_n)\}$ to $(M,x_\infty)$.
We can arrange $\rho_n$ so that for each block $b$ of $M[0]_\mathrm{int}$, its pull-back $\rho_n^{-1}(b)$ is also a block with respect to the block decomposition of the brick manifold $M_n$, and $\rho_n^{-1}|b$ preserves the vertical and horizontal leaves of $b$.
Recall that the embedding $\eta_n$ of $M_n$ into $S \times (0,1)$ preserves the vertical and the horizontal leaves of blocks.
Therefore, at each point of $M$ the (two-dimensional) horizontal directions and the vertical direction are well defined.
The horizontal directions in $M$ constitute a foliation whose leaves are incompressible in $M$ and  homeomorphic to essential subsurface of $S$ (including $S$ itself) as we can see by considering their image under $\rho_n^{-1}$ for large $n$.
We define a leaf of this foliation to be a horizontal leaf of $M$.
Horizontal leaves are transversely oriented, by defining the $+$-direction of the second factor of  $S \times (0,1)$ to be the positive direction.

Now, we define a brick in $M$ to be a closed submanifold which is the closure of a maximal union of  parallel horizontal leaves in $M$ if it has non-empty interior.
It is evident that the bricks defined in this way are pairwise disjoint.
We can further show the following, which implies that $M$ is a brick manifold.
\begin{lemma}
\label{bricks}
Every point in $M$ is contained in a brick.
The bricks are locally finite.
\end{lemma}
\begin{proof}
Let $x$ be a point in $M$, and $F$ a horizontal leaf of $M$ on which $x$ lies.
We say that a boundary component $T$ of $M$ touches $F$ from above  if $T \cap F \neq \emptyset$ and if any leaf near $F$ and above $F$ intersects $T$ whereas any leaf below $F$ is disjoint from $T$.
Similarly, we define touching from below.
Every component of $\partial M$ is either a torus or an open annulus, and they can intersect a horizontal leaf along annuli.
Since  an annulus component of $\partial M_n$ has only one horizontal annulus, and torus component has only two horizontal annulus situated at  different horizontal levels.
This property is preserved by taking geometric limit, and hence if a component of $\partial M$ intersects a horizontal leaf along annuli, the intersection consists of a single annulus.
Moreover, since $M$ is acylindrical, there are no two annuli on $\partial M \cap F$ which are parallel on $F$ and are contained in distinct components of $\partial M$.
Therefore, the number of the components of $F \cap \partial M$ is uniformly bounded by a constant depending only on $\xi(S)$.

Now, recall that the height (with respect to the metric determined by blocks) of each component of $\partial M_n$ is uniformly bounded from below by a positive constant $\zeta$.
We take a positive number $\theta < \zeta$, and let $F'$ be a horizontal leaf of $M$ above $F$ at the distance $\theta$ with respect to the metric determined by blocks.
Then each component of $\partial M$ within distance $\theta$ from $F$ which does not lie below $F$ must intersect either $F$ or $F'$.
Therefore, the number of such components is bounded by a constant depending only on $\xi(S)$.
The same holds for components of $\partial M$ within distance $\theta$ from $F$ not lying above $F$.

Let $h_1$ be the minimum of the heights above $F$ (with respect to the metric on $M$ determined by blocks) of components of $\partial M$ intersecting $F$ but not touching from below, which we allow to be $\infty$.
Since there are finitely many components of $\partial M$ intersecting $F$ as was shown above, we have $h_1 >0$.
Next let $h_1'$ be the minimal distance from $F$ to the components of $\partial M$ lying above $F$, which is defined to be $\infty$ if there are no such components.
By the observation in the previous paragraph, there are only finitely many components of $\partial M$ within a fixed distance from $F$, we have $h_1' >0$.
We set $\bar h_1$ to be $\min \{h_1, h_1'\}$.
Then, if we move $F$ in the vertical direction to the positive side within the distance $\bar h_1$, then the surface $F$ may lose the interior of annuli which are intersection with components of $\partial M$ touching from above, but the intersection with $\partial M$ does not change in other ways.
Therefore all the horizontal leaves above $F$ within distance $\bar h_1$ are parallel to each other.
It follows that if $x$ lies outside the intersection with components of $\partial M$ touching $F$ from above, then $x$ is contained in a brick which passes through $F$ or is situated above $F$ and touches $F$ at the boundary.
Similarly by defining $h_2$ and $h_2'$ turning everything upside down and setting $\bar h_2$ to be $\min\{h_2, h_2'\}$, we see that all the horizontal surface below $F$ within distance $\bar h_2$ are parallel to each other.
Also, if $x$ lies outside the intersection with components of $\partial M$ touching $F$ from below, then $x$ is contained in a brick which passes through $F$ or is situated below $F$ and touches $F$ at the boundary.
Since no components of $\partial M$ can touch $F$ from both above and below, this shows that $x$ is always contained in a brick.

Furthermore, there are only finitely many bricks at distance less than $\min\{\bar h_1, \bar h_2\}$ since $F$ is contained in  the (non-empty) union of finitely many  bricks whose heights are at least $\min\{\bar h_1, \bar h_2\}$ and one of which contains $x$.
This shows the local finiteness of the bricks.
\end{proof}

By our definition of bricks in $M$ and that for $M_n$, for any brick $B$ in $M$ its pull-back $\rho_n^{-1}(B)$ is contained in one brick in $M_n$ for large $n$.
Now, we are in a position to use Lemma \ref{l_21} to verify the condition (iv) in Theorem \ref{thm_a}.
For any $r \in \naturals$, let $M(r)$ be the submanifold of $M$ consisting of bricks intersecting the $r$-ball centred at $x_\infty$ with respect to the metric induced from those on blocks.
Then $M(r)$ contains only finitely many bricks by Lemma \ref{bricks}.
If we take a sufficiently large $n$, then we can pull back $M(r)$ to $M_n$ by $\rho_n^{-1}$.
Since the pull-back of each brick is contained in a brick of $M_n$, we can embed $M(r)$ to $S \times (0,1)$ by $\eta_n \circ \rho_n^{-1}$ preserving the vertical and the horizontal leaves.
Since $M= \cup_{r=1}^\infty M(r)$, by Lemma \ref{l_21}, we can embed $M$ into $S \times (0,1)$ in such a way that every brick is mapped to a submanifold of the form $F \times J$.
Since the geometrically finite ends of $M_n$ are peripheral, we see that the same holds for $M$ by Lemma \ref{l_21}.
This completes the proof of (iv).
Finally, we shall show (iii), that there is no incompressible half-open annulus tending to a wild end $e$ with core curve not homotopic to an annulus component of $\partial M$ tending to $e$.
Suppose, seeking a contradiction, that $M$ has such an end $e$ to which an incompressible half-open annulus $A$ tends, and that the core curve of $A$ is not homotopic into an annulus component of $\partial M$ tending to $e$.
Let $\{H_n\}$ be a sequence of properly embedded connected horizontal surfaces in $M$ meeting $A$ transversely and tending to $e$.
(Since every point lies on some horizontal leaf, such a sequence of horizontal surfaces exist.)
% which are not properly homotopic to each other in $M^\natural$, 
See Fig.\ \ref{fig4_1}\,(a).
%%%%%%%%%%%%%%%%%%%%%%%%%%%%%%%%%%%%%%%%%%%%%%%%%%%%%%%
\begin{figure}[hbtp]
\centering
\scalebox{0.5}{\includegraphics[clip]{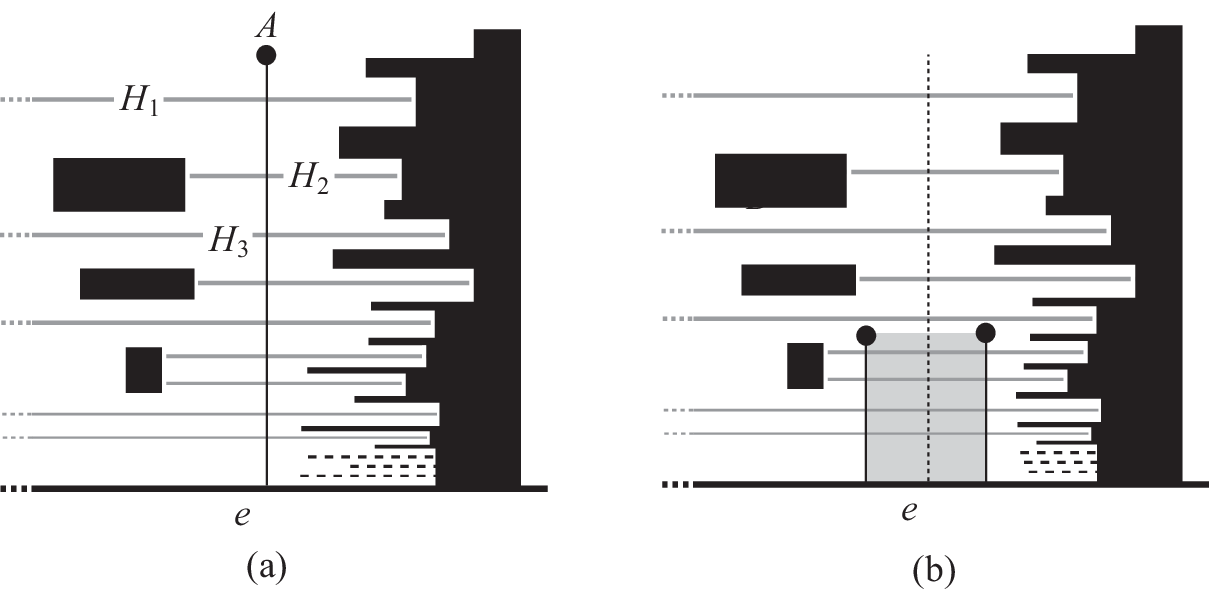}}
\caption{}
\label{fig4_1}
\end{figure}
%%%%%%%%%%%%%%%%%%%%%%%%%%%%%%%%%%%%%%%%%%%%%%%%%%%%%%% 
For each $n$, the intersection $A \cap H_n$ is an essential simple closed curve, which we denote by $l_n$.
By our assumption, $l_n$ is not homotopic into an annulus component of $\partial M$ tending to $e$.
Therefore, $g(l_n)$ either represents a loxodromic element or is homotopic into a cusp which is disjoint from a small neighbourhood of $e$.

We first assume that $g(l_n)$ represents a loxodromic element.
Let $h_n: H_n \rightarrow N_0$ be a pleated surface properly homotopic to $g|H_n$ realising $l_n$ as a closed geodesic, which we denote by $l^*$.
We should note that $H_n$ is homeomorphic to an essential subsurface of $S$.
For any $\delta >0$, the pleated surfaces $h_n$ have an upper bound depending only on $\chi(S)$ and $\delta$ for the diameters modulo their $\delta$-thin parts.
Since there are only finitely many $\varepsilon_1$-cusp neighbourhood within a bounded distance modulo the $\delta$-thin part of $N$ from $l^*$ and the images of $h_n$ contain $l^*$, by taking a subsequence we can assume that the homotopy class of $\partial H_n$ does not depend on $n$.
By the condition (ii) which we have already proved, this implies that the boundary components of $M$ on which $H_n$ lies does not depend on $n$.
It follows that there is an essential subsurface $R$ of $S$ such that all the $H_n$ are vertically parallel to $R \times \{1/2\}$ in $S \times (0,1)$.
(Notice that they may not be parallel in $M$.
To be more precise, we are claiming that the $\iota_M(H_n)$ are vertically parallel to $R \times \{1/2\}$ for the embedding $\iota_M$ of $M$ into $S \times (0,1)$ obtained above.
We omit to write $\iota_M$ here.)

Let $i_n: R \rightarrow H_n$ be a homeomorphism compatible with a homotopy from $R \times \{1/2\}$ within $S \times \{1/2\}$ to $H_n$ in $S \times (0,1)$.
Since the $l_n$ are homotopic to each other in $S \times (0,1)$, we can arrange the $i_n$ so that there is a simple closed curve $l$ on $R$ such that $i_n(l)=l_n$ for all $n$.
Recall that there are only finitely many $\varepsilon_1$-cusp neighbourhoods which $h_n(H_n)$ can touch.
We extend $l$ to a pants decomposition $P$ of $R$ so that no curve is mapped to a curve freely homotopic into a cusp which $h_n(H_n)$ can touch.
We now consider the sequence of pleated surfaces $h_n \circ i_n$, which realise $P$ as closed geodesics.
Since there are only finitely many cusps which we must take into account, by applying  the compactness of marked pleated surfaces without accidental parabolics (5.2.18 in Canary-Epstein-Green \cite{ceg}), we see that passing to a subsequence, $h_n \circ i_n$ converges to a pleated surface from a component $R'$ of $R \setminus \alpha$ containing $l$, where $\alpha$ is a possibly empty union of disjoint non-parallel essential annuli in $R$, uniformly on every compact subset of $R'$.
%Consider the pair $(H_n, l_n)$ in $S$.
%Since there are only finitely many ways of embeddings $l_n$ into $H_n$ up to homeomorphisms of $H_n$ and finitely many homeomorphism types of $H_n$,
% by taking a subsequence, we can assume that they are all homeomorphic to a pair of essential subsurface and an essential simple closed curve  $(R,l)$ in $S$.
%  Let $i_n: R \rightarrow H_n$ be a homeomorphism taking $l$ to $l_n$.
%We now consider the sequence of pleated surfaces $h_n \circ i_n$, which realises $l$.
%By the compactness of marked pleated surfaces (see 5.2.18 in Canary-Epstein-Green \cite{CEG}), passing to a subsequence, $h_n \circ i_n$ converges to a pleated surface from $R$.
%
%By the Ascoli-Arzel\`{a} Theorem, there exists a g-subsurface $R$ of $S$ 
%and an embedding $i_n:R\rightarrow H_n$ such that $R_n=j_n(R)$ is a g-subsurface of $H_n$ with $\Int R_n\supset 
%l_n=A\cap H_n$ and $\{h_n\circ i_n\}$ subconverges uniformly to a continuous map from $R$ to $N_0$.
It follows that there exists $n_0\in  \nn$, such that all $h_n\circ i_n|R'$ $(n\geq n_0)$ are properly homotopic in $N_0$.
Pulling back this to $M$, we see that there is no component of $S \times (0,1) \setminus M$ which obstructs homotopies between the $i_n|R'$.
Hence, the subsurfaces $i_n(R')$ of $H_n$ are vertically parallel in $M$ for all large $n$.
Therefore, there exists a submanifold $R' \times [0,1)$ embedded in $M$ preserving the horizontal and vertical leaves,  which contains a neighbourhood of the end of $A$ such that $R' \times \{t\}$ tends to $e$ as $t \rightarrow 1$.
%product region???????D????????????????????????horizontal surface??e?????????????????Cproduct region???L???????D????????e??nbd?????????D
See Fig.\ \ref{fig4_1}\,(b).

We shall next show that we have the same kind product region even when $g(l_n)$ represents a parabolic class.
Let $c$ be a cusp of $N$ homotopic to $g(l_n)$.
Then, we consider a pleated surface $h_n: H_n \setminus \Int N( l_n) \rightarrow N_0$ taking $\partial N(l_n)$ to the boundary of the cusp neighbourhood  of $c$ instead of the one realising $l_n$ as a closed geodesic, where $N(l_n)$ denotes an annular neighbourhood of $l_n$.
Even in this case, we have the finiteness of pleated surfaces which can be reached from the $\delta$-cusp neighbourhood $U_c$ of $c$ within a bounded distance modulo the thin part.
Therefore, as before, we can show that the $H_n$ are parallel in $S \times (0,1)$ after taking a subsequence.

As before, we can consider a homeomorphism $i_n : R \rightarrow H_n$ compatible with the inclusion of $R$ to $S$, and a sequence of pleated surfaces $h_n \circ i_n: R\setminus \Int N(l) \rightarrow N$ realising a pants decomposition $P$ containing $l$ none of whose curves except for $l$ is mapped to a cusp which can be reached by $h_n(H_n)$.
Then as in the previous case, there is a possibly empty union $\alpha$ of non-parallel disjoint essential annuli on $R$, and for components $R_1, R_2$ of $R \setminus (N(l) \cup \alpha)$ adjacent to $N(l)$, which may coincide if $l$ is non-separating, the pleated surfaces $h_n \circ i_n|R_1 \cup R_2$ converge uniformly on every compact set of $R_1 \cup R_2$.
Let $R'$ be $R_1 \cup R_2 \cup N(l)$.
Since the $h_n \circ i_n|R'$ are homotopic to each other for large $n$, we see that the subsurfaces $i_n(R')$ on $H_n$ are vertically parallel to each other.
This shows that there is a region $R' \times [0,1)$ embedded in $M$ preserving the horizontal and vertical leaves which contains a neighbourhood of the end of $A$ such that $R' \times \{t\}$ tends to $e$ as $t \rightarrow 1$.
%Therefore, we suppose that each component of $h_n \circ i_n(H_n) \cap N_0$ tends to an end $e'$ passing to a subsequence.

In both  cases, if every sequence of properly embedded connected horizontal surfaces tending to $e$ is eventually contained in $R' \times [0,1)$, then $R' \times [0,1)$ constitutes a neighbourhood of $e$, contradicting the assumption that $e$ is wild.
Suppose that this is not the case.
Then some component $c$ of  $\Fr R'$ is not homotopic to a core curve of an annulus component of $\partial M$ tending to $e$.
Therefore, we can repeat the above argument replacing $A$ with $c\times [0,1) \subset R' \times [0,1)$ and get a larger subsurface $R''$ properly containing $R'$ and an leaf-preserving embedding $R'' \times [0,1)$ such that $R'' \times \{t\}$ tends to $e$ as $t \rightarrow 1$.
Since the topological type of $S$ is fixed, in finite steps, this process terminates, and we get a neighbourhood of $e$ in the form $R_0 \times [0,1)$ for some essential subsurface $R_0$ of $S$ (which might be $S$ itself) such that $\Fr R_0 \times [0,1)$ lies on $\partial M$.
By our definition of brick decomposition of $M$ this $R_0 \times [0,1)$ is contained in one brick and $e$ must be simply degenerate.
This contradicts the assumption that $e$ is wild. 
\end{proof}

\begin{proof}[Proof of Corollary \ref{cor_b}]
By Theorem \ref{thm_a}, there is a brick manifold $M$ having the properties listed in the theorem with a bi-Lipschitz homeomorphism $g: M \rightarrow N_0$.
By Lemma \ref{l_21}, $M$ has at most countably many ends; hence so does $N_0$.

Now, we turn to the second statement of our corollary.
Let $x$ be a point in the convex core of $N_0$ and $y$ a point in $M$ with $g(y)=x$.
We can assume that $y$ does not lie in a geometrically finite brick since $x$ is contained in  the convex core.
Let $H_x$ be a properly embedded connected horizontal surface containing $x$.
If $y$ is contained in $\cv_\infty[k_u]$ for a constant $k_u$ as in the proof of Theorem \ref{thm_a}, then $x$ is contained in the $\varepsilon_1$-Margulis tube and we are done.
Otherwise, take a shortest loop $c_x$ on $H_x \setminus \cv_\infty[k_u]$ passing through $x$.
Recall that the moduli of the horizontal surfaces outside $\cv_\infty[k_u]$ are uniformly bounded.
Therefore, the length of $c_x$ is bounded uniformly from above by a constant depending only on $\chi(S)$.
Since $g$ is a $K$-Lipschitz map, the length of $g(c_x)$ is also bounded uniformly from above.
This shows that the injectivity radius at $x$ is uniformly bounded from above by a constant depending only on $\chi(S)$.
%We consider the intersection $H_x \cap M[k]$, where $k$ is a constant which appeared in the proof of Theorem \ref{thm_a} for which we have a $K$-bi-Lipschitz map $g|M[k] : M[k] \rightarrow N[k]$.
%Let $h_n : H_n \rightarrow N_0$ be a pleated surface homotopic to $g|H_n$ taking a core curve $c$ of each component of $H_n \cap \cv$ to the closed geodesic homotopic to $g(c)$.
%
%
% As in Proof of Theorem \ref{thm_a} above, there exist horizontal surfaces $F_n$ tending to $e$  which are homeomorphic each other and cut from $M$ a submanifold containing $e$.
%Again as above, we see the image of these surfaces by $g$ can be homotoped to pleated surfaces 
%
\end{proof}

\subsection{Proof of Theorem \ref{thm_c}}

\begin{proof}[Proof of Theorem \ref{thm_c}]
Let  $M$ be a labelled brick manifold satisfying the conditions (i)-(iv) in Theorem \ref{thm_a} with end invariants given so that the condition (EL) is satisfied.
Let $\ck$ be a brick complex with $\bigvee \ck=M$.
By Subsections \ref{SS_Conditions} and \ref{SS_block}, 
we know that $M$ admits a decomposition into blocks.
We use the symbols $\cv$ and $\cv[k]$ etc. to denote the unions of tubes inducing the decomposition into blocks as before.
This implies that  the condition (BB) also holds.
Since $M$ is assumed to be embedded in $S \times (0,1)$, we often identify $M$ and its image in $S\times (0,1)$.

For a simply degenerate brick $B=F\times [s,t)$ in $\ck$, we consider a monotone increasing sequence $\{p_n\}$ of 
positive numbers tending to $t$ such that, for any $n\in \nn$, every component of $F\times \{p_n\}\setminus \Int \cv$ 
is homeomorphic to $\Sg_{0,3}$ and $B(p_n)=F\times [s,p_n]$ contains at least $n$ components of $\cv[0]$.
We construct $B(p_n)$ in the same way when $B=F \times (t,s]$, just turning everything upside down.
Let $\{\ck_n\}$ be an ascending sequence of finite brick complexes with $\bigcup_n\ck_n=\ck$.
We may choose such $\ck_n$ so that $M_n=\bigvee \ck_n$ is connected for any $n\in \nn$.
Since all geometrically finite bricks in $\ck$ are peripheral in $S\times (0,1)$, the number of them 
is at most $-2\chi(S)$.
Hence we can choose $\{\ck_n\}$ so that $\ck_1$ contains $\ck_{\mathrm{gf}}$.

Consider a brick complex $\ck_n^-$ obtained from $\ck_n$ by replacing every simply degenerate bricks
 $B$ in 
$\ck_n$ with $B(p_n)$, and set $M_n^-=\bigvee \ck_n^-$.
For a simply degenerate brick $B$ in $\ck_n$ and for all $i \geq n$, the brick $B$ is contained in $\ck_i$ since $\{\ck_i\}$ is ascending.
Since $B=\bigcup_{i\geq n}B(p_i)$ by our definition of $B(p_n)$, we have $B \subset \bigcup_i M_i^-$.
Therefore we see that $M=\bigcup_n M_n^-$.

We fix a base point $x_0$ in $M_1^- \cap M[0]$.
Let $W_n[0]$ be the component of $M_n^-\cap M[0]$ containing  $x_0$, 
and  $W_n$ the union of $W_n[0]$  and 
the components  $\cv[0]$ whose boundaries lie on $\part W_n[0]$.
By the definition of $W_n$, we have $W_n\subset M_n^-\cap M^0$.
For any $n\in \nn$, there exists $m\geq n$ such that every component of $\cv[0]$ intersecting $M_n^-$ is contained in the component of $M_m^-$ containing $x_0$ since there are only finitely many components of $\cv[0]$ intersecting $M_n^-$.
%This implies that any component of $\cv[0]$ disconnecting $M_m^\dashv$ does not meet $M_n^\dashv\cap M^\natural$.
This means in particular that $M_n^-\cap M^0$ is contained in $W_m$, and hence that $\bigcup_m W_m=\bigcup_n (M_n^-\cap 
M^0)=M^0$.
Taking a subsequence if necessary, we may assume that $W_1$ contains all of the geometrically finite bricks in $\ck$.

Let $\cv_{n}^{\mathrm{ext}}$ be the union of components of $\cv\setminus \Int W_n$ intersecting $\part W_n$.
It should be noted that  $\cv_n^{\mathrm{ext}}$ might contain a component of $\cv\setminus \cv[0]$.
By the definition of $W_n$, each component of $W_n\cap \cv_n^{\mathrm{ext}}$ is an annulus.
Since $M[0]$ is acylindrical, there is no essential annulus $A$ in $W_n$ with $\part A\subset W_n\cap \cv_n^{\mathrm{ext}}$. 
Still there might be an annulus $A$ in $S\times (0,1)$ with $\partial A \subset \cv_n^{\mathrm{ext}}$.
Figure \ref{fig4_2} illustrates such a situation.
By the acylindricity of $M[0]$, for such an annulus $A$,  either there is a tube $V_A$ in $\cv$ obstructing $A$, or $A$ goes out of $M$ (\ie $A$ cannot be homotoped into $M$).
In the latter case, $A$ must go out from a simply degenerate end $B$.
%, hence must intersect the boundary of $B(p_n)$ which $\cv$ decomposes into thrice-punctured spheres.
Since the core curves of $\cv$ converges to an ending lamination, which is filling, we see that also in this case there is a tube $V_A$ in $\cv$ obstructing $A$.
Since there are only finitely many homotopy classes of such annuli,
we can choose finitely many pairwise disjoint tubes $V_1',\dots,V_m'$ in $S\times (0,1)\setminus W_n$ which obstruct all of such annuli. 
We note that all of these tubes $V_1', \dots, V_m'$ are chosen from $\cv$.
Then setting $\cv_n'=\cv_n^{\mathrm{ext}}\cup V_1'\cup\cdots\cup V_m'$ and $Z_n=S\times (0,1)\setminus \cv_n'$ in $S\times (0,1)$, we see that $Z_n$ is an acylindrical finite brick 
manifold with a brick decomposition $\mathcal L_n$ extending $\ck_n|_{W_n}$.
(Figure \ref{fig4_2} is an example of $Z_n$.)
Note that $Z_n$ is not necessarily a subset of $Z_{n+1}$ although $W_n\subset W_{n+1}$.

Since $W_n$ contains all geometrically finite bricks of $\ck$ and they are peripheral, we have $\part_\infty M\subset \part_\infty Z_n$.
Using the conformal structure given on $\partial_\infty M$, we regard $Z_n$ as a labelled brick manifold.
We can take tight tube unions so that their restriction to $W_n$ coincide with $\cv \cap W_n$.
As was shown in \S\ref{SS_block}, these tubes induce a decomposition of $Z_n[0]$  into blocks.
By the condition (BB) in \S \ref{SS_block}, the closure of each component of 
$\part W_n\setminus \cv_n^{\mathrm{ext}}$ is homeomorphic to $\Sg_{0,3}$.
For any $B$ in $\mathcal{L}_n$ with $\partial_\pm B\cap \cv_n^{\mathrm{ext}}\neq \eset$, each component of $\part_\pm B\setminus \Int \cv_n^{\mathrm{ext}}$ is homeomorphic to 
$\Sg_{0,3}$.
Therefore, this block decomposition of $Z_n[0]$ can be taken so that is restriction to $W_n$ is equal to 
the original block decomposition on $W_n[0]$.
As in \S \ref{SS_block_metric}, we  define a model metric on $Z_n[0]$ using the blocks and the conformal structure on $\partial_\infty Z_n$,
and the model metric on $Z_n[0]$ is extended to the one on $Z_n$ so that each component of 
$Z_n\setminus Z_n[0]$ is a Margulis tube.
Since $d_{Z_n}(x_0,Z_n\setminus W_n)$ goes to $\infty$ as $n\sto \infty$ with respect to the model
metric $d_{Z_n}$ on $Z_n$, the geometric limit of $\{Z_n\}$ is equal to the geometric limit $M^0$ of 
$\{W_n\}$.
 %%%%%%%%%%%%%%%%%%%%%%%%%%%%%%%%%%%%%%%%%%%%%%%%%%%%%%%
\begin{figure}[hbtp]
\centering
\scalebox{0.5}{\includegraphics[clip]{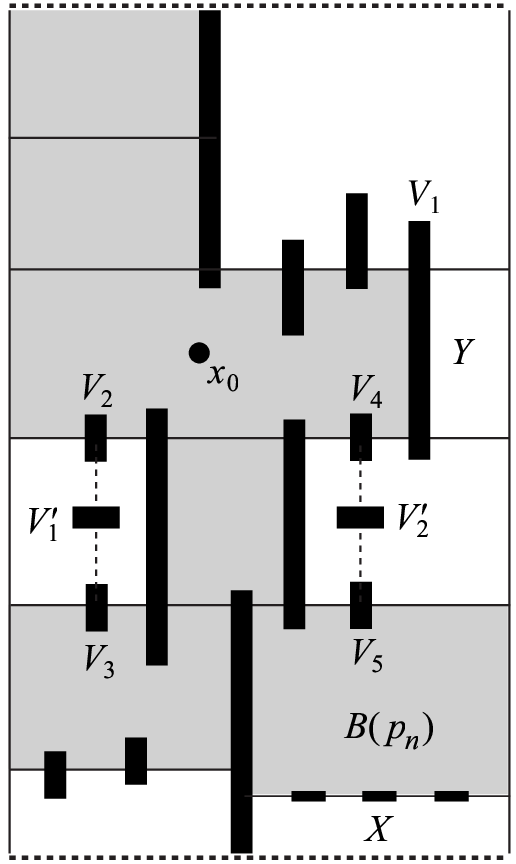}}
\caption{This figure illustrates $Z_n$.
The shaded region represents $W_n$ and the union of black rectangles is $\cv_n'$.
$B(p_n)\in \ck_n^-$ is contained in a simply degenerate brick $B$ in $\ck_n$ with $B= B(p_n)\cup X$. 
$V_1$ splits $M_n^-$ into $W_n$ and $Y= M_n^-\setminus W_n$.
$V_2$ and $V_3$ (resp.\ $V_4$ and $V_5$) are components of $\cv_n^{\mathrm{ext}}$ parallel to each other in 
$S\times (0,1)$.
$V_1'$ (resp.\ $V_2'$) obstructs an annulus between $V_2$ and $V_3$ (resp. $V_4$ and $V_5$).
}
\label{fig4_2}
\end{figure}
%%%%%%%%%%%%%%%%%%%%%%%%%%%%%%%%%%%%%%%%%%%%%%%%%%%%%%%
It is easy to check that $Z_n$ is irreducible and atoroidal.
By Thurston's uniformisation theorem for atoroidal Haken manifolds \cite{th2} 
(see Morgan \cite{Mo} and Kapovich \cite{ka} for the proof), there exists a geometrically finite hyperbolic 3-manifold $N_n$ 
with a homeomorphism $f_n: Z_n\rightarrow (N_n)_0$ which can be extended to 
the conformal map from $\part_\infty Z_n$ to $\part_\infty N_n$.
By Theorem \ref{blm} (or the original bi-Lipschitz theorem by Brock-Canary-Minsky), we may assume that $f_n$ is a $K$-bi-Lipschitz map.
Since the geometric limit of $Z_n$ based at $x_0$ is $M^0$, by the Ascoli-Arzel\`{a} theorem, $\{f_n\}$ 
converges uniformly on any compact  set of $M^0$ to a $K$-bi-Lipschitz map $f:M^0 \rightarrow N_0$, 
where $N$ is a geometric limit of $N_n$.
By our construction of block decomposition of $M[0]$,  each simply degenerate brick $F\times J$ has a sequence of tubes whose longitudes $l_n$ regarded as simple closed curves on  $F$ converge to the ending lamination $\nu(e)$ given on the end $e$ contained in $F \times J$.
By our definition of the metric on $M^0$, and the Lipschitzity of $f$, the lengths of the $l_n$ with respect to the model metric on $M^0$ are uniformly bounded.
Since $f$ is bi-Lipschitz, the closed geodesics $l_n^*$ in $N$ freely homotopic to $f(l_n)$ have also uniformly bounded 
lengths.
This shows that $l_n^*$ must  tend to the end $f(e)$ by the argument of \S\S6.3-6.4 of Bonahon \cite{bon}.
Therefore, the end $f(e)$ of $N_0$ has the ending lamination $f_*(\nu(e))$.

%Applying the Maskit Second Combination Theorem \cite[Theorem E.5]{mas}, one can obtain a new hyperbolic 3-manifold $N_n^{(1)}$ which replaces each $\zz$-cusp of $N_n$ by a $\zz\times \zz$-cusp  and 
%s arbitrarily geometrically close to $N_n$.
Let $G_n$ be a Kleinian group with $\hyperbolic^3/G_n=N_n$.
By the main theorem of \cite{OhE}, there is a sequence of geometrically finite hyperbolic $3$-manifolds $N^k_n=\hyperbolic^3/G^k_n$ without $\integers$-cusps such that $G_n^k$ converging algebraically to $G_n$.
We can choose $N^k_n$ so that the domain of discontinuity of $G^k_n$ converges to that of $G_n$ by defining $G_n^k$ to be obtained by pinching the conformal structure at infinity along curves corresponding to the $\integers$-cusps of $N_n$ and using Lemma 3 of Abikoff \cite{Ab}.
By Proposition 4.2 of J{\o}rgensen-Marden \cite{jm}, this implies that $G^k_n$ converges strongly to $G_n$ as $k \rightarrow \infty$.
By performing hyperbolic Dehn surgeries along the torus cusps of $N_n^{k}$ of type $(1,u_n)$ with sufficiently large $u_n\in\nn$, we have quasi-Fuchsian manifolds $N_n^{'k}$ geometrically approximating $N_n^{k}$ closer and closer  as $k \rightarrow \infty$ as was shown in Bonahon-Otal \cite{BO} and Ohshika \cite{oh1}.
This gives rise to a sequence of  quasi-Fuchsian manifolds $N_n^{'k}$ converging geometrically to $N_n$ as $k \rightarrow \infty$.
By the diagonal argument, we have a sequence of quasi-Fuchsian manifolds $N_n'$ converging geometrically to $N$.
This completes the proof of Theorem \ref{thm_c}.
\end{proof}

\subsection{Proof of Theorem \ref{thm_d}}

\begin{proof}[Proof of Theorem \ref{thm_d}]
Let $G_1$ and $G_2$ be non-elementary geometric limits of Kleinian surface groups isomorphic to $\pi_1(S)$ preserving the parabolicity, and $f:N_1=\hyperbolic^3/G_1\rightarrow N_2=\hyperbolic^3/G_2$ a homeomorphism preserving their end invariants.
We may assume that $f((N_1)_0)=(N_2)_0$.
By Theorem \ref{thm_a}, there exists a brick manifold $M$ and a homeomorphism $\eta_1:M \rightarrow (N_1)_0$ preserving  the end invariants.
Then the composition $\eta_2=f\circ \eta_1:M\rightarrow (N_2)_0$ is also a homeomorphism preserving the end invariants.
By Theorem \ref{blm},  we can properly homotope $\eta_1$ and $\eta_2$ to $K$-bi-Lipschitz homeomorphisms, which we denote by the same symbols $\eta_1, \eta_2$.
Therefore  $\eta_2\circ \eta_1^{-1}:(N_1)_0 \rightarrow (N_2)_0$ is a bi-Lipschitz homeomorphism preserving the  end invariants, 
which can be extended to a bi-Lipschitz map $\Phi:N_1\rightarrow N_2$.
This $\Phi$ can be lifted to a bi-Lipschitz homeomorphism $\widetilde \Phi:\hh^3\rightarrow \hh^3$  between the universal coverings, 
which is equivariant with respect to the covering translations.
Furthermore $\widetilde \Phi$ is extended to a quasi-conformal homeomorphism $\wt \Phi_{\part}$ 
on the Riemann sphere $\wh{\mathbf{C}}$ such that $\wt\Phi_{\part}|_{\Omega_{G_1}}$ is a conformal homeomorphism 
from $\Omega_{G1}$ to $\Omega_{G_2}$.
On the other hand, by Corollary \ref{cor_b}, the injectivity radii in the convex cores of our manifolds $N_1$ and $N_2$ are bounded above.
This makes it possible to apply McMullen's generalisation (Theorem 2.9 in \cite{Mc}) of Sullivan's rigidity theorem, and we see that $\eta_2 \circ \eta_1^{-1}$, which is properly homotopic to $f$, is properly homotopic to an isometry.
%If $\mu(\L_{G_1})=0$, then the quasi-conformal map $\wt \Phi_{\part}$ is conformal almost everywhere, 
%and hence is a conformal map.
%See Corollary 2 in \cite[Chapter 2]{ah}.
%Otherwise, by Lemma \ref{l_div}, $\L_{G_1}=\wh{\mathbf C}$ and $G_1$ is of divergence type.
%Then, by Sullivan's Rigidity Theorem \cite{su}, $\wt\Phi_{\part}$ is a conformal map on $\wh{\mathbf C}$.
%Thus, in either case, $\wt\Phi_{\part}$ is extended to an equivariant isometry $\widetilde \Psi:\hh^3\rightarrow \hh^3$, 
%which covers an isometry $\psi:N_1\rightarrow N_2$ properly homotopic to $f$.
\end{proof}

\end{document}